\documentclass[a4paper, 11pt]{amsart}

\usepackage[T1]{fontenc}
\usepackage[english]{babel}
\usepackage[utf8]{inputenx}
\usepackage{amsmath,amsfonts,amssymb,amsthm,mathrsfs,pinlabel,graphicx} 
\usepackage{tikz}
\usetikzlibrary{arrows}
\usepackage[all,graph]{xy}
\usepackage{todonotes}
\usepackage{fancyhdr}
\usepackage{url}

\numberwithin{equation}{section}
\numberwithin{figure}{section}

\def\ol#1{\overline{#1}}

\newtheorem{theorem}{Theorem}[section]
\newtheorem*{theorem*}{Theorem}
\newtheorem{lemma}[theorem]{Lemma}
\newtheorem{proposition}[theorem]{Proposition}
\newtheorem{corollary}[theorem]{Corollary}

\theoremstyle{remark} 
\newtheorem{remark}[theorem]{Remark}
\newtheorem{example}[theorem]{Example}

\newtheorem*{ack}{Acknowledgements}

\theoremstyle{definition} 
\newtheorem{definition}[theorem]{Definition}
\newtheorem*{motto}{Motto}
\newtheorem{conjecture}[theorem]{Conjecture}
\newtheorem{problem}{Problem}

\def\spinc{Spin${}^c$} 
\def\sss{\mathfrak{s}}
\def\sst{\mathfrak{t}}

\DeclareMathOperator{\rk}{\operatorname{rank}}

\DeclareMathOperator{\Sym}{\operatorname{Sym}}
\DeclareMathOperator{\Spin}{\operatorname{Spin}}
\DeclareMathOperator{\SO}{\operatorname{SO}}

\DeclareMathOperator{\U}{\operatorname{U}}

\newcommand\cL{{\mathcal L}}
\newcommand\cM{{\mathcal M}}

\newcommand\cO{{\mathcal O}}

\newcommand\CC{{\mathbb C}}

\newcommand\QQ{{\mathbb Q}}
\newcommand\RR{{\mathbb R}}

\newcommand\TT{{\mathbb T}}
\newcommand\ZZ{{\mathbb Z}}

\newcommand\fs{{\mathfrak s}}
\newcommand\ft{{\mathfrak t}}

\newcommand\gL{{\Lambda}}

\newcommand\gS{{\Sigma}}

\newcommand\ga{{\alpha}}
\newcommand\gb{{\beta}}
\newcommand\ggm{{\gamma}}

\newcommand\gs{{\sigma}}

\newcommand\bfx{{\mathbf{x}}}
\newcommand\bfy{{\mathbf{y}}}
\newcommand\bfz{{\mathbf{z}}}

\newcommand\bfga{{\boldsymbol\alpha}}
\newcommand\bfgb{{\boldsymbol\beta}}
\newcommand\bfgc{{\boldsymbol\gamma}}

\mathchardef\ordinarycolon\mathcode`\:
\mathcode`\:=\string"8000
\begingroup \catcode`\:=\active
\gdef:{\mathrel{\mathop\ordinarycolon}}
\endgroup

\title[Heegaard Floer Homologies]{Heegaard Floer Homologies and Rational Cuspidal Curves. Lecture notes.} 

\author{Adam Baranowski}
\address{Institute of Mathematics, University of Warsaw, ul. Banacha 2,
02-097 Warsaw, Poland.}
\email{adam.baranowski@protonmail.com}
\author{Maciej Borodzik}
\address{Institute of Mathematics, Polish Academy of Science, ul. \'Sniadeckich 8, Warsaw, Poland.}
\address{Institute of Mathematics, University of Warsaw, ul. Banacha 2,
02-097 Warsaw, Poland.}
\email{mcboro@mimuw.edu.pl}
\author{Juan Serrano de Rodrigo}
\address{Dpto. de Matemáticas, Universidad de Zaragoza, C/ Pedro Cerbuna 12, 50009 Zaragoza, España.}
\email{juansr@unizar.es}
\begin{document}
	

\begin{abstract} 
This is an expanded version of the lecture course the second author gave at Winterbraids VI in Lille in February 2016.
\end{abstract}
	
\maketitle


%

\section{Introduction} \label{Introduction}

Heegaard Floer homologies were defined around 2000 by Ozsv\'ath and Szab\'o. Since then a lot of research has been done in the subject and the number of papers that have appeared in the last 15 years is immense. It appears now that the whole knot theory and topology of  three--manifolds has been affected at least in some way by this new theory. 

Even though it is generally believed and almost completely proved  (see \cite{KuLeTa11})   that Heegaard Floer theory contains the same amount of information as the Seiberg--Witten theory, the Heegaard Floer theory has an advantage over the latter, namely often problems in Heegaard Floer theory
can be reduced to combinatorics of Heegaard diagrams, which makes Heegaard Floer theory  more accessible to an inexperienced reader. Moreover, this
combinatorial flavor of Heegaard Floer theory sometimes makes it possible to effectively calculate Heegaard Floer homology groups, for example
from a surgery formula \cite{OzSz-integer,OzSz11,OM}.
 
As for the knot Floer theory: given any knot, there is not only an algorithm calculating knot homology groups \cite{OSS-book}, 
but also one often understands general properties of Floer chain complexes for knots, like torus knots and alternating knots, including two--bridge knots.

The immense speed of the development of Heegaard Floer theory makes it quite difficult for a non--expert to get an overview of the field. In the ever-growing pile of articles on the subject it might be hard not to get lost and to find the most important articles.  Luckily, a few excellent
survey papers have appeared: those by Ozsv\'ath and Szab\'o \cite{OzSz-intro1,OzSz-intro2}, and more modern ones of Juh\'asz \cite{Juh} 
and Manolescu \cite{Man}. A recent book \cite{OSS-book} covers the grid diagram approaches to Heegaard Floer theory.

The aim of these notes is to give another introduction into the subject but this time with a clear view towards algebraic geometry. We focus on parts of the theory which are relevant in the applications, like L--space knots and  $d$--invariants. We omit parts which, at least at present, 
have little application in algebraic geometry.

\subsection{What is not in the notes?}

Actually only a small part of the theory is covered in the notes. We do not mention any analytic difficulties with defining the Heegaard Floer
theory rigorously, like compactness and smoothness of the moduli space of holomorphic disks used in \cite{os-threemanifold}. We focus mostly on rational homology three--spheres, not mentioning technical issues with defining Heegaard Floer homologies on manifolds with $b_1(Y)>0$. In particular, we do not discuss the action of $\Lambda^* H_1(Y;\ZZ)$ on the Heegaard Floer chain complex. Refer to \cite{levine-ruberman} for more details. 

Knot Floer homology is defined via Heegaard diagrams and only for knots. In the notes we do not give any construction via grid diagrams,  even though
it is purely combinatorial and has much less prerequisite knowledge;  nonetheless it seems somehow  that the original approach of Rasmussen and Ozsv\'ath--Szab\'o reveals better why knot Floer homology is such a powerful tool. For a detailed account on grid Floer homology we refer to an excellent book of Ozsv\'ath, Stipsicz and Szab\'o \cite{OSS-book} mentioned above. For other ways to define the knot Floer homology we refer to the survey of Manolescu \cite{Man} and references therein.

We do not discuss the construction and properties of Heegaard Floer theory for links. The definition might seem very similar for links as it is for knots, yet the applications are much harder. In particular, the surgery formula for links is very hard, see \cite{OM} for details and \cite{Liu} for exemplary applications.

We do not introduce the $\tau$--invariant, which is a smooth concordance invariant that detects the four-genus of many knots, including torus knots, 
see \cite{OS-fourball}: for algebraic links it is equal to the three--genus anyway, so it does not bring any new piece of information about algebraic knots. Likewise, we do not discuss the $\Upsilon$ function of Ozsv\'ath, Stipsicz and Szab\'o \cite{OSS}, which is a significant refinement of the $\tau$ invariant. 
For algebraic knots the $\Upsilon$ function is related to the $V_m$ invariants; see \cite{BH}.

Concordance invariants are only mentioned in the paper, we refer to a recent survey of Hom \cite{Hom} for more details.
The whole research concerning alternating links and Heegaard Floer--thin links is not mentioned in the article; see \cite{OS-alternating,MOq}. We do not discuss double branched covers of links and their $d$--invariants, like in \cite{MOw}. We do not provide any relations of Heegaard Floer theory with Khovanov homology; like in
\cite{OS-alternating}.

Sutured Heegaard Floer theory \cite{Juh-sut} as well as its younger cousin, the bordered Floer theory, see \cite{LOT0,LOT1,LOT2}, is not covered in these
notes. Bordered Heegaard Floer theory is a generalization of the Heegaard Floer theory for three-manifolds with boundary, with the aim to calculate
Heegaard Floer homology groups by a cut-and-paste method. The algebraic setup for the bordered Floer theory is rather complicated, but the theory itself contains
a lot of information, for example the S-equivalence class of a Seifert matrix of a knot can be read off from the bordered Floer homology of the knot
complement, see \cite{HLW}. It is known that knot Floer homology does not determine the Seifert matrix, see the discussion in \cite[Section 1]{HLW}.

On the singularity theory side, we do not give full details on the classification of algebraic knots (and links). A concise but self--contained description is given in the book of Eisenbud--Neumann \cite{EN}, which is also very well suited for topologists. We discuss only quickly and superficially the theory of rational cuspidal curves, referring to the thesis of Moe \cite{Moe08} or to a book of Namba \cite{Namba} for a more classical version. 
The techniques such as spectrum semicontinuity or applications of 
the Bogomolov--Miyaoka--Yau inequality are not given. A reader wishing to learn methods of spectrum semicontinuity is referred to \cite{FLMN04}, a nice application
of the Bogomolov--Miyaoka--Yau inequality in the theory of rational cuspidal curves is given also in \cite{Orev02}. 

\subsection{What is in the notes?}

Compared to what is not in the notes, the content of the paper is very scarce. 
With a view towards applications in algebraic geometry 
we try to give just about enough details for the reader
to understand the two results about semigroup distribution property of rational cuspidal curves:
Theorem~\ref{thm:BL-main1} and Theorem~\ref{thm:BL-main2}, as well as their proofs. Consequently, we introduce Heegard Floer homology
in Section~\ref{KFH}, where we also give a very brief description of \spinc{} structures on three-- and four--manifolds. In Section~\ref{sec:why}
we state two main results on Heegaard Floer theory: the adjunction inequality and the surgery exact sequence. These results are used in proofs
of most of the main theorems on Heegaard Floer theory. Section~\ref{sec:cobordism} deals with cobordisms in Heegaard Floer theory,
in particular, we define $d$--invariants, show their behavior on the cobordism and define the absolute grading in the Heegaard Floer homology. At present,
Theorem~\ref{thm:dinvariants} is the most important result of Ozsv\'ath--Szab\'o from the point of view of applications in algebraic geometry.

Next we discuss knot Floer homology in Section~\ref{sec:knots}. Our emphasis is on the $V_m$--invariants for knots introduced in
Section~\ref{sec:Vm} and then on L--space knots, which we discuss in detail in Section~\ref{sec:Lspace}.  

Section~\ref{sec:cuspidal} contains a (short and by no means complete) account on cuspidal singularities. We give basic definitions and
pass quickly to the construction and basic properties of semigroups of singular points. We provide relations between semigroups and
Alexander polynomials. We finish by linking the semigroups of singular points with the $V_m$--invariants of the links of singularities.

In Section~\ref{sec:rational} we first go quickly through recent results on rational cuspidal curves and give Theorem~\ref{thm:BL-main1}
and \ref{thm:BL-main2}, which are central results of these notes. We then discuss a relation of these results with the FLMN conjecture (Conjecture~\ref{conj:flmn}), whose motivation we also recall. Finally, we show highlights and weak points of Theorem~\ref{thm:BL-main2}, as well as a counterexample to the original Conjecture~\ref{conj:flmn} found by Bodn\'ar and N\'emethi.

We have decided to give the reader a lot of problems to solve. Most of these are quick observations, some of them might require extra work.
There is one problem, namely Problem~\ref{prob:onlyopen}, which is a research problem.

\begin{ack}
The authors would like to thank the organizers of Winterbraids VI for their effort in organizing the conference and for creating a place
for disseminating new ideas and building new perspectives in low-dimensional topology. The authors would also like to thank 
Marco Golla, Jen Hom, Charles Livingston and Andr\'as Stipsicz for many valuable comments on a preliminary version of the notes. We are
particularly indebted to the referee for his/her remarks that led to a significant improvement of the article. 
\end{ack}
\section{Heegaard Floer homology} \label{KFH}

\subsection{Preliminaries. \spinc{} structures on three-- and four--manifolds.} \label{section:spinc_structures}
This section gathers some facts about \spinc{} structures, which will be used in later sections. We will consider only \spinc{} structures
on the tangent bundle of a manifold. We refer to \cite[Chapter 2]{Fri} for a more detailed discussion. Other, concise
references are \cite[Section 1.4]{GS} or \cite[Section 1.3]{Nic}. A reader might want to skip this section at first reading.

Recall that for $n \ge 3$ the fundamental group of the special orthogonal group $\SO(n) := \SO(n,\RR)$ is $\pi_1(\SO(n))=\ZZ_2$. 
We define the \emph{spin group} $\Spin(n)$ to be the non-trivial double cover of $\SO(n)$, thus in the case $n \ge 3$, it is the universal cover of $\SO(n)$. By the construction there is a canonical inclusion $\ZZ_2 \hookrightarrow \Spin(n)$.
The group $\Spin^c(n)$ is defined to be
\begin{equation}\label{eq:spincdef}
\Spin^c(n) := \big( \Spin(n) \times \U(1) \big) / \ZZ_2.
\end{equation}
It fits into the following short exact sequence
$$
1 \rightarrow \U(1) \xrightarrow{i} \Spin^c(n) \xrightarrow{p} \SO(n) \rightarrow 1,
$$
where $i$ sends $z$ to $[1, z]$ and $p$ is the projection of $\Spin^c(n)$ onto $\SO(n)$ via $\Spin(n)$.
\begin{problem}
	Verify that the projection $p$ is well defined and gives rise to the short exact sequence above.
\end{problem}

Consider now an oriented, $n$--dimensional Riemannian manifold $M$. We can regard the tangent bundle $TM$ as associated to the $SO(n)$--principal bundle $P_{SO(n)}$ of oriented orthonormal frames.

\begin{definition}[$\Spin^c$ structure]
	A \emph{$\Spin^c$ structure} on $M$ is a pair $(P, \gL)$ consisting of a $\Spin^c(n)$--principal bundle $P$ over $M$ and a map $\gL : P \to P_{SO(n)}$ such that the diagram
	\[
\xymatrix{\Spin^c(n)\times P \ar[r] \ar[dd]^{p \times \gL} & P \ar[dd]^{\gL} \ar[dr] & \\
		& & M \\
		SO(n) \times P_{SO(n)}\ar[r] & P_{SO(n)} \ar[ur] &
	}
	\]
with horizontal maps being the group actions on principal bundles, commutes.
We denote the set of all $\Spin^c$ structures on $M$ as $\Spin^c(M)$.
\end{definition}




There is a group homomorphism $\pi : \Spin^c(n)\to S^1$, a projection on the second factor in \eqref{eq:spincdef} given by $[g, z] \mapsto z^2$. The composition
$\pi \circ i$ is then a double cover of $S^1$. Thus, given a \spinc{} structure $(P, \gL)$ on $M$, the map $\pi$ can be used to construct an $S^1$--principal bundle $P_1 = P / \Spin(n)$ over $M$. From this we can define the so-called \emph{determinant line bundle} $L\to M$, which is given by $L = P_1 \times_{S^1} \CC$. One can in fact think of a \spinc{} structure on $M$ as of a choice of a complex line bundle $L$ and a Spin structure on $TM\otimes L^{-1}$. We refer to \cite[Section 1.3]{Nic} and \cite[Section 2.4]{Fri} for more details.

\begin{definition}
The \emph{first Chern class} of a \spinc{} structure $\sss$ on a manifold $M$ is $c_1(\sss)=c_1(L)$.
\end{definition}
As $TM\otimes L^{-1}$ is a Spin bundle, its second Stiefel--Whitney class vanishes. A quick calculus on
characteristic classes yields the 
following fact, see \cite[Section 1.3.3]{Nic}.
\begin{proposition}
We have that $c_1(\sss)\bmod 2\equiv w_2(M)$, where $w_2(M)$ is the second Stiefel--Whitney class of $M$.
\end{proposition}
\begin{remark}
The meaning of `mod $2$' can be made precise by considering the short exact sequence $0\to\ZZ\stackrel{\cdot 2}\to\ZZ\to\ZZ_2\to 0$. Associated
with it is a long cohomology exact sequence (this is best seen, when using \v{C}ech homology, see \cite{Nic}), in particular there is a well-defined
map $H^j(X;\ZZ)\to H^j(X;\ZZ_2)$ for any compact topological space $X$ and any $j\ge 0$. This map is often denoted $x\mapsto x\bmod 2$.
\end{remark}

\begin{proposition}[see \expandafter{\cite[Proposition 1.3.14, Exercise 1.3.12]{Nic}}]\label{prop:charel}
Let $M$ be a closed oriented manifold.
Let $\cL_M\subset H^2(M;\ZZ)$ be the set of integral lifts of the second Stiefel--Whitney class $w_2(M)$. The map $c_1\colon Spin^c(M)\to\cL_M$
is surjective. Moreover, if $H^2(M;\ZZ)$ has no 2-torsion, then this map is also injective.
\end{proposition}
\begin{problem}
Show that if $M$ is simply connected, then $H^2(M;\ZZ)$ has no 2-torsion.
\end{problem}
\begin{definition}
An element $x\in\cL_M$ is called a \emph{characteristic element}.
\end{definition}
In other words for manifolds such that $H^2(M;\ZZ)$ has no 2-torsion, \spinc{} structures correspond precisely to characteristic elements.

Another way of understanding \spinc{} structures on a manifold is to see that two different \spinc{} structures on $M$ differ by a complex line bundle,
hence the class of isomorphisms of complex line bundles (which in the smooth category is the same as $H^2(M;\ZZ)$) acts
on the set of all \spinc{} structures on $M$. This action can be shown to be transitive and free, see again \cite[Section 1.3]{Nic}, however
there is (usually) no canonical identification of $\Spin^c(M)$ with $H^2(M;\ZZ)$. Anyway, if $H^2(M;\ZZ)$ is finite, then the number
of \spinc{} structures on $M$ is equal to the cardinality of $H^2(M;\ZZ)$.

We also recall another equivalent formulation of $\Spin^c$ structures on three--manifolds due to Turaev \cite{Tur97}. Let $M$ be a closed, connected, oriented three--manifold. An \emph{Euler structure} is an equivalence class of non-vanishing vector field on $M$, where two vector fields $v$ and $w$ are said
to be equivalent if there exists a closed ball $B\subset M$ such that $v$ is homotopic to $w$ through non-vanishing vector fields on $M \setminus \operatorname{Int} B$.
\begin{proposition}[see \cite{Tur97}]
The set of Euler structures on a three--manifold is in a one-to-one correspondence with the set of \spinc{} structures.
\end{proposition}
\begin{problem}
Construct geometrically a transitive and free action of $H_1(M;\ZZ)$ on the set of all Euler structures on a closed three--manifold.
\end{problem}

We pass to a description of \spinc{} structures on four---manifolds. We begin with the following fact.
\begin{lemma}[see \expandafter{\cite[Proposition 1.4.18]{GS}}]\label{lem:char}
Let $M$ be a four--manifold with the intersection form $Q\colon H_2(M;\ZZ)\times H_2(M;\ZZ)\to\ZZ$. Then for any $x\in H_2(M;\ZZ)$ we have
$\langle w_2(M),x\rangle\equiv Q(x,x)\bmod 2$. 
\end{lemma}
\begin{corollary}\label{cor:spinc4}
If $M$ is a closed simply-connected four--manifold, then \spinc{} structures on $M$ are in a one-to-one correspondence with the elements $K\in H^2(M;\ZZ)$
such that $Q(x,x)\equiv\langle K,x\rangle\bmod 2$ for all $x\in H_2(M;\ZZ)$.
\end{corollary}

\subsection{Heegaard diagrams} \label{Heegard_diagrams}

A \emph{genus $g$ handlebody} $U$ is the boundary connected sum of $g$ copies of a solid torus $D^2\times S^1$. In other words, it is a three--manifold
 diffeomorphic to a regular neighborhood of a bouquet of $g$ circles in $\RR^3$. 
The boundary of $U$ is an oriented surface $\gS$ of genus $g$.

\begin{definition}[Heegaard decomposition]
	Let $M$ be a closed, oriented, connected three--manifold. A \emph{Heegaard decomposition} is a presentation of $M$ as a union $U_0 \cup_\gS U_1$, where $U_0$ and $U_1$ are handlebodies and $\gS$ is a closed, connected surface.
\end{definition}

\begin{problem}
	Show that the only manifold admitting a Heegaard decomposition of genus $0$ is $S^3$.
\end{problem}

\begin{example}
If $U_0$ and $U_1$ are two solid tori glued along their boundary, then $M$ is either $S^3$, $S^2 \times S^1$ or a lens space.

To see this, denote by $m_i$ and $l_i$ the meridian and the longitude of the solid torus $U_i$, $i=1,2$. In order to glue the two tori we need to determine which curve on the torus $\partial U_0$ will be the meridian of $\partial U_1$, that is, $m_1 = p m_0 + q l_0$ for some $p,q \in \ZZ$. Since $m_1$ is a closed curve, $\gcd(p,q)=1$. We consider two cases: if $q=0$, then $p=1$, and we identify $m_0$ with $m_1$ and $l_0$ with $l_1$. The resulting three--manifold is $S^2 \times S^1$. For the case $q \neq 0$ we will show that the construction above is equivalent to the usual construction of a lens space defined as the quotient of $S^3$ by an action of $\ZZ_q$. In order to do that, consider $S^3$ as a subset of $\CC^2$ obtained by gluing two solid tori $U_0 = \{(z_1, z_2) \in \CC^2 : |z_1|^2 + |z_2|^2 = 1, |z_1|^2 \ge \frac{1}{2} \ge |z_2|^2 \}, U_1 = \{(z_1, z_2) \in \CC^2 : |z_1|^2 + |z_2|^2 = 1, |z_2|^2 \ge \frac{1}{2} \ge |z_1|^2 \}$ along the torus $\Sigma = \{(z_1, z_2) \in \CC^2 : |z_1|^2 = |z_2|^2 = \frac{1}{2}\}$.
Observe that each of these sets is preserved by an action of $\ZZ_q$ given by $[1] \cdot (z_1, z_2) = (e^{2\pi i/q} \cdot z_1, e^{2 \pi i p/q}\cdot z_2)$, and the orbits $U_0 / \ZZ_q, U_1 / \ZZ_q$ are again solid tori. Finally, the quotient $\Sigma / \ZZ_q$ is a torus.
Upon closer examination of the way these two quotient tori are glued under this action, one may notice that the meridian $m_1$ of $U_1 / \ZZ_q$ is mapped exactly to the curve $p \, m_0 + q \, l_0$ on $U_0 / \ZZ_q$; see e.g. \cite{Rolfsen} for the details.
\end{example}

\begin{theorem}\label{thm:admitsheegaard}
	Each three--manifold $M$ admits a Heegaard decomposition. 
\end{theorem}

\begin{proof}[Sketch of proof]
Let $F\colon M\to[0,3]$ be a self--indexing Morse function, that is, a Morse function such that the critical levels of index $k$ are all at the level set $F^{-1}(k)$. (Such a function exists by \cite{Milnor-cob}.) Using an argument of \cite{Milnor-cob}, we might and actually will assume that $F$ has only one minimum and only one maximum. Define $U_0=F^{-1}[0,3/2]$, $U_1=F^{-1}[3/2,3]$ and $\gS=F^{-1}(3/2)$. As $F$ has only one minimum and one maximum, all of the three 
spaces $U_0,U_1$ and $\gS$ are connected. In particular, $\gS$ is a closed connected surface. The genus $g(\gS)$ is equal to the number of critical points $F$ of index $1$. By construction, $U_0$ and $U_1$ are genus $g$ handlebodies. This shows the existence of a Heegaard decomposition.
\end{proof}

A Heegaard decomposition is definitely not unique. One of the
methods of obtaining a new Heegaard decomposition from another one is the following.

Given a Heegaard decomposition $M = U_0 \cup_\gS U_1$ of genus $g$, choose two points in $\gS$ and connect them by an unknotted arc $\ggm$ in $U_1$. Let $U'_0$ be the union of $U_0$ and a small tubular neighborhood $N$ of $\ggm$. Similarly, let $U'_1 = U_1 \setminus N$. The new decomposition $M = U'_0 \cup_{\gS'} U'_1$ is called the \emph{stabilization} of $M = U_0 \cup_\gS U_1$. Clearly $g(\gS') = g(\gS) + 1$. Stabilizations and destabilizations will be discussed in a greater detail below (see Theorem~\ref{thm:handleslide}).

In fact any two Heegaard decompositions are related by stabilizations and destabilizations (a precise statement is given
in Theorem~\ref{thm:handleslide} below). This can be seen using Cerf theory \cite{Cerf}.
Any two Morse functions $F_0$ and $F_1$ on $M$ can be connected by a path $F_t$, $t\in[0,1]$ in the space of all smooth functions from $M$ to $\RR$
in such a way that for all but finitely many values $t\in[0,1]$, $F_t$ is a Morse function and there is a finite number of special values $t_1,\ldots,t_n$ at which a cancellation or a creation of a pair of critical points occurs. A more detailed analysis reveals 
that stabilizations and destabilizations of Heegaard diagrams correspond to creations, respectively, cancellations, of pairs of critical points of index $1$ and $2$. We omit the details, referring to \cite{Cerf}. An interested reader might find helpful a detailed exposition of
the subject in \cite{JT}.

\begin{problem}
	Construct explicitly a Heegaard decomposition of $S^3$ of an arbitrary genus $g$.
\end{problem}

Theorem~\ref{thm:admitsheegaard} allows us to think of a three--manifold $Y$ as a pair of two handlebodies $U_0$ and $U_1$ glued along their boundaries via a homeomorphism $\phi\colon\partial U_0\to\partial U_1$. As isotopic homeomorphisms $\phi$ give rise to homeomorphic manifolds, 
in general, $\phi$ is an element of the mapping class group of $\partial U_0$, and
elements in mapping class groups are rather hard to deal with. Luckily, there is a more geometric point of view of a Heegaard decomposition. Suppose that $F$ is a Morse function on $Y$ such that $F^{-1}[0,3/2]=U_0$ and $F^{-1}[3/2,3]=U_1$. Let $\gS=F^{-1}(3/2)=\partial U_0=\partial U_1$. Choose a Riemannian metric on $M$ and consider the gradient $\nabla F$. Critical points of $F$ correspond to stationary points of the vector field $\nabla F$ and the Morse condition means that the stationary points are hyperbolic, hence the stable and unstable manifolds are well defined. (We refer to \cite{GH} for more details on stable and unstable manifolds.) Moreover, the Morse index of $F$ gives precise information about the dimensions
of the stable and unstable manifolds  given by the  stationary points of $\nabla F$. Each index $1$ critical point  of $F$ has  a $2$-dimensional unstable manifold of $\nabla F$. Likewise, each index $2$ critical point of $F$ has a $2$-dimensional stable manifold of $\nabla F$. The unstable manifold of a critical point of index $1$ intersects $\gS$ along a simple closed curve and the stable manifold of a critical point of index $2$ intersects $\gS$ along a simple closed curve.

If the genus of the Heegaard decomposition is $g$, the above procedure yields precisely $g$ simple closed curves on $\gS$ obtained as intersections of unstable manifolds of critical points of index $1$ with $\gS$, and $g$ simple closed curves obtained as intersections of stable manifolds of critical points of index $2$ with $\gS$. Call the first set of curves $\alpha_1,\ldots,\alpha_g$ and the second set $\beta_1,\ldots,\beta_g$. We will often call these curves $\alpha$--curves and $\beta$--curves. By construction both the
$\alpha$--curves and the $\beta$--curves are pairwise disjoint. If $\nabla F$ satisfies the
Morse--Smale condition, then the $\alpha$--curves intersect the 
$\beta$--curves transversally.

\begin{problem}
Prove that each of the $\alpha$ curves constructed above is homologically trivial in $U_0$ and each of the $\beta$--curves is homologically trivial in $U_1$.

Show even more, namely, that the curves $\alpha_1,\ldots,\alpha_g$ span $\ker H_1(\Sigma;\ZZ)\to H_1(U_0;\ZZ)$ and that a similar statement holds for the $\beta$--curves.
\end{problem}

The last problem leads to the following definition:

\begin{definition}[Heegaard diagram]
	Let $Y = U_0 \cup_\gS U_1$ be a Heegaard decomposition of a three--manifold $Y$, and let $g$ be the genus of $\gS$. A \emph{Heegaard diagram} is a triple $(\gS, \bfga, \bfgb)$, where $\bfga$ and $\bfgb$ are unordered collections of $g$ simple closed curves $\ga_1, \ldots, \ga_g$ and $\gb_1, \ldots, \gb_g$, such that:
	
\begin{itemize}
	\item[$\bullet$] $\ga_i \cap \ga_j = \gb_i \cap \gb_j = \emptyset$ if $i \neq j$.
	
	\item[$\bullet$] The curves $\{ \ga_1, \ldots, \ga_g \}$ form a basis of $\ker \big( H_1(\gS;\ZZ) \rightarrow H_1(U_0;\ZZ) \big)$ and $\{ \gb_1, \ldots, \gb_g \}$ form a basis of $\ker \big( H_1(\gS;\ZZ) \rightarrow H_1(U_1;\ZZ) \big)$.
\end{itemize}

\end{definition}

\begin{problem}
	Consider $\gS\times[0,1]$. Thicken all the $\alpha$--curves on $\gS\times\{1\}$ to obtain pairwise disjoint annuli $A_1,\ldots,A_g\subset \gS\times\{1\}$. Set $A=A_1\cup\ldots\cup A_g$. Define \[H=\gS\times[0,1]\cup_{A} \bigcup_{j=1}^g D_j,\] where $D_j=D\times I$ is a disk $D$ cross the interval $I$, glued to $\gS\times\{1\}$ along $A_j$. Prove that $\partial H$ is a disjoint union of $\gS\times\{0\}$ and a two--sphere.

Use this problem to explicitly reconstruct a three--manifold from $\gS$ and the collection of $\alpha$-- and $\beta$--curves. 
\end{problem}
The approach to three--manifolds via Heegaard diagrams allows us to obtain a combinatorial approach to studying three--manifolds. Heegaard Floer theory can be regarded as a way
of extracting information about the three--manifold from the combinatorics of a Heegaard diagram.

Before we go further, we need to understand how a Heegaard diagram depends on the choice of the Morse function $F$.
\begin{remark}
If the Heegaard diagram is built from the vector field $\nabla F$, that is, the $\alpha$--curves and the $\beta$--curves
are the intersections of the unstable and stable manifolds with $\gS$, then the Heegaard diagram depends also on the Riemannian metric
used to define the vector field $\nabla F$.
\end{remark} 
\begin{definition}
	Two Heegaard diagrams $(\gS, \bfga, \bfgb)$, $(\gS', \bfga', \bfgb')$ are \emph{diffeomorphic} if there is an orientation--preserving diffeomorphism of $\gS$ to $\gS'$ that carries $\bfga$ to $\bfga'$ and $\bfgb$ to $\bfgb'$.
\end{definition}  

\begin{definition}[Handlesliding]
	Let $U$ be a handlebody and denote by $\bfgc = \{ \ggm_1, \ldots, \ggm_g \}$ a set of attaching circles for $U$. Let $\ggm_i, \ggm_j \in \bfgc$ with $i \neq j$. We say that $\ggm'_i$ is obtained from \emph{handlesliding $\ggm_i$ over $\ggm_j$} if $\ggm'_i$ is any simple closed curve which is disjoint from the $\ggm_1, \ldots, \ggm_g$, and the curves $\ggm'_i, \ggm_i, \ggm_j$ bound a pair of pants in $\gS$ (see Figure \ref{fig:handlesliding}). In that case, the set $\bfgc' = \{ \ggm_1, \ldots, \ggm_{i-1}, \ggm_{i}', \ggm_{i+1},\ldots, \ggm_g \}$ (with $\ggm_{i}$ replaced by $\ggm_{i}'$) 
is also a set of attaching circles for $U$.
\end{definition}

\begin{figure}[h]
  \centering
  \begin{tikzpicture}[y=0.80pt, x=0.80pt, yscale=-1.000000, xscale=1.000000, inner sep=0pt, outer sep=0pt]
  \path[draw=black,line join=miter,line cap=butt,even odd rule,line width=0.451pt]
    (33.5401,44.5472) .. controls (33.5401,44.5472) and (71.0762,16.1629) ..
    (101.2000,16.3556) .. controls (137.3810,16.5870) and (139.2959,17.6708) ..
    (168.8600,44.5472) .. controls (183.3303,57.7020) and (218.7609,59.0175) ..
    (236.5199,44.5472) .. controls (266.6223,18.2379) and (268.9243,16.2875) ..
    (304.1799,16.3556) .. controls (340.3583,16.1037) and (371.8398,44.5472) ..
    (371.8398,44.5472);
  \path[draw=black,line join=miter,line cap=butt,even odd rule,line width=0.451pt]
    (33.5401,100.6265) .. controls (33.5401,100.6265) and (71.0762,129.0108) ..
    (101.2000,128.8182) .. controls (137.3810,128.5868) and (139.2959,127.5030) ..
    (168.8600,100.6265) .. controls (183.3303,87.4717) and (218.7609,86.1562) ..
    (236.5199,100.6265) .. controls (266.6223,126.9358) and (268.9243,128.8862) ..
    (304.1799,128.8182) .. controls (340.3583,129.0701) and (371.8398,100.6265) ..
    (371.8398,100.6265);
  \path[draw=black,line join=miter,line cap=butt,even odd rule,line width=0.451pt]
    (67.5109,63.7093) .. controls (67.5109,63.7093) and (76.6910,76.8749) ..
    (101.3409,77.8051) .. controls (126.7403,78.2086) and (135.1709,63.7093) ..
    (135.1709,63.7093);
  \path[draw=black,line join=miter,line cap=butt,even odd rule,line width=0.451pt]
    (75.6177,70.5716) .. controls (75.6177,70.5716) and (90.1784,62.5374) ..
    (101.5541,62.4940) .. controls (112.1315,62.4537) and (128.0152,70.6368) ..
    (128.0152,70.6368);
  \path[draw=black,line join=miter,line cap=butt,even odd rule,line width=0.451pt]
    (265.4260,63.7093) .. controls (265.4260,63.7093) and (274.6060,76.8749) ..
    (299.2559,77.8051) .. controls (324.6554,78.2086) and (333.0859,63.7093) ..
    (333.0859,63.7093);
  \path[draw=black,line join=miter,line cap=butt,even odd rule,line width=0.451pt]
    (273.1840,70.4923) .. controls (273.1840,70.4923) and (288.0934,62.5374) ..
    (299.4691,62.4940) .. controls (310.0466,62.4537) and (326.3762,70.3760) ..
    (326.3762,70.3760);
  \path[draw=black,line join=miter,line cap=butt,even odd rule,line width=0.451pt]
    (105.3961,16.3959) .. controls (105.5680,16.2404) and (110.6006,17.9888) ..
    (113.6064,25.5947) .. controls (117.4039,35.2040) and (117.3268,42.2774) ..
    (114.3164,51.6559) .. controls (112.0622,58.6784) and (106.6902,63.1189) ..
    (105.9005,62.8353);
  \path[draw=black,dash pattern=on 0.45pt off 0.45pt,line join=miter,line
    cap=butt,miter limit=4.00,even odd rule,line width=0.451pt] (106.3750,16.4458)
    .. controls (106.1976,16.2910) and (101.0021,18.0316) .. (97.8991,25.6035) ..
    controls (93.9788,35.1697) and (94.0584,42.2115) .. (97.1662,51.5481) ..
    controls (99.4933,58.5392) and (105.0390,62.9598) .. (105.8543,62.6775);
  \path[draw=black,line join=miter,line cap=butt,even odd rule,line width=0.451pt]
    (295.3567,16.4067) .. controls (295.5296,16.2516) and (300.5920,17.9954) ..
    (303.6155,25.5814) .. controls (307.4354,35.1653) and (307.3579,42.2202) ..
    (304.3297,51.5741) .. controls (302.0622,58.5781) and (296.6584,63.0069) ..
    (295.8641,62.7241);
  \path[draw=black,dash pattern=on 0.45pt off 0.45pt,line join=miter,line
    cap=butt,miter limit=4.00,even odd rule,line width=0.451pt] (296.3137,16.4130)
    .. controls (296.1376,16.2584) and (290.9818,17.9961) .. (287.9024,25.5552) ..
    controls (284.0121,35.1053) and (284.0910,42.1353) .. (287.1751,51.4561) ..
    controls (289.4844,58.4354) and (294.9879,62.8485) .. (295.7969,62.5667);
  \path[draw=black,line join=miter,line cap=butt,even odd rule,line width=0.451pt]
    (129.1814,69.9110) .. controls (129.1814,69.9110) and (131.1622,64.7002) ..
    (133.4932,62.3088) .. controls (136.6278,59.0929) and (141.7109,58.3536) ..
    (146.1547,57.7046) .. controls (151.6963,56.8954) and (157.4453,57.5441) ..
    (162.9271,58.6912) .. controls (169.2754,60.0197) and (174.8011,64.2567) ..
    (181.1794,65.4331) .. controls (195.9496,68.1572) and (211.4073,68.3203) ..
    (226.2346,65.9264) .. controls (233.5586,64.7439) and (240.1480,60.7261) ..
    (247.2823,58.6912) .. controls (252.6141,57.1705) and (254.9736,56.9180) ..
    (261.3415,58.2801) .. controls (264.1732,58.8859) and (266.7787,59.7428) ..
    (269.2790,62.8810) .. controls (270.3624,64.2409) and (271.6867,69.4646) ..
    (271.6867,69.4646);
  \path[draw=black,dash pattern=on 0.45pt off 0.45pt,line join=miter,line
    cap=butt,miter limit=4.00,even odd rule,line width=0.451pt] (129.2369,69.8529)
    .. controls (129.2369,69.8529) and (129.6153,72.8890) .. (131.1911,74.9703) ..
    controls (133.2963,77.7511) and (134.8677,78.1711) .. (140.0706,79.5745) ..
    controls (145.2735,80.9779) and (152.4552,82.2845) .. (159.7147,81.6193) ..
    controls (164.6384,81.1680) and (169.3040,78.8081) .. (174.2731,78.0946) ..
    controls (182.6041,76.8984) and (191.0165,75.8111) .. (199.4317,75.9569) ..
    controls (209.0911,76.1244) and (207.3361,75.8209) .. (225.5769,77.1080) ..
    controls (234.4650,77.7351) and (241.5648,81.7227) .. (249.4200,81.8766) ..
    controls (253.6762,81.9600) and (258.8990,82.3339) .. (262.7184,80.4540) ..
    controls (265.6323,79.0197) and (266.3990,79.0069) .. (268.6255,76.4311) ..
    controls (270.0991,74.7265) and (271.6187,69.3986) .. (271.6187,69.3986);
  \path[fill=black,line join=miter,line cap=butt,line width=0.800pt]
    (119.3427,41.6682) node[above right] {$\gamma_i$};
  \path[fill=black,line join=miter,line cap=butt,line width=0.800pt]
    (200,67) node[above right] {$\gamma_i'$};
  \path[fill=black,line join=miter,line cap=butt,line width=0.800pt]
    (310.0966,41.6682) node[above right] {$\gamma_j$};
  \end{tikzpicture}
  \caption{Handlesliding $\ggm_i$ over $\ggm_j$.}
  \label{fig:handlesliding}
\end{figure}

The following result is classical, we refer to \cite[Proposition 2.2]{os-threemanifold} for a proof. One can find a detailed discussion in \cite{JT} as well.
\begin{theorem}\label{thm:handleslide}
Two Heegaard diagrams $(\gS,\bfga,\bfgb)$ and $(\gS',\bfga',\bfgb')$ represent the same three--manifold if and only if  they are diffeomorphic after a finite sequence  of the following moves:

\begin{enumerate} \setlength\itemsep{1em}
	 \item \textbf{Isotopy.} Two Heegaard diagrams $(\gS,\bfga,\bfgb)$ and $(\gS',\bfga',\bfgb')$ are isotopic if $\gS$ and $\gS'$ are of the same genus and there are two one--parameter families $\bfga_t$ and $\bfgb_t$ of $g$--tuples of curves, moving by isotopies so that, for each $t$, both the $\bfga_t$ and the $\bfgb_t$ are $g$--tuples of smoothly embedded, pairwise disjoint curves such that $(\bfga_0, \bfgb_0) = (\bfga, \bfgb), (\bfga_1, \bfgb_1) = (\bfga', \bfgb')$.
	
	 \item \textbf{Stabilization.} We say that the diagram $(\gS',\bfga',\bfgb')$ is obtained from $(\gS,\bfga,\bfgb)$ by stabilization if $\gS' \cong \gS \, \# \, T^2$ (connected sum) and $\bfga' = \{ \ga_1, \ldots, \ga_g, \ga_{g+1}\}$, $\bfgb' = \{ \gb_1, \ldots, \gb_g, \gb_{g+1}\}$, where $\ga_{g+1}$, $\gb_{g+1}$ is a pair of curves in $T^2$ which meet transversally in a single point. 
	
	\item \textbf{Destabilization.} It is an inverse move to a stabilization.
	
	 \item \textbf{Handleslide.} We say that the diagram $(\gS',\bfga',\bfgb')$ is obtained from $(\gS,\bfga,\bfgb)$ by a handleslide if $\gS$ and $\gS'$ are of the same genus and either $\bfga=\bfga'$ and $\bfgb'$ is obtained from $\bfgb$ by a handleslide, or $\bfgb'=\bfgb$ and $\bfga'$ is obtained
from $\bfga$ by a handleslide. 
\end{enumerate}

\end{theorem}
\begin{proof}[Idea of proof]
One of the methods of proving, or at least understanding, the result, is to use Cerf theory again. Namely, choose a Riemannian metric
on a three-manifold $Y$ and suppose $F_0$ and $F_1$ are two different Morse functions on $Y$ having a single minimum. We connect
$F_0$ and $F_1$ by a generic path $F_t$ in the space of smooth functions on $Y$ as we did above. This time, however, we take into
account not only situations, where $F_t$ ceases to be a Morse function (which correspond to births/deaths of critical points),
but also situations, where $\nabla F_t$ ceases to be a Morse--Smale flow. 
These situations correspond precisely to the handle slides. A very detailed discussion is included in \cite{JT}; the proof
of Theorem~\ref{thm:handleslide} in \cite{os-threemanifold} does not appeal to Cerf theory.
\end{proof}

In Heegaard Floer theory, we will need to add an extra structure on Heegaard diagrams.
\begin{definition}[Pointed Heegaard diagram]
A \emph{pointed Heegaard diagram} is a quadruple $(\gS, \bfga, \bfgb, z)$, where $z \in \gS \setminus (\bfga \cup \bfgb)$. 
\end{definition}

\subsection{Symmetric products} \label{Symmetric_products}

Let $(\gS, \bfga, \bfgb, z)$ be a pointed Heegaard diagram. Let us consider the symmetric product
\[ 
\Sym^g(\gS) = \overbrace{\gS \times \ldots \times \gS}^{g} / S_g,
\]
where $S_g$ is the symmetric group on $g$ letters.
In other words, $\Sym^g(\gS)$ consists of unordered $g$-tuples of points in $\gS$ where we also allow repeated points. Observe that $\Sym^g(\gS)$ is a manifold.
\begin{problem}
Prove that $\pi_1(\Sym^g(\gS))$ is abelian.
\end{problem}

\begin{problem}\label{prob:h1}
Let $i\colon H_1(\gS; \ZZ)\to H_1(\Sym^g(\gS); \ZZ)$ be a map induced by the inclusion $\gS\times\{*\}\times\dots\times\{*\}$ to $\Sym^g(\gS)$.
On the other hand, observe that a curve in $\Sym^g(\gS)$ in a general position corresponds to a map from a $g$--fold cover of $S^1$ to $\gS$ and in this
way we might define a map $j\colon H_1(\Sym^g(\gS); \ZZ)\to H_1(\gS; \ZZ)$. Show that the two maps $i$ and $j$ are inverse to each other.
\end{problem}
\begin{proposition}[see \expandafter{\cite[Proposition 2.7]{os-threemanifold}}]\label{prop:g2}
Let $g>2$, then $\pi_2(\Sym^g(\gS))=\ZZ$ and $\pi_1(\Sym^g(\gS))$ on $\pi_2(\Sym^g(\gS))$ is trivial.
\end{proposition}
\begin{remark}\label{rem:g=2}
For $g=2$ we still have $\pi_2(\Sym^g(\gS))=\ZZ$, but the action of $\pi_1(\Sym^g(\gS))$ is no longer trivial, this poses minor problems, one avoids them
by requiring $g>2$.
\end{remark}
The manifold $\Sym^g(\gS)$ inherits a complex structure and a symplectic structure from $\gS$. Let $J$ denote this complex structure on $\Sym^g(\gS)$. Consider the products
\[
\TT_\ga = \ga_1 \times \ldots \times \ga_g / S_g
\]
and
\[
\TT_\gb = \gb_1 \times \ldots \times \gb_g / S_g.
\]
\begin{problem}
Show that $\TT_\ga$ and $\TT_\gb$ are totally real, that is, at each point $x\in\TT_\ga$ we have $T_x\TT_\ga\cap JT_x\TT_\ga=\{0\}$.
\end{problem}
\begin{remark}
The fact that $\TT_\ga$ and $\TT_\gb$ are totally real might make one think that Heegaard Floer theory is a Lagrangian Floer theory of 
the intersections of $\TT_\ga$ and $\TT_\gb$. While this was generally believed since the dawn of Heegaard Floer theory,
the details were worked out only a few years later by Perutz \cite{Per}.
\end{remark}
\begin{problem}\label{prob:intersect}
Show that there is a $1-1$ correspondence between points $\bfx \in \TT_\ga \cap \TT_\gb$ and $g$-tuples of points $(x_1, \ldots, x_g) \in \gS \times \ldots \times \gS$ such that there exists a permutation $\gs \in S_g$ and $x_i \in \ga_i \cap \gb_{\gs(i)}$. 
\end{problem}

\begin{problem}
Show that if each of the $\ga$--curves is transverse to each of the $\gb$--curves, then also $\TT_\ga$ intersects $\TT_\gb$ transversally.
\end{problem}

\begin{problem}
Let $T_\ga$ be the image of $H_1(\TT_\ga; \ZZ)$ in $H_1(\Sym^g(\gS); \ZZ)$, let also $T_\gb$ be the image of $H_1(\TT_\gb; \ZZ)$ in $H_1(\Sym^g(\gS); \ZZ)$.
Prove that
\[H_1(\Sym^g(\gS); \ZZ)/(T_\ga+T_\gb)\cong H_1(\gS; \ZZ)/([\ga_1],\ldots,[\gb_g])\cong H_1(Y; \ZZ).\]
\end{problem}

Chose two paths $a$ and $b$, one belonging to $\TT_\ga$, the other belonging to $\TT_\gb$. Assume that they have the same endpoints 
$x,y\in\TT_\ga\cap\TT_\gb$. These two paths form a loop $\gamma\in\pi_1(\Sym^g(\gS))$.
\begin{problem}\label{prob:gammadoesnotdepend}
Prove that $\gamma$ depends only on $x$ and $y$ and not on $a$ and $b$.
\end{problem}
Taking the solution of Problem~\ref{prob:gammadoesnotdepend} for granted, with each pair of points $x,y\in\TT_{\ga}\cap\TT_{\gb}$ we
associate an element $\epsilon(x,y)\in H_1(Y; \ZZ)$.
\begin{problem}
Prove that $\epsilon$ is additive in the sense that $\epsilon(x,y)+\epsilon(y,z)=\epsilon(x,z)\in H_1(Y; \ZZ)$.
\end{problem}
There exists another description of the class $\epsilon(x,y)$ (the reader might want to look back to Section~\ref{section:spinc_structures} before reading
this paragraph). To begin with, choose a point $x\in\TT_{\ga}\cap\TT_{\gb}$.
Each such  point, by Problem~\ref{prob:intersect}, corresponds to a set of $g$--points $x_1,\ldots,x_g$, such that 
$x_i\in\ga_i\cap\gb_{\gs(i)}$, where $\gs$ is some permutation of the set $\{1,\ldots,g\}$. Each of the $x_i$ corresponds to a trajectory
$\gamma_i$ of the vector field $\nabla F$, which connects a critical point of index $1$ to a critical point of index $2$. There is 
also a unique trajectory $\gamma_z$ passing through the point $z$. It connects the critical point of index $0$ with the critical point of index $3$.
Take now small neighborhoods $U_1,\ldots,U_g,U_z$ of the trajectories $\gamma_1,\ldots,\gamma_g,\gamma_z$. Let $Y^o$ be the complement
$Y\setminus (U_1\cup\ldots\cup U_g\cup U_z)$. The vector field $\nabla F$ does not vanish on $Y^o$. The pair $(Y^o,\nabla F)$ defines then
a so--called smooth Euler structure on $Y$; see \cite{Tur97}.
By the result of Turaev, a smooth Euler structure corresponds to a \spinc{} structure
on $Y$ \cite[Proposition 2.7]{Tur97}. Call this structure $\sss_x$. Each \spinc{} structure has its Chern class $c_1\in H^2(Y;\ZZ)$, as was discussed in Section \ref{section:spinc_structures}.
\begin{proposition}[see \expandafter{\cite[Lemma 2.19]{os-threemanifold}}]
Given any two points $x,y\in\TT_\ga\cap\TT_\gb$, the difference $c_1(\sss_x)-c_1(\sss_y)$ is the Poincar\'e dual to $\epsilon(x,y)$.
\end{proposition}

\subsection{The chain complex $\widehat{CF}$} \label{Chain_complexes1}
We will work mostly over $\ZZ_2$. For simplicity, unless specified otherwise, we will assume that $b_1(Y)=0$.

Let $(\gS,\bfga,\bfgb,z)$ be a pointed Heegaard diagram for $Y$. Assume that the $\bfga$ and $\bfgb$ curves intersect transversally. Then, the chain complex $\widehat{CF}$ is defined (over $\ZZ_2$) to be generated by the intersection points $\TT_\ga \cap \TT_\gb$. 
\begin{remark}\label{rem:admissible}
There are a few technical assumptions on the Heegaard diagram used in the construction of the chain complex. First of all, we usually
assume that $g>2$ (case $g=1$ is very special and also possible, see \cite[Remark 2.16]{os-threemanifold}); see Remark~\ref{rem:g=2} for the case
$g=2$. 

If $b_1(Y)>0$, one adds an extra assumption on the Heegard diagram, namely admissibility, see \cite[Section 5]{os-threemanifold}.
For example, this condition rules out a diagram for $S^2\times S^1$, where $\gS$ is a torus and the $\alpha$--curve and the $\beta$--curve
are parallel, so $\TT_\ga\cap\TT_\gb$ is empty. An admissible Heegaard diagram for $S^2\times S^1$ can be obtained by moving the $\beta$--curve
by an isotopy
in such a way that two intersection points with the $\alpha$--curve are created.
\end{remark}

We now define the differential $\partial$. Let $\bfx, \bfy \in \TT_\ga \cap \TT_\gb$ be two intersection points. 
Denote by $\pi_2(\bfx, \bfy)$ the set of relative homotopy classes of disks $\phi\colon D^2 \rightarrow \Sym^g(\gS)$ with $\phi(-1) = \bfx$, $\phi(1) = \bfy$, $\phi(\partial_+D^2) \subset \TT_\ga$ and $\phi(\partial_-D^2) \subset \TT_\gb$. 
Here $D^2$ is the unit disk in the complex plane, $\partial_\pm D^2$ is the part of the boundary having positive (respectively: negative) 
imaginary part.

\begin{problem}
Show that $\pi_2(\bfx,\bfy)$ can be non-empty only if $\epsilon(\bfx,\bfy)=0$.
\end{problem}
\begin{problem}
Show that $\pi_2(\bfx,\bfy)$ admits a multiplication defined as 
$$\pi_2(\bfx,\bfy)\star\pi_2(\bfy,\bfz)\to\pi_2(\bfx,\bfz)$$
Show that $\star$ is associative. Prove also that $\pi_2(\bfx,\bfx)$ is a group.
\end{problem}
\begin{problem}
Show that there is an action of $\pi_2(\Sym^g(\gS)) = \ZZ$ on each of the sets $\pi_2(\bfx,\bfy)$.
\end{problem}
\begin{problem}
Draw a standard $g=1$ Heegaard diagram for a lens space $L(p,q)$. Show that there are precisely $p$ intersection points of the $\ga$--curves with
$\gb$--curves and $\epsilon(\bfx,\bfy)\neq 0$ as long as $\bfx\neq\bfy$. 
\end{problem}

\begin{problem}\label{prob:holofix}
Write explicitly all holomorphic maps from $D^2$ to $D^2$ that fix $-1$ and $1$ and take $\partial_+D^2$ to $\partial_+D^2$. Show that
the space of these maps can be parametrized by $\RR$.
\end{problem}

Given $\phi \in \pi_2(\bfx, \bfy)$, a \emph{holomorphic} representative for $\phi$ is a map $u\colon D^2\to\Sym^g(\gS)$ 
in the homotopy class $\phi$ that is holomorphic. Recall that $\Sym^g(\gS)$ has a complex structure induced from $\gS$ and $D^2$ has a standard
complex structure.
\begin{remark}
For various genericity results, the complex structure on $\Sym^g(\gS)$ induced from a complex structure on $\gS$ might be too rigid and one often needs to consider almost
complex structures (that is, endomorphisms of the tangent bundle whose square is minus the identity) and pseudo-holomorphic maps instead
of holomorphic. We refer to \cite[Section 3.1]{os-threemanifold} for more details.
\end{remark}

We denote by $\cM(\phi)$ the space of holomorphic representatives of $\phi$. 
For any class $\phi\in\pi_2(\bfx,\bfy)$, there is an integer $\mu(\phi) \in \ZZ$ called the \emph{Maslov index}. 
A detailed definition of the Maslov index in Heegaard
Floer theory can be found in \cite[Section 4]{Lip}.
The Maslov index is the dimension of the moduli space of holomorphic representatives (provided the almost complex structure
is sufficiently generic).
By Problem~\ref{prob:holofix}, 
there is an action of $\RR$ on $\cM(\phi)$ given by the automorphisms of the domain $D^2$ that fix $1$ and $-1$. 
If $\phi$ is not the class of a constant map, and the complex structure on $\gS$ was generic,
then the quotient $\widehat{\cM}(\phi) = \cM(\phi) / \RR$ is a smooth manifold of dimension $\mu(\phi) - 1$; see \cite{RS93}. 
If additionally $\mu(\phi) = 1$, we define
\[\# \widehat{\cM}(\phi) \in \ZZ\]
to be the number of the elements in the quotient. 
\begin{remark}
In \cite[Section 3.6]{os-threemanifold} there is described a way to associate a sign to each element $\widehat{\cM}(\phi)$ as long as $\mu(\phi)=1$. This allows us to define
the differential in the Heegaard Floer theory over $\ZZ$. As we already mentioned above, we will mostly focus on the theory over $\ZZ_2$.
\end{remark}

The basepoint $z$ can be used to construct a codimension two submanifold (in the language of algebraic geometry: a divisor), 
$R_z := \gS\times\ldots\times\gS\times\{z\}\subset \Sym^g(\gS)$ (the product is formally defined in $\gS^{\times g}$, we project it
to $\Sym^g(\gS)$).
\begin{problem}
Observe that, by construction, $\TT_\ga$ and $\TT_\gb$ are disjoint from $R_z$. 
\end{problem}
Given intersection points $\bfx, \bfy \in \TT_\ga \cap \TT_\gb$ and a class $\phi \in \pi_2(\bfx, \bfy)$, we define $n_z(\phi)$ to be the intersection number between $\phi$ and $R_z$. 

The differential for $\widehat{CF}$ is then given by

\begin{equation} \label{diff:hat}
\partial \bfx = \sum_{\bfy \in \TT_\ga \cap \TT_\gb} \sum_{\substack{\phi \in \pi_2(\bfx, \bfy) \\ n_z(\phi)=0, \; \mu(\phi)=1}} 
\# \widehat{\cM}(\phi) \; \bfy.
\end{equation}
In a few words, the differential counts holomorphic disks between $\bfx$ and $\bfy$ which do not intersect the divisor $R_z$.

\begin{problem}
Show that for a lens space with a standard Heegaard diagram and $g=1$, $\partial\bfx=0$ for all $\bfx\in\TT_\ga\cap\TT_\gb$.
\end{problem}
\begin{problem}
Take the standard diagram for $S^2\times S^1$ with $g=1$. Move the $\ga$--curve so that it intersects the $\gb$--curve at precisely two
points $\bfx$ and $\bfy$. Calculate the differential and the homology groups 
(compare Remark~\ref{rem:admissible}).
\end{problem}
The following fact holds.
\begin{theorem}[see \expandafter{\cite[Theorem 4.1]{os-threemanifold}}]\label{thm:hfhat}
We have $\partial^2=0$. The homologies $\widehat{HF}(Y)$ are independent of the choice of the Heegaard diagram, and, therefore, are
invariants of the three--manifold $Y$.
\end{theorem}
\begin{remark}\label{rem:based}
The words `independent of the choice' might have different meanings. Originally, in \cite{os-threemanifold}, 
it was proved that a change of the Heegaard diagrams as in Section~\ref{Heegard_diagrams} above
changes $\widehat{HF}(Y)$ by an isomorphism. Therefore, $\widehat{HF}(Y)$ was well defined up to isomorphism. In \cite{JT} 
Juh\'asz and Thurston showed more, namely the naturality of the Heegaard Floer theory. Naturality means that the Heegaard Floer theory
assigns a concrete group to each based\footnote{A \emph{based manifold} is a manifold with a choice of a base point.} three--dimensional manifold and each diffeomorphism of a based manifold induces an isomorphism of corresponding Heegaard Floer groups. This naturality property is proved for all flavors of the Heegaard Floer homology. It lies at the heart of
the involutive Floer theory as defined in \cite{HM} via the maps studied in detail in \cite{Sar,Zem1}; see also \cite{HMZ}.
\end{remark}
\begin{problem}
Prove that $\widehat{HF}(Y)$ splits as a direct sum $\widehat{HF}(Y,\sss)$ over all the \spinc{} structures of $Y$.
\end{problem}
\begin{problem}\label{prob:nontrivialp}
Prove that if $Y$ is a rational homology sphere, then 
$\widehat{HF}(Y,\sss)$ is non-trivial for any \spinc{} structure. In particular, $\rk\widehat{HF}(Y)\ge |H_1(Y;\ZZ)|$, where
$|\cdot|$ denotes the cardinality of a set.
\end{problem}
\begin{definition}[L--space]
A rational homology sphere is called an \emph{L--space} if $$\rk\widehat{HF}(Y)= |H_1(Y;\ZZ)|.$$
\end{definition}
\begin{problem}
Prove that all the lens spaces are L--spaces.
\end{problem}

\subsection{Complexes $CF^-$, $CF^+$ and $CF^\infty$} \label{section_complexes}
The complex structure on $\Sym^{g}(\gS)$ and the holomorphicity of the maps used in the definition of $\cM$ were used to give rigidity to the space $\cM$
(to make sure it has a finite dimension). 

The existence of this structure has one more consequence. Namely that the $n_z(\phi)$ defined above is always non--negative. We will define a new chain complex,
 where we count all the holomorphic
disks with $\mu(\phi)=1$, regardless of the value of $n_z(\phi)$. The chain complex $CF^-$ is generated by the intersection points
$\TT_\ga\cap\TT_\gb$, but this time not over $\ZZ_2$, but over the ring $\ZZ_2[U]$, where $U$ is a formal variable.
The differential for $CF^-$  is defined by

\begin{equation} \label{diff:minus}
\partial \bfx := \sum_{\bfy \in \TT_\ga \cap \TT_\gb} \sum_{\substack{\phi \in \pi_2(\bfx, \bfy) \\ \mu(\phi)=1}} \# \widehat{\cM}(\phi) \, U^{n_z(\phi)} \, \bfy.
\end{equation}

\begin{theorem}
We have $\partial^2=0$. The homology groups $HF^-(Y)$ do not depend on the choice of the Heegaard diagram.
\end{theorem}
Remark~\ref{rem:based}, explaining the meaning of `do not depend', still applies in the case of $HF^-$.
As before, the group $HF^-(Y)$ splits as a sum over the \spinc{} structures of $Y$. We also have the following fact, which is not very hard to prove.
\begin{proposition}
A three--manifold $Y$ is an L--space if and only if, for every $\sss$, $HF^-(Y,\sss)\cong\ZZ_2[U]$.
\end{proposition}

In algebra there is a procedure called localization, which roughly means inverting formally some variables in a ring. For example, the localization
of $\ZZ_2[U]$ with respect to the multiplicative system generated by $U$ is the ring $\ZZ_2[U,U^{-1}]$. We can perform this operation on the
module $CF^-$: define a chain complex as generated by $\TT_\ga\cap\TT_\gb$, but this time over $\ZZ_2[U,U^{-1}]$. The chain complex will be denoted
by $CF^\infty$. The differential is defined as in \eqref{diff:minus}. The homology of the complex is well-defined and will be denoted by 
$HF^\infty(Y,\sss)$. As it might be expected, by passing to a localization, we lose some information. Actually we lose a lot: namely, the following holds.

\begin{theorem}[see \expandafter{\cite[Theorem 10.1]{os-threemanifoldapps}}]\label{thm:itistrivial}
Suppose $Y$ is a rational homology sphere.
We have an isomorphism of $\ZZ_2[U,U^{-1}]$--modules $$HF^\infty(Y,\sss)\cong \ZZ_2[U,U^{-1}].$$ 
\end{theorem}
\begin{remark}
Theorem~\ref{thm:itistrivial} allows generalizations for non rational homology spheres; again, see \cite[Theorem 10.1]{os-threemanifoldapps}.
\end{remark}
The chain complex $CF^-$ can be regarded as a subcomplex of $CF^\infty$. For this, we need to regard $CF^\infty$ as a complex over $\ZZ_2[U]$. 
The quotient complex $CF^+(Y)$ is well defined. This is a chain complex over $\ZZ_2[U]$. 
The homologies are called $HF^+(Y)$.
\begin{problem}
Prove that for every element $a\in CF^+$ there exists $k\ge 0$ such that $U^ka=0$.
\end{problem}

The short exact sequence
\[ 0 \rightarrow CF^- \rightarrow CF^\infty \rightarrow CF^+ \rightarrow 0 \]
gives rise to an exact triangle in homology.

\begin{proposition}
There exists yet another short exact sequence
\[
0 \rightarrow \widehat{CF} \rightarrow CF^+ \overset{\cdot U}{\rightarrow} CF^+ \rightarrow 0\]
giving rise to a long exact sequence in homology.
\end{proposition}
\begin{problem}
Write precisely the two long exact sequences mentioned above. Watch out for grading shifts; these will be introduced below.
\end{problem}
\begin{problem}
Prove that $HF^+(Y,\sss)$ splits non-canonically as a sum of a part isomorphic to $\ZZ_2[U,U^{-1}]/(U)$ and a part finitely generated over $\ZZ_2$. Show that
$Y$ is an L--space if and only if for every $\sss$ we have $HF^+(Y,\sss)=\ZZ_2[U,U^{-1}]/(U)$ as $\ZZ_2[U]$ modules.
\end{problem}

So far we have defined various chain complexes, but we have not defined a grading yet. We have the following useful Lemma.

\begin{lemma}[see \expandafter{\cite[Lemma 3.3]{os-threemanifold}, \cite{OzSz-intro2}}]\label{lem:grading1}
If $g>2$, then for any $\phi\in\pi_2(\bfx,\bfy)$ the difference $\mu(\phi)-2n_z(\phi)$ does not depend on the specific choice 
of $\phi$, only on $\bfx$ and $\bfy$.
\end{lemma}
Lemma~\ref{lem:grading1} allows us to define the relative grading of chain complexes. Namely, we define
the \emph{Maslov grading} $M(\bfx)-M(\bfy)=\mu(\phi)-2n_z(\phi)$. The differential decreases the Maslov grading by $1$, provided
we require that the multiplication by $U$ shifts the Maslov grading by $-2$. Later on we will show that the Maslov grading
gives rise to an absolute grading.


\begin{problem}\label{prob:connected}
Suppose that $(M_1, \fs_1)$ and $(M_2, \fs_2)$ are two three--manifolds. Prove the following Künneth formula for $\widehat{CF}$:	
$$ \widehat{CF}( M_1 \# M_2, \fs_1 \# \fs_2 ) \cong \widehat{CF}(M_1, \fs_1) \otimes \widehat{CF}(M_2, \fs_2) $$
\end{problem}

\section{Why do things work?}\label{sec:why}
It is not that hard to define invariants of three--manifolds. It is hard, though, to construct \emph{meaningful} invariants. This means, invariants
over which we have some control, and for which we can calculate some non-trivial estimates. In this section we are going to give two highly non-trivial
results,  which lie at the heart of the Heegaard Floer 
theory. These are the adjunction inequality
and the surgery exact sequence. Many crucial results in Heegaard Floer theory rely on these two results. 
\subsection{Adjunction inequality}
In algebraic geometry one has the so-called adjunction formula. In short if $D$ is a smooth divisor in a projective variety $X$ and $K_D$, $K_X$ denote canonical divisors, then $K_D=(K_X+D)|_D$. For readers not aquainted with the language of algebraic geometry, one can think of $K_D$ and $K_X$ as (first Chern classes of)
complex line bundles $K_D=\Lambda^{\dim D} T^*D$, $K_X=\Lambda^{\dim X}T^*X$ and the divisor $D$ defines a complex line bundle, whose first Chern class is Poincar\'e dual to the class of $D$. The sum of divisors corresponds to a tensor product of line bundles and restriction means the restriction of line bundles in the ordinary sense. We refer to any textbook in algebraic geometry, like \cite{Hart}, for more details. With this setting, the adjunction formula is almost a tautology.

As a special case, suppose that $C$ is a smooth complex curve in a projective surface $X$ and
$K$ is the canonical divisor. We have that $K_C=(K_X+C)|_C$ and applying the classical Riemann--Roch theorem yields 
\begin{equation}\label{eq:adjunctioneq}
\chi(C)=-C(C+K_X).
\end{equation} 
For example, if $X=\CC P^2$ and $C$ is a smooth complex curve of degree $d$, then in $H_2(X;\ZZ)$ we have $C=dH$, $K=-3H$, where $H$ is
the class of a line and so $\chi(C)=-d(d-3)$. Equation~\eqref{eq:adjunctioneq}
is sometimes referred to as the adjunction equality.

It is trivial to see that the adjunction equality \eqref{eq:adjunctioneq} has no chances to hold 
in a smooth category. For example, draw a genus $g$ surface in $\CC^2$, it is a homologically trivial surface in the compactification $\CC P^2$,
so \eqref{eq:adjunctioneq} would imply that $2-2g=0$.

A wonderful tool in Seiberg--Witten theory is the \emph{adjunction inequality}. Recall that Seiberg--Witten theory assigns
to every \spinc{} structure $\sss$ on a smooth four--manifold $X$ with $b_2^+(X)>1$ an integer number $SW_X(\sss)$. We have the following
remarkable theorem,  which we state in a simple form, see e.g. \cite[Section 10]{Scor} for a more detailed version.
Other sources are \cite[Section 40]{KM-book} and \cite[Section 4.6]{Nic}.

\begin{theorem}[Adjunction inequality in Seiberg--Witten theory]\label{thm:swadj}
Suppose $X$ is a smooth four--manifold with $b_2^+(X)>1$. Let $C\subset X$ be a smooth closed connected embedded surface such that
$C^2\ge 0$ and $C$ is homologically non--trivial. If $\sss$ is a \spinc{} structure on $X$ such that $SW_X(\sss)\neq 0$,
then $\chi(C)+C^2\le -|\langle c_1(\sss),C\rangle|$.
\end{theorem}

The assumption that $C$ is smooth is essential. For example, in \cite{LeeWil} there are constructed locally flat embedded surfaces $C$
in $\CC P^2$ such that $\chi(C)>-d(d-3)$, where $d$ is the degree of $C$. This problem is related to showing that the topological four--genus
of some algebraic knots is strictly less than the smooth four--genus; see \cite{Rud,Baa}.\footnote{Of course, one can complain
that $b_2^+(\CC P^2)=1$, so technically speaking locally flat 
curves in $\CC P^2$ are not counterexamples to the statement of Theorem~\ref{thm:swadj}, but they give an idea of the reason why
Theorem~\ref{thm:swadj} does not hold in the topological locally flat category.}

\smallskip
In Heegaard Floer theory, the adjunction inequality is a key tool in proving many important theorems. The formulation below
involves manifolds with $b_1>0$. In that case, the homology $HF^+(Y,\sss)$ can be zero for some \spinc{} structures,
unlike in the case $b_1=0$ (cf. Problem~\ref{prob:nontrivialp}).

\begin{theorem}[Adjunction Inequality] 
Suppose $Y$ is a three--manifold with $b_1(Y)>0$. Let $\sss$ be a \spinc{} structure for which $HF^+(Y,\sss)$ is non--zero. Suppose $Z\subset Y$
is a smooth closed oriented surface and $g(Z)>0$. Then $|\langle c_1(\sss),[Z]\rangle|\le 2g(Z)-2$.
\end{theorem}
The adjunction inequality is proved in \cite[Section 7]{os-threemanifoldapps}.

\subsection{The surgery exact sequence}
One of the most important basic tools for calculating the Heegaard Floer invariants is the surgery exact sequence. The most basic
form of it is often used as a template for proving more general statements. A surgery exact sequence exists in the Seiberg--Witten
Floer theory (see for example \cite[Section 42]{KM-book} and references therein).
In Heegaard Floer theory, we have a way of calculating any surgery on a null--homologous knot in an integer homology three--sphere,
provided we know its knot Floer chain complex; see \cite{OzSz-integer} for details. This general surgery formula relies on the
following fundamental result, see \cite[Theorem 1.7]{os-threemanifoldapps}.

\begin{theorem}[Surgery Exact Sequence]
Let $Y$ be an integral homology three--sphere and $K \subset Y$ be a knot. Then there exists a $U$--equivariant exact sequence:
\[ \ldots \rightarrow HF^+(Y) \rightarrow HF^+(Y_0) \rightarrow HF^+(Y_1) \rightarrow HF^+(Y) \rightarrow \ldots\]
where $Y_1$ is the $+1$ surgery and $Y_0$ is the $0$ surgery on $K$.
\end{theorem}
The surgery exact sequence is proved in \cite[Section 9]{os-threemanifoldapps}. The key idea is to find a triple Heegaard diagram,
that is a quintuple $(\gS,\bfga,\bfgb,\bfgc,z)$, such that $(\gS,\bfga,\bfgb,z)$ is a Heegaard diagram for $Y$,
$(\gS,\bfga,\bfgc,z)$ is a Heegaard diagram for $Y_0$ and $(\gS,\bfgb,\bfgc,z)$ is a Heegaard diagram for $Y_1$.
The details and the proof of the
existence of such a triple Heegaard diagram are given in \cite[Lemma 9.2]{os-threemanifoldapps}. Speaking very roughly, given the Heegaard
diagram, the maps in the surgery long exact sequence are built by counting holomorphic triangles, instead of holomorphic disks.

\section{Cobordisms and $d$--invariants.}\label{sec:cobordism}

\subsection{Absolute grading}
This section is based on \cite{OzSz-absolute}.
\begin{definition}
Let $(Y_1,\sss_1)$, $(Y_2,\sss_2)$ be two \spinc{} three--manifolds. We say that $(W,\sst)$ is a \emph{\spinc{} cobordism}
between $Y_1$ and $Y_2$ if $W$ is a smooth four--manifold with boundary $Y_2\sqcup -Y_1$ and $\sst$ is a \spinc{} structure on $W$
whose restriction to $Y_i$ is $\sss_i$, $i=1,2$.
\end{definition}

\begin{theorem}[see e.g. \expandafter{\cite[Section 2]{OzSz-absolute}}]
If $(W, \ft)$ is a smooth \spinc{} cobordism between $(Y_1,\sss_1)$ and $(Y_2,\sss_2)$, 
then there exist maps $F_{W, \ft}^{\bullet} : HF^{\bullet}(Y_1, \fs_1) \rightarrow HF^{\bullet}(Y_2, \fs_2)$ 
with $\bullet \in \{+, -, \infty\}$, making the following diagram commute
\begin{equation}\label{eq:newmap}	
\xymatrix{%
\ldots \ar[r] & HF^-(Y_1, \sss_1) \ar[d]^{F^-_{W, \ft}}\ar[r] & HF^{\infty}(Y_1, \sss_1) \ar[d]^{F^{\infty}_{W, \ft}} \ar[r] & 
HF^+(Y_1, \sss_1) \ar[d]^{F^+_{W, \ft}} \ar[r] &  \ldots \\
\ldots \ar[r] & HF^-(Y_2, \sss_2) \ar[r] & HF^{\infty}(Y_2, \sss_2) \ar[r] & HF^+(Y_2, \sss_2) \ar[r] &  \ldots \\	
}
\end{equation}
\end{theorem}
The idea of the proof is to split the cobordism into handle attachments. The non-trivial part comes from two--handle attachments, which
are basically dealt with using a refined version of the surgery exact sequence. 
We define a relative grading of the map induced by $F$.
\begin{theorem}[see \expandafter{\cite[Theorem 7.1]{OzSz-triangles}}]
The map $F_{W, \ft}^\bullet$ has relative Maslov grading equal to
\[\deg F_{W, \ft}^\bullet := \dfrac{c_1(\ft)^2 - 2 \chi(W) - 3\gs(W)}{4}.\]
\end{theorem} 

We can now make the gradings in Heegaard Floer homology groups absolute by requiring that the generator of $HF^-(S^3)$ be 
at Maslov grading $-2$, or, equivalently, that the lowest grading of
$HF^+(S^3)$ be at Maslov grading $0$.

\subsection{The $d$--invariants}
The fact that $F_{W,\sst}$ preserves the grading is very interesting, but on its own does not give much of insight in the behavior of Heegaard Floer
homology under cobordisms. A reader with some experience in Khovanov homology surely knows that the map in Khovanov homology induced
by a knot cobordism has a fixed grading, but we do not know much more about this map; even the question whether it is non-trivial is not 
well understood.

Luckily, in the Heegaard Floer case, we have the following crucial fact.

\begin{theorem}[see \expandafter{\cite[Proof of Theorem 9.1]{OzSz-absolute}}]\label{thm:zerononzero}
If $W$ has negative definite intersection form, and $Y_1,Y_2$ are rational homology spheres, then $F^{\infty}_{W,\sst}$ is an isomorphism. On the
contrary, if $b_2^+(W)>0$, then $F^\infty_{W,\sst}$ is the zero map.
\end{theorem}

\begin{definition}
Let $(Y,\sss)$ be a rational homology three--sphere. The \emph{$d$--invariant} or the \emph{correction term} $d(Y,\sss)$ is defined
as the minimal absolute grading of a non-trivial element $x\in HF^+(Y,\sss)$ which is in the image of $HF^\infty(Y,\sss)$.
\end{definition}

Let $(W, \ft)$ be a $\Spin^c$ cobordism between $(Y_1, \sss_1)$ and $(Y_2, \sss_2)$.
The main result  related to the $d$--invariants is the following.
\begin{theorem}[see \expandafter{\cite[proof of Theorem 9.9]{OzSz-absolute}}]\label{thm:dinvariants}
Suppose $(W,\sst)$ is a \spinc{} cobordism between rational homology spheres $(Y_1,\sss_1)$ and $(Y_2,\sss_2)$. If $b_2^+(W)=0$, then
\begin{equation}\label{eq:maininequality}
d(Y_2,\sss_2)-d(Y_1,\sss_1)\ge\frac14(c_1(\sst)^2-2\chi(W)-3\sigma(W)).
\end{equation}
\end{theorem}
\begin{problem}
Using \eqref{eq:newmap} and Theorem~\ref{thm:zerononzero}, prove Theorem~\ref{thm:dinvariants}.
\end{problem}

The $d$--invariants are strong enough to prove Donaldson's diagonalization theorem via Elkies' theorem; see \cite[Section 9]{OzSz-absolute}.
A version of $d$--invariants for manifolds with $b_1>0$, whose rudiments were established in \cite{OzSz-absolute},
and which was developed in full details in \cite{levine-ruberman},
can be used to reprove the 
Kronheimer-Mrowka result on the smooth four--genus of torus knots.
We refer again to \cite[Section 9]{OzSz-absolute}.

We gather now a few facts about the $d$--invariant, the first one is proved in \cite[Theorem 4.3]{OzSz-absolute}, 
while the second is proved in \cite[Proposition 4.2]{OzSz-absolute}.
\begin{proposition}\ \label{prop:additive}
\begin{itemize}
\item The $d$--invariant is additive. That is, if $(Y_1,\sss_1)$ and $(Y_2,\sss_2)$ are two rational homology three--spheres, then
$d(Y_1\# Y_2,\sss_1\#\sss_2)=d(Y_1,\sss_1)+d(Y_2,\sss_2)$.
\item Let $(Y,\sss)$ be a rational homology three--sphere. Then $d(-Y,\sss)=-d(Y,\sss)$. 
\end{itemize}
\end{proposition}
The first part of the proposition follows essentially from the K\"unneth principle (with some technical problems in homological algebra). 
However, the second part is more difficult than one could expect.

Using second part of Proposition~\ref{prop:additive} together with Theorem~\ref{thm:dinvariants} we obtain the following result.
\begin{corollary}\label{cor:rational}
If $(Y, \sss)$ bounds a rational homology ball $W$ (that is, if $H_k(W \, ; \, \QQ) = 0$ for $k \geq 1$) and the \spinc{} 
structure $\sss$ extends over $W$, then~$d(Y, \sss) = 0$.
\end{corollary}
\begin{problem}
Prove Corollary~\ref{cor:rational}.
\end{problem}
We will be able to calculate the $d$--invariants for a large class of three--manifolds using Heegaard Floer homology for knots.
This theory, usually called knot Floer theory, will be discussed in the next section.

\begin{problem}
Drill two balls from $\CC P^2$ so as to obtain a cobordism between two copies of $S^3$. Find all \spinc{} structures on the cobordism
that extend the \spinc{} structure on $S^3$ (use Corollary~\ref{cor:spinc4}). 
Use this example to show that Theorem~\ref{thm:dinvariants} dramatically fails if $b_2^+(W)>0$.
\end{problem}

\section{Heegaard Floer homology for knots}\label{sec:knots}
There is a variant of Heegaard Floer homology for knots and links. We will focus on
knots in $S^3$, although a significant part of the results carries through to null-homologous knots in rational homology spheres.
The case of links, though, does not seem to be more complicated at the beginning,
but there are surprisingly many highly non-trivial technical problems, e.g. if one tries to establish a surgery formula. The reader
with some experience in link theory might think that Heegaard Floer homology for links is more complicated than for knots in a similar manner
as Blanchfield forms for links are way more complicated than for knots.

\subsection{Heegaard diagrams and knots}\label{sec:heegaardknots}

Suppose $Y$ is a three--manifold and $(\gS,\bfga,\bfgb)$ is a Heegaard diagram for $Y$. Choose \emph{two} base points $z$ and $w$ in 
$\gS\setminus(\bfga\cup\bfgb)$. Such quintuple $(\gS,\bfga,\bfgb,z,w)$ is called a \emph{doubly pointed Heegaard diagram}.

Given a doubly pointed Heegaard diagram $(\gS,\bfga,\bfgb,z,w)$ we not only recover the manifold $Y$, but we obtain a way to encode
a knot in $Y$. To this end, suppose the Heegaard decomposition is $Y=U_0\cup_{\gS} U_1$.
Connect points $w$ and $z$ by two curves $a \subset \gS \setminus \{\alpha_1, \ldots, \alpha_g\}$, $b \subset \gS \setminus \{\beta_1, \ldots, \beta_g\}$, and then push $a$ into $U_0$ and $b$ into $U_1$. These two curves together result in a knot $K \subset Y$.

\begin{problem}
Prove that the isotopy type of $K$ does not depend on the actual choice of the curves $a$ and $b$.
\end{problem}

Conversely, a knot $K \subset Y$ determines a doubly pointed Heegaard diagram $(\gS, \bfga, \bfgb, w, z)$. We focus on the case $Y=S^3$.
Take a \emph{bridge presentation} of $K$, i.e., its projection with a division of $K$ into $2g+2$ segments (for some $g\ge 0$) $a_1, \ldots, a_{g+1}, b_1, \ldots, b_{g+1} \subset K$, such that all the crossings are only between segments $a_i$ and $b_j$ and in such a way that, for every intersection, $a_i$ always goes transversely over $b_j$ (see Figure \ref{fig:bridge_presentation_eight}).

\begin{figure}
  \centering
  \begin{tikzpicture}[y=0.80pt, x=0.80pt, yscale=-1.000000, xscale=1.000000, inner sep=0pt, outer sep=0pt]
  \path[draw=black,line join=miter,line cap=butt,even odd rule,line width=0.800pt]
    (60.0000,21.3622) -- (310.5000,21.3622);
  \path[draw=black,line join=miter,line cap=butt,even odd rule,line width=0.800pt]
    (60.0000,41.3622) -- (300.0000,41.3622);
  \path[draw=black,line join=miter,line cap=butt,even odd rule,line width=0.800pt]
    (89.5000,61.3622) -- (300.0000,61.3622);
  \path[draw=black,line join=miter,line cap=butt,even odd rule,line width=0.800pt]
    (100.0000,81.3622) -- (300.0000,81.3622);
  \path[draw=black,line join=miter,line cap=butt,even odd rule,line width=0.800pt]
    (100.0000,101.3622) -- (300.0000,101.3622);
  \path[draw=black,line join=miter,line cap=butt,even odd rule,line width=0.800pt]
    (100.0000,121.3622) -- (310.5000,121.3622);
  \path[draw=black,line join=miter,line cap=butt,even odd rule,line width=0.800pt]
    (100.0000,141.3622) -- (340.0000,141.3622);
  \path[draw=black,line join=miter,line cap=butt,even odd rule,line width=0.800pt]
    (89.5000,161.3622) -- (340.0000,161.3622);
  \path[draw=black,line join=miter,line cap=butt,even odd rule,line width=0.800pt]
    (60.0000,181.3622) -- (340.0000,181.3622);
  \path[draw=black,line join=miter,line cap=butt,even odd rule,line width=0.800pt]
    (60.0000,201.3622) -- (340.0000,201.3622);
  \path[draw=black,line join=miter,line cap=butt,even odd rule,line width=0.800pt]
    (90.0000,61.3622) -- (90.0000,161.3622);
  \path[draw=black,line join=miter,line cap=butt,even odd rule,line width=0.800pt]
    (310.0000,21.3622) -- (310.0000,121.3622);
  \path[draw=black,line join=miter,line cap=butt,even odd rule,line width=0.800pt]
    (60.0000,81.3622) -- (80.0000,81.3622);
  \path[draw=black,line join=miter,line cap=butt,even odd rule,line width=0.800pt]
    (60.0000,101.3622) -- (72.7022,101.3622) -- (80.0000,101.3622);
  \path[draw=black,line join=miter,line cap=butt,even odd rule,line width=0.800pt]
    (60.0000,121.3622) -- (80.0000,121.3622);
  \path[draw=black,line join=miter,line cap=butt,even odd rule,line width=0.800pt]
    (60.0000,141.3622) -- (80.0000,141.3622);
  \path[draw=black,line join=miter,line cap=butt,even odd rule,line width=0.800pt]
    (320.0000,41.3622) -- (340.0000,41.3622);
  \path[draw=black,line join=miter,line cap=butt,even odd rule,line width=0.800pt]
    (320.0000,61.3622) -- (340.0000,61.3622);
  \path[draw=black,line join=miter,line cap=butt,even odd rule,line width=0.800pt]
    (320.0000,81.3622) -- (340.0000,81.3622);
  \path[draw=black,line join=miter,line cap=butt,even odd rule,line width=0.800pt]
    (320.0000,101.3622) -- (340.0000,101.3622);
  \path[draw=black,line join=miter,miter limit=4.00,line width=0.800pt]
    (60.0000,81.3622)arc(90.000:150.000:19.499706 and
    20.000)arc(150.000:210.000:19.499706 and 20.000)arc(210.000:270.000:19.499706
    and 20.000);
  \path[draw=black,line join=miter,miter limit=4.00,line width=0.800pt]
    (60.0000,181.3622)arc(90.000:150.000:19.499706 and
    20.000)arc(150.000:210.000:19.499706 and 20.000)arc(210.000:270.000:19.499706
    and 20.000);
  \path[draw=black,line join=miter,miter limit=4.00,line width=0.800pt]
    (60.0000,101.3622)arc(90.000:150.000:39.499996 and
    40.000)arc(150.000:210.000:39.499996 and 40.000)arc(210.000:270.000:39.499996
    and 40.000);
  \path[draw=black,line join=miter,miter limit=4.00,line width=0.800pt]
    (60.0000,201.3622)arc(90.000:150.000:39.499996 and
    40.000)arc(150.000:210.000:39.499996 and 40.000)arc(210.000:270.000:39.499996
    and 40.000);
  \path[draw=black,line join=miter,miter limit=4.00,line width=0.800pt]
    (340.0000,101.3622)arc(-90.000:0.000:19.500000 and
    20.000)arc(0.000:90.000:19.500000 and 20.000);
  \path[draw=black,line join=miter,miter limit=4.00,line width=0.800pt]
    (340.0000,81.3621)arc(-90.000:0.000:39.500004 and
    40.000)arc(-0.000:90.000:39.500004 and 40.000);
  \path[draw=black,line join=miter,miter limit=4.00,line width=0.800pt]
    (340.0000,61.3622)arc(-90.000:0.000:59.500004 and
    60.000)arc(0.000:90.000:59.500004 and 60.000);
  \path[draw=black,line join=miter,miter limit=4.00,line width=0.800pt]
    (340.0000,41.3622)arc(270.000:360.000:79.500008 and
    80.000)arc(-0.000:90.000:79.500008 and 80.000);
  \path[fill=blue,line join=miter,line cap=butt,line width=0.800pt]
    (75,77) node[above right] {$a_1$};
  \path[fill=blue,line join=miter,line cap=butt,line width=0.800pt]
    (293,37) node[above right] {$a_2$};
  \path[fill=blue,line join=miter,line cap=butt,line width=0.800pt]
    (184,17) node[above right] {$b_1$};
  \path[fill=blue,line join=miter,line cap=butt,line width=0.800pt]
    (184,117) node[above right] {$b_2$};
  \end{tikzpicture}
  \caption{A bridge presentation of a figure-eight knot.}
  \label{fig:bridge_presentation_eight}
\end{figure}

Consider the plane with this projection as $\{z=0\} \subset \RR^3$ and add to it a point at infinity, so that we may consider it as a subset of a $2$-sphere $S^2 \subset S^3$.
Let us define $\beta_1, \ldots, \beta_g$ as boundaries of some small pairwise non-intersecting tubular neighborhoods of $b_1, \ldots, b_g$ in this sphere.
Now attach to the resulting sphere $g+1$ handles at the endpoints of segments $a_1, \ldots, a_{g+1}$ in such a way that $\beta_1, \ldots, \beta_{g}$ 
encircle attaching discs of handles $a_1, \ldots, a_{g}$ respectively. We imagine these handles as sitting above the plane, i.e., as subsets of $\{z \ge 0\} \subset \RR^3$. By this construction we clearly obtain a genus $g+1$ surface $\gS$.
We define the remaining $\beta_{g+1}$ curve as a meridian of the handle corresponding to the curve $a_{g+1}$.
Finally, define the loops $\alpha_1, \ldots, \alpha_{g+1}$ as curves going along these attached handles and connected at the ends via the remaining parts of $a_1, \ldots, a_{g+1}$, 
respectively.
We arrange all the intersections to be transversal.
This is the \emph{stabilized Heegaard diagram $(\Sigma, \bfga, \bfgb)$} associated to the knot $K$; see Figure~\ref{fig:stabilized_eight}.
\begin{problem}
Show that the stabilized Heegaard diagram $(\Sigma, \bfga, \bfgb)$ constructed above represents $S^3$.
\end{problem}
\begin{figure}
  \centering
\begin{tikzpicture}[y=0.80pt, x=0.80pt, yscale=-1.000000, xscale=1.000000, inner sep=0pt, outer sep=0pt]

  \definecolor{fig52color1}{RGB}{150,0,0}
  \definecolor{fig52color2}{RGB}{0,0,255}
  
  \path[draw=fig52color1,line join=miter,miter limit=4.00,line width=0.800pt]
    (60.0000,91.3622)arc(90.000:150.000:29.500315 and
    30.000)arc(150.000:210.000:29.500315 and 30.000)arc(210.000:270.000:29.500315
    and 30.000);
  \path[draw=fig52color1,line join=miter,miter limit=4.00,line width=0.800pt]
    (340.0000,111.3622)arc(-90.000:0.000:9.500000 and
    9.750)arc(0.000:90.000:9.500000 and 9.750);
  \path[draw=black,line join=miter,miter limit=4.00,line width=0.800pt]
    (90.0000,61.3622) circle (0.1270cm);
  \path[draw=black,line join=miter,miter limit=4.00,line width=0.800pt]
    (90.0000,161.3621) circle (0.1270cm);
  \path[draw=fig52color1,line join=miter,line cap=butt,even odd rule,line
    width=0.800pt] (90.0000,51.3622) -- (340.0000,51.3622);
  \path[draw=fig52color1,line join=miter,line cap=butt,even odd rule,line
    width=0.800pt] (90.0000,71.3622) -- (340.0000,71.3622);
  \path[draw=fig52color1,line join=miter,miter limit=4.00,line width=0.800pt]
    (90.0010,71.3622)arc(90.000:150.000:9.499697 and
    10.000)arc(150.000:210.000:9.499697 and 10.000)arc(210.000:270.000:9.499697
    and 10.000);
  \path[draw=fig52color1,line join=miter,miter limit=4.00,line width=0.800pt]
    (340.0000,51.3622)arc(270.000:360.000:69.500000 and
    69.750)arc(-0.000:90.000:69.500000 and 69.750);
  \path[draw=fig52color1,line join=miter,miter limit=4.00,line width=0.800pt]
    (340.0000,71.3622)arc(270.000:360.000:49.500000 and
    49.750)arc(0.000:90.000:49.500000 and 49.750);
  \path[draw=fig52color1,line join=miter,line cap=butt,even odd rule,line
    width=0.800pt] (60.0000,170.8622) -- (340.0000,170.8622);
  \path[draw=fig52color1,line join=miter,line cap=butt,even odd rule,line
    width=0.800pt] (60.0000,190.8622) -- (340.0000,190.8622);
  \path[draw=fig52color1,line join=miter,miter limit=4.00,line width=0.800pt]
    (60.0008,170.8622)arc(90.000:150.000:9.500000 and
    10.000)arc(150.000:210.000:9.500000 and 10.000)arc(210.000:270.000:9.500000
    and 10.000);
  \path[draw=fig52color1,line join=miter,line cap=butt,even odd rule,line
    width=0.800pt] (60.0000,150.8622) -- (340.0000,150.8622);
  \path[draw=fig52color1,line join=miter,miter limit=4.00,line width=0.800pt]
    (60.0123,190.8622)arc(90.000:150.000:29.512169 and
    30.000)arc(150.000:210.000:29.512169 and 30.000)arc(210.000:270.000:29.512169
    and 30.000);
  \path[draw=fig52color1,line join=miter,line cap=butt,even odd rule,line
    width=0.800pt] (60.0000,130.8622) -- (340.0000,130.8622);
  \path[draw=fig52color1,line join=miter,line cap=butt,even odd rule,line
    width=0.800pt] (60.0000,111.3622) -- (340.0000,111.3622);
  \path[draw=fig52color1,line join=miter,line cap=butt,even odd rule,line
    width=0.800pt] (60.0000,91.3622) -- (340.0000,91.3622);
  \path[draw=fig52color1,line join=miter,miter limit=4.00,line width=0.800pt]
    (340.0000,91.3622)arc(270.000:360.000:29.500000 and
    29.750)arc(0.000:90.000:29.500000 and 29.750);
  \path[draw=fig52color1,line join=miter,line cap=butt,even odd rule,line
    width=0.800pt] (60.0000,11.3622) -- (315.5000,11.3622);
  \path[draw=fig52color1,line join=miter,line cap=butt,even odd rule,line
    width=0.800pt] (60.0000,31.3622) -- (315.6000,31.3622);
  \path[draw=fig52color1,line join=miter,miter limit=4.00,line width=0.800pt]
    (59.9966,111.3622)arc(90.000:150.000:49.496902 and
    50.000)arc(150.000:210.000:49.496902 and 50.000)arc(210.000:270.000:49.496902
    and 50.000);
  \path[draw=fig52color1,line join=miter,miter limit=4.00,line width=0.800pt]
    (315.5000,11.3532)arc(-90.000:0.000:9.500000 and
    10.000)arc(0.000:90.000:9.500000 and 10.000);
  \path[draw=fig52color2,line join=miter,line cap=butt,even odd rule,line
    width=0.800pt] (90.0000,65.8622) -- (90.0000,156.3622);
  \path[draw=black,line join=miter,miter limit=4.00,line width=0.800pt]
    (310.2861,21.2622) circle (0.1270cm);
  \path[draw=black,line join=miter,miter limit=4.00,line width=0.800pt]
    (310.2861,121.2622) circle (0.1270cm);
  \path[draw=fig52color2,line join=miter,line cap=butt,even odd rule,line
    width=0.800pt] (310.2861,25.7623) -- (310.2861,116.2623);
  \path[fill=fig52color2,line join=miter,line cap=butt,line width=0.800pt]
    (291,47.1907) node[above right] (text6083-0-1-5) {\color{blue} $\alpha_2$};
  \path[draw=fig52color1,line join=miter,miter limit=4.00,line width=0.800pt]
    (310.2805,21.2038) circle (0.1976cm);
  \path[fill=fig52color1,line join=miter,line cap=butt,line width=0.800pt]
    (287.7012,27.4912) node[above right] (text6083-0-6) {\color{fig52color1} $\beta_2$};
  \path[fill=fig52color1,line join=miter,line cap=butt,line width=0.800pt]
    (110,69) node[above right] {\color{fig52color1} $\beta_1$};
  \path[fill=blue,line join=miter,line cap=butt,line width=0.800pt]
    (72,87) node[above right] {\color{fig52color2} $\alpha_1$};
\end{tikzpicture}
  \caption{A stabilized diagram $(\Sigma, \bfga, \bfgb)$ associated to a figure-eight knot bridge presentation from Figure \ref{fig:bridge_presentation_eight}. Empty circles at the endpoints of $\alpha_i$ denote the disks where the handles are attached.}
  \label{fig:stabilized_eight}
\end{figure}

For the construction of a chain complex associated to a knot $K$ we need to introduce basepoints.
They are obtained by destabilizing the diagram $(\Sigma, \bfga, \bfgb)$ (cf. Theorem~\ref{thm:handleslide}).
Namely, we forget about the curves $\alpha_{g+1}, \beta_{g+1}$, and remove the handle associated to the curve $a_{g+1}$, defining points $w, z$ as the endpoints of $a_{g+1}$.
This results is a \emph{destabilized Heegaard diagram $(\Sigma, \bfga, \bfgb, w, z)$}, where $\Sigma$ is now a surface of a genus $g$, and $\bfga = \{\alpha_1, \ldots, \alpha_{g}\}$, $\bfgb = \{\beta_1, \ldots, \beta_{g}\}$; see Figure~\ref{fig:destabilized_eight}.

\begin{figure}
  \centering
  \begin{tikzpicture}[y=0.80pt, x=0.80pt, yscale=-1.000000, xscale=1.000000, inner sep=0pt, outer sep=0pt]

  \definecolor{fig53color1}{RGB}{150,0,0}
  \definecolor{fig53color2}{RGB}{0,0,255}
  
  \path[draw=fig53color1,line join=miter,miter limit=4.00,line width=0.800pt]
    (60.0000,91.3622)arc(90.000:150.000:29.500315 and
    30.000)arc(150.000:210.000:29.500315 and 30.000)arc(210.000:270.000:29.500315
    and 30.000);
  \path[draw=fig53color1,line join=miter,miter limit=4.00,line width=0.800pt]
    (340.0000,111.3622)arc(-90.000:0.000:9.500000 and
    9.750)arc(0.000:90.000:9.500000 and 9.750);
  \path[draw=black,line join=miter,miter limit=4.00,line width=0.800pt]
    (90.0000,61.3622) circle (0.1270cm);
  \path[draw=black,line join=miter,miter limit=4.00,line width=0.800pt]
    (90.0000,161.3621) circle (0.1270cm);
  \path[draw=fig53color1,line join=miter,line cap=butt,even odd rule,line
    width=0.800pt] (90.0000,51.3622) -- (340.0000,51.3622);
  \path[draw=fig53color1,line join=miter,line cap=butt,even odd rule,line
    width=0.800pt] (90.0000,71.3622) -- (340.0000,71.3622);
  \path[draw=fig53color1,line join=miter,miter limit=4.00,line width=0.800pt]
    (90.0010,71.3622)arc(90.000:150.000:9.499697 and
    10.000)arc(150.000:210.000:9.499697 and 10.000)arc(210.000:270.000:9.499697
    and 10.000);
  \path[draw=fig53color1,line join=miter,miter limit=4.00,line width=0.800pt]
    (340.0000,51.3622)arc(270.000:360.000:69.500000 and
    69.750)arc(-0.000:90.000:69.500000 and 69.750);
  \path[draw=fig53color1,line join=miter,miter limit=4.00,line width=0.800pt]
    (340.0000,71.3622)arc(270.000:360.000:49.500000 and
    49.750)arc(0.000:90.000:49.500000 and 49.750);
  \path[draw=fig53color1,line join=miter,line cap=butt,even odd rule,line
    width=0.800pt] (60.0000,170.8622) -- (340.0000,170.8622);
  \path[draw=fig53color1,line join=miter,line cap=butt,even odd rule,line
    width=0.800pt] (60.0000,190.8622) -- (340.0000,190.8622);
  \path[draw=fig53color1,line join=miter,miter limit=4.00,line width=0.800pt]
    (60.0008,170.8622)arc(90.000:150.000:9.500000 and
    10.000)arc(150.000:210.000:9.500000 and 10.000)arc(210.000:270.000:9.500000
    and 10.000);
  \path[draw=fig53color1,line join=miter,line cap=butt,even odd rule,line
    width=0.800pt] (60.0000,150.8622) -- (340.0000,150.8622);
  \path[draw=fig53color1,line join=miter,miter limit=4.00,line width=0.800pt]
    (60.0123,190.8622)arc(90.000:150.000:29.512169 and
    30.000)arc(150.000:210.000:29.512169 and 30.000)arc(210.000:270.000:29.512169
    and 30.000);
  \path[draw=fig53color1,line join=miter,line cap=butt,even odd rule,line
    width=0.800pt] (60.0000,130.8622) -- (340.0000,130.8622);
  \path[draw=fig53color1,line join=miter,line cap=butt,even odd rule,line
    width=0.800pt] (60.0000,111.3622) -- (340.0000,111.3622);
  \path[draw=fig53color1,line join=miter,line cap=butt,even odd rule,line
    width=0.800pt] (60.0000,91.3622) -- (340.0000,91.3622);
  \path[draw=fig53color1,line join=miter,miter limit=4.00,line width=0.800pt]
    (340.0000,91.3622)arc(270.000:360.000:29.500000 and
    29.750)arc(0.000:90.000:29.500000 and 29.750);
  \path[draw=fig53color1,line join=miter,line cap=butt,even odd rule,line
    width=0.800pt] (60.0000,11.3622) -- (315.5000,11.3622);
  \path[draw=fig53color1,line join=miter,line cap=butt,even odd rule,line
    width=0.800pt] (60.0000,31.3622) -- (315.6000,31.3622);
  \path[draw=fig53color1,line join=miter,miter limit=4.00,line width=0.800pt]
    (59.9966,111.3622)arc(90.000:150.000:49.496902 and
    50.000)arc(150.000:210.000:49.496902 and 50.000)arc(210.000:270.000:49.496902
    and 50.000);
  \path[draw=fig53color1,line join=miter,miter limit=4.00,line width=0.800pt]
    (315.5000,11.3532)arc(-90.000:0.000:9.500000 and
    10.000)arc(0.000:90.000:9.500000 and 10.000);
  \path[draw=fig53color2,line join=miter,line cap=butt,even odd rule,line
    width=0.800pt] (90.0000,65.8622) -- (90.0000,156.3622);
    
  \path[fill=black,dash pattern=on 1.60pt off 1.60pt,line join=miter,miter
    limit=4.00,even odd rule,line width=1.600pt] (90.0000,71.3622) circle
    (0.0564cm);
  \path[fill=black,dash pattern=on 1.60pt off 1.60pt,line join=miter,miter
    limit=4.00,even odd rule,line width=1.600pt] (90.0000,91.3622) circle
    (0.0564cm);
  \path[fill=black,dash pattern=on 1.60pt off 1.60pt,line join=miter,miter
    limit=4.00,even odd rule,line width=1.600pt] (90.0000,111.3622) circle
    (0.0564cm);
  \path[fill=black,dash pattern=on 1.60pt off 1.60pt,line join=miter,miter
    limit=4.00,even odd rule,line width=1.600pt] (90.0000,130.8622) circle
    (0.0564cm);
  \path[fill=black,dash pattern=on 1.60pt off 1.60pt,line join=miter,miter
    limit=4.00,even odd rule,line width=1.600pt] (90.0000,150.8622) circle
    (0.0564cm);
  \path[fill=black,line join=miter,line cap=butt,line width=0.800pt]
    (304.8851,24.7312) node[above right] {$w$};
  \fill[color=black] (300,22) circle (2);
  \path[fill=black,line join=miter,line cap=butt,line width=0.800pt]
    (304.7799,124.8092) node[above right] {$z$};
  \fill[color=black] (300,122) circle (2);
  \path[fill=red,line join=miter,line cap=butt,line width=0.800pt]
    (110,69) node[above right] {\color{fig53color1} $\beta_1$};
  \path[fill=blue,line join=miter,line cap=butt,line width=0.800pt]
    (72,87) node[above right] {\color{fig53color2} $\alpha_1$};
  \path[fill=black,line join=miter,line cap=butt,line width=0.800pt]
    (96,83) node[above right] {$x_1$};
  \path[fill=black,line join=miter,line cap=butt,line width=0.800pt]
    (96,103) node[above right] {$x_2$};
  \path[fill=black,line join=miter,line cap=butt,line width=0.800pt]
    (96,123) node[above right] {$x_3$};
  \path[fill=black,line join=miter,line cap=butt,line width=0.800pt]
    (96,143) node[above right] {$x_4$};
  \path[fill=black,line join=miter,line cap=butt,line width=0.800pt]
    (96,163) node[above right] {$x_5$};
\end{tikzpicture}
  \caption{A destabilized version of the Heegaard diagram \ref{fig:stabilized_eight}, with the intersection points $\TT_\alpha \cap \TT_\beta$ depicted.}
  \label{fig:destabilized_eight}
\end{figure}
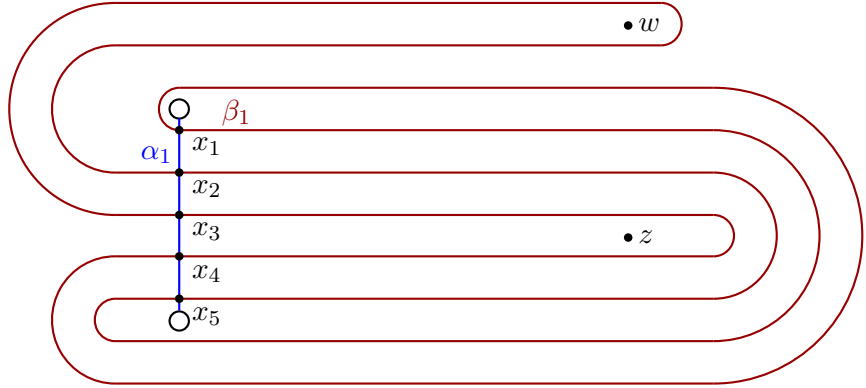

\begin{problem}
At the beginning of Section~\ref{sec:heegaardknots} we described a recipe for obtaining a knot from a doubly pointed
Heegaard diagram and later we sketched
a way to obtain a doubly pointed
Heegaard diagram from a knot. Show that if one starts with an arbitrary knot $K\subset S^3$, passes to a Heegaard diagram and
then recovers a knot $K'$ from the Heegaard diagram, then $K'$ is isotopic to $K$.
\end{problem}

\begin{remark}
For simplicity we described a construction of a doubly pointed Heegaard diagram from a knot in $S^3$. We refer to \cite[Section 2.2]{OzSz-knot}
for a construction of Heegaard diagrams for a knot in an arbitrary three--manifold. 
\end{remark}
\subsection{The hat chain complex associated to a doubly pointed Heegaard diagram}\label{sec:hatknot}

Consider a doubly pointed Heegaard diagram $(\gS,\bfga,\bfgb,z,w)$ representing $(Y,K)$. Let $g=g(\gS)$. We define real $g$--dimensional
tori $\TT_{\ga},\TT_{\gb}\subset\Sym^g(\gS)$ as in Section~\ref{Symmetric_products} above. Moreover, let $R_z,R_w\subset\Sym^{g}(\gS)$ be given by
$(\{z\}\times\gS\times\ldots\times\gS)/S_g$ and $(\{w\}\times\gS\times\ldots\times\gS)/S_g$.

The chain complex $\widehat{CFK}(Y,K)$ is generated by the intersection points $\TT_\ga\cap\TT_\gb$. For any pair
$\bfx,\bfy\in\TT_\ga\cap\TT_\gb$ and $\phi\in\pi_2(\bfx,\bfy)$ we define the \emph{relative Maslov grading}
\[M(\bfx)-M(\bfy)=\mu(\phi)-2n_w(\phi),\]
where $n_w(\phi)$ is the intersection index of $\phi$ and $R_w$. Likewise, we define the \emph{relative Alexander grading}
\begin{equation}\label{eq:A}
A(\bfx)-A(\bfy)=n_z(\phi)-n_w(\phi).
\end{equation}
Various aspects on the Alexander grading are elaborated in \cite[Section 4]{Ras03}.
If $Y=S^3$, there is a way of fixing the Maslov grading, so that it becomes an absolute grading (over $\ZZ$). We refer
to \cite[Section 3.4]{Man}
for more details.
\begin{proposition}[see \expandafter{\cite[Section 1.1]{OzSz-knot}}]\label{prop:alexander_conway_poly}
If $Y=S^3$, then there exists a way of assigning the absolute Alexander grading $A(\bfx)$ in such a way that \eqref{eq:A} holds and moreover
\[\sum_{\bfx\in\TT_\ga\cap\TT_\gb}(-1)^{M(\bfx)}t^{A(\bfx)}=\Delta_K(t),\]
where $\Delta_K(t)$ is the symmetrized Alexander polynomial of the knot $K$.
\end{proposition}

Now we come to a potential source of confusion, because there are two choices of a differential in $\widehat{CFK}(Y,K)$.
We can either set:
\begin{align*}
\partial_{grad} \bfx& = \sum_{\bfy \in \TT_\ga \cap \TT_\gb} \sum_{\substack{\phi \in \pi_2(\bfx, \bfy) \\ n_z(\phi)=n_w(\phi)=0\\ \mu(\phi)=1}} 
\# \widehat{\cM}(\phi) \; \bfy,\\
\intertext{or}
\partial_{fil} \bfx& = \sum_{\bfy \in \TT_\ga \cap \TT_\gb} \sum_{\substack{\phi \in \pi_2(\bfx, \bfy) \\ n_w(\phi)=0\\ \mu(\phi)=1}} 
\# \widehat{\cM}(\phi) \; \bfy.
\end{align*}
\begin{problem}
Prove that $\partial_{grad}$ preserves the Alexander grading, while $\partial_{fil}$ is a filtered map with respect
to the Alexander grading.
\end{problem}
The map $\partial_{fil}$ is the differential in the complex $\widehat{CF}(Y)$, hence 
$\partial_{fil}^2=0$ (by Theorem~\ref{thm:hfhat}), and $\partial_{grad}$ is a part of $\partial_{fil}$ that preserves the Alexander grading,
we have that $\partial_{grad}^2=0$; compare \cite[Section 4.4]{Ras03}.
\begin{problem}
Show also that the homology of $(\widehat{CFK}(Y,K),\partial_{fil})$ is isomorphic to $\widehat{HF}(Y)$.
\end{problem}
\begin{definition}
The homology of the complex $(\widehat{CFK}(Y,K),\partial_{grad})$ is called the \emph{hat knot Floer homology} and denoted $\widehat{HFK}(Y,K)$.
\end{definition}
From the point of view of homological algebra, if we have a filtered complex, like in our case $(\widehat{CFK}(Y,K),\partial_{fil})$, we
can associate with it a graded complex, whose underlying space is isomorphic (at least if the complex is defined over a field),
and with the differential consisting only of the graded part. In our case this is $(\widehat{CFK}(Y,K),\partial_{grad})$. There
is a spectral sequence whose first page is the homology of the graded part, which abuts (under some finiteness assumptions on the complex,
which are satisfied in Heegaard Floer theory) to the homology of the filtered complex. This spectral sequence is used in \cite{OzSz-knot}
to define an important knot invariant, called the $\tau$-invariant. We will not discuss it here.

\begin{theorem}[see \expandafter{\cite[Corollary 3.2]{OzSz-knot}, \cite[Theorem 1]{Ras03}}]

The homology $\widehat{HFK}(Y,K)$ is a knot invariant. Moreover, 

$$\sum_{a}(-1)^{M(a)}t^{A(a)}=\Delta_K(t),$$
where the sum is taken over a graded basis of $\widehat{HFK}(Y,K)$.
\end{theorem}
One of the consequences of the adjunction inequality is the following result; see \cite[Theorem 5.1]{OzSz-knot}.
\begin{theorem}[Adjunction inequality in $\widehat{HFK}(Y,K)$]
Suppose that $K\subset Y$ is a null-homologous knot. Suppose $\sss$ is such that $\widehat{HFK}(Y,K,\sss)\neq 0$. Then for every
Seifert surface $F$ for $K$ of genus $g>0$ we have
\[\left|\langle c_1(\sss),F\rangle\right|\le 2g(F).\]
\end{theorem}

The knot Floer homologies have two wonderful properties. The first one was proved in \cite{OzSz-genus}, the second one is proved in \cite{Gigi,Ni}.
\begin{theorem}\label{thm:giggini} The following two properties hold: \

\begin{itemize}
\item If $K$ is a knot in $S^3$, then $\widehat{HFK}$ detects the three-genus. More precisely, for a knot $K \subset S^3$, 
$$ g_3(K)=\max_a\colon\widehat{HFK}_*(K,a)\neq 0. $$ 
\item $\widehat{HFK}$ detects fibredness. That is, for a null-homologous knot $K$ in a closed, oriented,
connected $3$--manifold, $K$ is fibered if and only if
$$\rk\widehat{HFK}_{*}(K,g_3(K))=1.$$ 
\end{itemize}
Here $\widehat{HFK}_*(K,a)$ denotes the part of $\widehat{HFK}$ with the Alexander grading $a$.
\end{theorem}
\begin{remark}
The fact that $\widehat{HFK}$ detects the three-genus of a knot, can be generalized for null-homologous knots in rational
homology three-spheres, where
the notion of the three-genus is replaced by the Thurston norm; see \cite[Section 1]{OzSz-genus} and \cite{Ni-thurston}.
The fibreness part works for arbitrary null-homologous knots in arbitrary closed three-manifold; see \cite{Ni}. 
\end{remark}

\subsection{The complexes $CFK^-$ and $CFK^\infty$}\label{sec:cfkinf}
The chain complex $CFK^-$ is built in an analogous way, although some subtleties arise. The generators are again
intersection points $\TT_\ga\cap\TT_\gb$, the complex is defined over $\ZZ_2[U]$ and the Maslov and Alexander gradings are as above.
The multiplication by $U$ by definition decreases the Alexander grading by $1$.
The differential is the following
\begin{equation} \label{diff:minus2}
\partial \bfx := \sum_{\bfy \in \TT_\ga \cap \TT_\gb} \sum_{\substack{\phi \in \pi_2(\bfx, \bfy) \\ \mu(\phi)=1}} \# \widehat{\cM}(\phi) \, U^{n_w(\phi)} \, \bfy.
\end{equation}
The only difference with respect to \eqref{diff:minus} is that in the exponent of $U$ we have $n_w$ and not $n_z$. In the sense of 
Section~\ref{sec:hatknot}, the differential should be called $\partial_{fil}$. If we take the graded differential, that is,
the one that does not count discs crossing the first base point (that is, one adds the condition $n_z(\phi)=0$ in the sum in \eqref{diff:minus2}), 
we will get a graded chain complex.
In \cite[Section 3.4]{Man} this complex is denoted $gCFK^-$ and the homology is $HFK^-$. 

Unlike in the hat version, we are not as much interested in the graded complex as in the filtered one, that is, the one with the differential given by \eqref{diff:minus2}. Even though the homologies of complexes $CFK^-(Y,K)$ and $CF^-(Y)$ are the same, there is a substantial difference between $CFK^-(Y,K)$ and $CF^-(Y)$. Namely, in $CFK^-(Y,K)$ we have the Alexander grading. The differential does not necessarily preserve the grading, but as multiplication by $U$ drops the Alexander grading by $1$, we will obtain that the differential never increases the grading.
\begin{problem}
Check that the last statement is true.
\end{problem}
This means that $CFK^-(Y,K)$ is a filtered chain complex over $\ZZ_2[U]$, or a bifiltered chain complex over $\ZZ_2$ (with the other filtration given by powers of $U$, we will explain this in a while). This filtration is independent of the choice of the Heegaard
diagram, in fact we have the following fact; see \cite{OzSz-knot,Ras03}.
\begin{theorem} 
The filtered chain homotopy type of $CFK^-(Y,K)$ is an invariant of the isotopy type of the knot.
\end{theorem}
As it might be expected, the filtered chain homotopy type of $CFK^-(Y,K)$ contains much more information than just the homology of
the chain complex. The famous saying of Andrew Ranicki, one of the inventors of algebraic surgery theory:
\begin{motto}[Ranicki]
\emph{``Chain complexes are good, homologies are bad''}
\end{motto}
\noindent is very true also in Heegaard Floer theory.

As in Section~\ref{section_complexes} above, we can invert formally the variable $U$ to obtain another chain complex, called $CFK^\infty$. 
Here we give a slightly different point of view of this object.

Consider a chain complex whose generators are triples $[\bfx,i,j]$ such that $i,j\in\ZZ$ and $A(\bfx)=j-i$. The triple $[\bfx,i,j]$ will
correspond to the generator $U^{-i}\bfx$. The differential is as in \eqref{diff:minus2}.
\begin{problem}\label{problem:cfkinfty_differential}
Show that with this notation the definition in \eqref{diff:minus2} boils down to 
\begin{equation}\label{eq:boilsdown}
\partial[\bfx,i,j]=\sum_{\bfy \in \TT_\ga \cap \TT_\gb}\sum_{\substack{\phi \in \pi_2(\bfx, \bfy)\\\mu(\phi)=1}} \#\widehat{\cM}(\phi)[\bfy,i-n_w(\phi),j-n_z(\phi)].
\end{equation}
\end{problem}
The chain complex with such a differential is denoted by $CFK^\infty(Y,K)$. The homology is clearly $HFK^\infty(Y,K)\cong HFK^\infty(Y)$.
The chain complex admits an action of $U$, namely $U[\bfx,i,j]=[\bfx,i-1,j-1]$. One of the advantages of \eqref{eq:boilsdown}
over \eqref{diff:minus2} is that the symmetry between the first and the second filtration levels is clearly seen in \eqref{eq:boilsdown}.
This symmetry is a generalization of the symmetry of the Alexander polynomial of a knot.
\begin{problem}
Prove that the subcomplex $CFK^\infty(Y,K)\{i\le 0\}$ is
the chain complex $CFK^-(Y,K)$.
\end{problem}
\begin{remark}
Sometimes one considers $CFK^-=CFK^\infty(Y,K)\{i<0\}$, instead of $CFK^\infty(Y,K)\{i\le 0\}$. 
This does not affect the isomorphism type of the relatively graded complex.
\end{remark}
The definition of $CFK^\infty$ via $[\bfx,i,j]$ allows us to present it graphically. Namely, for any element $[\bfx,i,j]$ we
can put a dot in a plane with coordinates $(i,j)$. The arrows denote differentials, often one draws only
an edge, the direction of an arrow can be determined by the fact that the differential does not increase any of the two filtration levels. 
The Maslov grading is usually
not presented, or denoted near the dots, if necessary.

One of the features of the chain complex $CFK^\infty$ is its behavior under connected sums, which is an analogue of Problem~\ref{prob:connected}.
\begin{proposition}\label{prop:connectedknots}
Suppose $K_1,K_2$ are two knots in $S^3$. Then 
$$ CFK^\infty(K_1\# K_2)\cong CFK^\infty(K_1)\otimes CFK^\infty(K_2) $$ 
where ``$\cong$'' denotes a bifiltered chain homotopy equivalence. The tensor product is taken over the ring $\ZZ_2[U,U^{-1}]$.
\end{proposition}
\begin{problem}
Take two knots $K_1$ and $K_2$. Draw a knot diagram for $K_1$ and $K_2$ and connect them by a band to obtain a knot diagram
for $K_1\#K_2$; try to control the bridge presentation.
Using Section~\ref{sec:heegaardknots} calculate $CFK^\infty(K_1\# K_2)$ and prove as much as you can of Proposition~\ref{prop:connectedknots}
(existence of maps, gradings, filtrations, etc).
\end{problem}
\begin{example}
	Let us revisit the example of a figure-eight knot. The underlying surface of the Heegaard diagram $(\Sigma, \bfga, \bfgb, w, z)$ 
(see Figure~\ref{fig:destabilized_eight}) is of genus $1$, thus its universal cover is $\CC$. Therefore, by combining this fact with the Riemann mapping theorem, we get that if there exists a topological disk $\phi \in \pi_2(\bfx, \bfy)$, then it is uniquely represented by a holomorphic disk. Using this fact it is straightforward to find all holomorphic disks as in a Figure~\ref{fig:eight_heegaard_disks}.
	
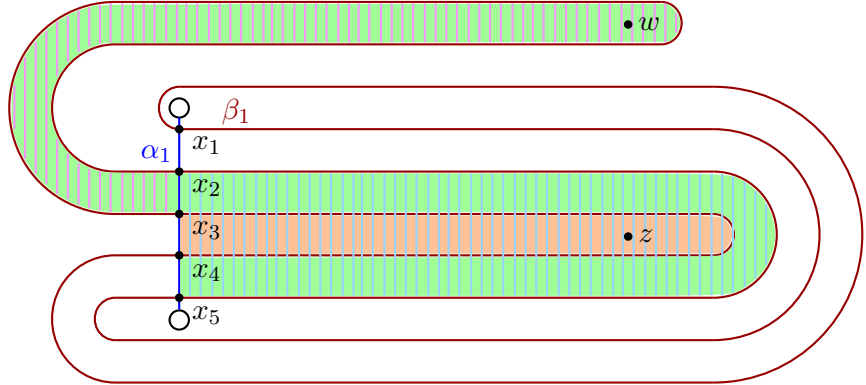
\begin{figure}
\begin{tikzpicture}[y=0.80pt, x=0.80pt, yscale=-1.000000, xscale=1.000000, inner sep=0pt, outer sep=0pt]

  \definecolor{fig54color1}{RGB}{150,0,0}
  \definecolor{fig54color2}{RGB}{0,0,255}
  \definecolor{color1}{RGB}{246,150,255}
  \definecolor{color3}{RGB}{255,194,150}
  \definecolor{color2}{RGB}{159,255,150}
  \definecolor{color4}{RGB}{150,211,255}
  
  \path[draw=fig54color1,line join=miter,miter limit=4.00,line width=0.800pt]
    (60.0000,91.3622)arc(90.000:150.000:29.500315 and
    30.000)arc(150.000:210.000:29.500315 and 30.000)arc(210.000:270.000:29.500315
    and 30.000);
  \path[draw=fig54color1,line join=miter,miter limit=4.00,line width=0.800pt]
    (340.0000,111.3622)arc(-90.000:0.000:9.500000 and
    9.750)arc(0.000:90.000:9.500000 and 9.750);
  \path[draw=black,line join=miter,miter limit=4.00,line width=0.800pt]
    (90.0000,61.3622) circle (0.1270cm);
  \path[draw=black,line join=miter,miter limit=4.00,line width=0.800pt]
    (90.0000,161.3621) circle (0.1270cm);
  \path[draw=fig54color1,line join=miter,line cap=butt,even odd rule,line
    width=0.800pt] (90.0000,51.3622) -- (340.0000,51.3622);
  \path[draw=fig54color1,line join=miter,line cap=butt,even odd rule,line
    width=0.800pt] (90.0000,71.3622) -- (340.0000,71.3622);
  \path[draw=fig54color1,line join=miter,miter limit=4.00,line width=0.800pt]
    (90.0010,71.3622)arc(90.000:150.000:9.499697 and
    10.000)arc(150.000:210.000:9.499697 and 10.000)arc(210.000:270.000:9.499697
    and 10.000);
  \path[draw=fig54color1,line join=miter,miter limit=4.00,line width=0.800pt]
    (340.0000,51.3622)arc(270.000:360.000:69.500000 and
    69.750)arc(-0.000:90.000:69.500000 and 69.750);
  \path[draw=fig54color1,line join=miter,miter limit=4.00,line width=0.800pt]
    (340.0000,71.3622)arc(270.000:360.000:49.500000 and
    49.750)arc(0.000:90.000:49.500000 and 49.750);
  \path[draw=fig54color1,line join=miter,line cap=butt,even odd rule,line
    width=0.800pt] (60.0000,170.8622) -- (340.0000,170.8622);
  \path[draw=fig54color1,line join=miter,line cap=butt,even odd rule,line
    width=0.800pt] (60.0000,190.8622) -- (340.0000,190.8622);
  \path[draw=fig54color1,line join=miter,miter limit=4.00,line width=0.800pt]
    (60.0008,170.8622)arc(90.000:150.000:9.500000 and
    10.000)arc(150.000:210.000:9.500000 and 10.000)arc(210.000:270.000:9.500000
    and 10.000);
  \path[draw=fig54color1,line join=miter,line cap=butt,even odd rule,line
    width=0.800pt] (60.0000,150.8622) -- (340.0000,150.8622);
  \path[draw=fig54color1,line join=miter,miter limit=4.00,line width=0.800pt]
    (60.0123,190.8622)arc(90.000:150.000:29.512169 and
    30.000)arc(150.000:210.000:29.512169 and 30.000)arc(210.000:270.000:29.512169
    and 30.000);
  \path[draw=fig54color1,line join=miter,line cap=butt,even odd rule,line
    width=0.800pt] (60.0000,130.8622) -- (340.0000,130.8622);
  \path[draw=fig54color1,line join=miter,line cap=butt,even odd rule,line
    width=0.800pt] (60.0000,111.3622) -- (340.0000,111.3622);
  \path[draw=fig54color1,line join=miter,line cap=butt,even odd rule,line
    width=0.800pt] (60.0000,91.3622) -- (340.0000,91.3622);
  \path[draw=fig54color1,line join=miter,miter limit=4.00,line width=0.800pt]
    (340.0000,91.3622)arc(270.000:360.000:29.500000 and
    29.750)arc(0.000:90.000:29.500000 and 29.750);
  \path[draw=fig54color1,line join=miter,line cap=butt,even odd rule,line
    width=0.800pt] (60.0000,11.3622) -- (315.5000,11.3622);
  \path[draw=fig54color1,line join=miter,line cap=butt,even odd rule,line
    width=0.800pt] (60.0000,31.3622) -- (315.6000,31.3622);
  \path[draw=fig54color1,line join=miter,miter limit=4.00,line width=0.800pt]
    (59.9966,111.3622)arc(90.000:150.000:49.496902 and
    50.000)arc(150.000:210.000:49.496902 and 50.000)arc(210.000:270.000:49.496902
    and 50.000);
  \path[draw=fig54color1,line join=miter,miter limit=4.00,line width=0.800pt]
    (315.5000,11.3532)arc(-90.000:0.000:9.500000 and
    10.000)arc(0.000:90.000:9.500000 and 10.000);
  \path[draw=blue,line join=miter,line cap=butt,even odd rule,line
    width=0.800pt] (90.0000,65.8622) -- (90.0000,156.3622);
  \path[fill=color2,line join=miter,line cap=butt,line width=0.800pt]
    (90.5613,140.9530) -- (90.5613,131.9200) -- (216.4031,131.7664) .. controls
    (339.3414,131.6163) and (342.2973,131.5868) .. (344.5110,130.4882) .. controls
    (347.7061,128.9026) and (350.1192,124.9353) .. (350.1192,121.2680) .. controls
    (350.1192,115.9745) and (347.0600,112.1339) .. (341.8951,110.9433) .. controls
    (340.2039,110.5535) and (296.9102,110.3468) .. (214.9449,110.3374) --
    (90.5613,110.3231) -- (90.5613,101.2824) -- (90.5613,92.2416) --
    (217.2780,92.2485) .. controls (355.2654,92.2560) and (346.5503,92.0466) ..
    (353.4160,95.5203) .. controls (355.1532,96.3992) and (358.0902,98.6200) ..
    (359.9427,100.4554) .. controls (370.6502,111.0637) and (371.6346,127.5468) ..
    (362.2650,139.3409) .. controls (360.0101,142.1793) and (354.9571,146.1405) ..
    (351.9064,147.4613) .. controls (345.8796,150.0706) and (351.1711,149.9717) ..
    (217.2780,149.9790) -- (90.5613,149.9860) -- (90.5613,140.9530) -- cycle;
  \path[fill=color2,line join=miter,line cap=butt,line width=0.800pt]
    (55.5647,110.2454) .. controls (35.7444,108.5716) and (19.2370,94.8463) ..
    (13.2170,75.0350) .. controls (11.7937,70.3510) and (11.6219,68.8744) ..
    (11.6219,61.3280) .. controls (11.6219,53.8204) and (11.7975,52.2935) ..
    (13.1857,47.7333) .. controls (14.0459,44.9078) and (15.4389,41.2332) ..
    (16.2814,39.5674) .. controls (23.1487,25.9892) and (35.6757,16.2601) ..
    (50.2768,13.1648) .. controls (53.4835,12.4851) and (75.6204,12.3514) ..
    (186.3627,12.3429) -- (318.6190,12.3327) -- (320.9482,14.4371) .. controls
    (325.1877,18.2672) and (325.4048,23.5665) .. (321.4941,27.7591) --
    (319.2894,30.1227) -- (186.9896,30.4143) -- (54.6898,30.7060) --
    (50.8110,32.0237) .. controls (29.3243,39.3237) and (22.7710,66.7596) ..
    (38.6497,82.9384) .. controls (42.2444,86.6011) and (45.5599,88.7152) ..
    (50.6534,90.5925) .. controls (53.8538,91.7720) and (55.3974,91.8826) ..
    (71.7506,92.1046) -- (89.3947,92.3441) -- (89.3947,101.3144) --
    (89.3947,110.2848) -- (74.6670,110.4498) .. controls (66.5667,110.5406) and
    (57.9707,110.4488) .. (55.5647,110.2454) -- cycle;
  \path[fill=color3,line join=miter,line cap=butt,line width=0.800pt]
    (90.5613,121.1177) -- (90.5613,112.0730) -- (216.1115,112.0739) .. controls
    (334.4768,112.0748) and (341.7779,112.1332) .. (343.6916,113.0947) .. controls
    (346.0617,114.2855) and (348.1863,117.2117) .. (348.7109,120.0079) .. controls
    (349.1710,122.4604) and (347.1014,126.9070) .. (344.6812,128.6661) --
    (343.0344,129.8629) -- (216.7978,130.0127) -- (90.5613,130.1625) --
    (90.5613,121.1177) -- cycle;
  \path[draw=color4,line join=miter,line cap=butt,even odd rule,line
    width=0.800pt] (95.0000,92.3622) -- (95.0000,149.8622);
  \path[draw=color4,line join=miter,line cap=butt,even odd rule,line
    width=0.800pt] (100.0000,92.3622) -- (100.0000,149.8622);
  \path[draw=color4,line join=miter,line cap=butt,even odd rule,line
    width=0.800pt] (105.5000,92.3622) -- (105.5000,149.8622);
  \path[draw=color4,line join=miter,line cap=butt,even odd rule,line
    width=0.800pt] (110.5000,92.3622) -- (110.5000,149.8622);
  \path[draw=color4,line join=miter,line cap=butt,even odd rule,line
    width=0.800pt] (115.5000,92.3622) -- (115.5000,149.8622);
  \path[draw=color4,line join=miter,line cap=butt,even odd rule,line
    width=0.800pt] (120.5000,92.3622) -- (120.5000,149.8622);
  \path[draw=color4,line join=miter,line cap=butt,even odd rule,line
    width=0.800pt] (126.0000,92.3622) -- (126.0000,149.8622);
  \path[draw=color4,line join=miter,line cap=butt,even odd rule,line
    width=0.800pt] (131.0000,92.3622) -- (131.0000,149.8622);
  \path[draw=color4,line join=miter,line cap=butt,even odd rule,line
    width=0.800pt] (135.5000,92.3622) -- (135.5000,149.8622);
  \path[draw=color4,line join=miter,line cap=butt,even odd rule,line
    width=0.800pt] (140.5000,92.3622) -- (140.5000,149.8622);
  \path[draw=color4,line join=miter,line cap=butt,even odd rule,line
    width=0.800pt] (146.0000,92.3622) -- (146.0000,149.8622);
  \path[draw=color4,line join=miter,line cap=butt,even odd rule,line
    width=0.800pt] (151.0000,92.3622) -- (151.0000,149.8622);
  \path[draw=color4,line join=miter,line cap=butt,even odd rule,line
    width=0.800pt] (156.0000,92.3622) -- (156.0000,149.8622);
  \path[draw=color4,line join=miter,line cap=butt,even odd rule,line
    width=0.800pt] (161.0000,92.3622) -- (161.0000,149.8622);
  \path[draw=color4,line join=miter,line cap=butt,even odd rule,line
    width=0.800pt] (166.5000,92.3622) -- (166.5000,149.8622);
  \path[draw=color4,line join=miter,line cap=butt,even odd rule,line
    width=0.800pt] (171.5000,92.3622) -- (171.5000,149.8622);
  \path[draw=color4,line join=miter,line cap=butt,even odd rule,line
    width=0.800pt] (175.5000,92.3622) -- (175.5000,149.8622);
  \path[draw=color4,line join=miter,line cap=butt,even odd rule,line
    width=0.800pt] (180.5000,92.3622) -- (180.5000,149.8622);
  \path[draw=color4,line join=miter,line cap=butt,even odd rule,line
    width=0.800pt] (186.0000,92.3622) -- (186.0000,149.8622);
  \path[draw=color4,line join=miter,line cap=butt,even odd rule,line
    width=0.800pt] (191.0000,92.3622) -- (191.0000,149.8622);
  \path[draw=color4,line join=miter,line cap=butt,even odd rule,line
    width=0.800pt] (196.0000,92.3622) -- (196.0000,149.8622);
  \path[draw=color4,line join=miter,line cap=butt,even odd rule,line
    width=0.800pt] (201.0000,92.3622) -- (201.0000,149.8622);
  \path[draw=color4,line join=miter,line cap=butt,even odd rule,line
    width=0.800pt] (206.5000,92.3622) -- (206.5000,149.8622);
  \path[draw=color4,line join=miter,line cap=butt,even odd rule,line
    width=0.800pt] (211.5000,92.3622) -- (211.5000,149.8622);
  \path[draw=color4,line join=miter,line cap=butt,even odd rule,line
    width=0.800pt] (216.0000,92.3622) -- (216.0000,149.8622);
  \path[draw=color4,line join=miter,line cap=butt,even odd rule,line
    width=0.800pt] (221.0000,92.3622) -- (221.0000,149.8622);
  \path[draw=color4,line join=miter,line cap=butt,even odd rule,line
    width=0.800pt] (226.5000,92.3622) -- (226.5000,149.8622);
  \path[draw=color4,line join=miter,line cap=butt,even odd rule,line
    width=0.800pt] (231.5000,92.3622) -- (231.5000,149.8622);
  \path[draw=color4,line join=miter,line cap=butt,even odd rule,line
    width=0.800pt] (236.5000,92.3622) -- (236.5000,149.8622);
  \path[draw=color4,line join=miter,line cap=butt,even odd rule,line
    width=0.800pt] (241.5000,92.3622) -- (241.5000,149.8622);
  \path[draw=color4,line join=miter,line cap=butt,even odd rule,line
    width=0.800pt] (247.0000,92.3622) -- (247.0000,149.8622);
  \path[draw=color4,line join=miter,line cap=butt,even odd rule,line
    width=0.800pt] (252.0000,92.3622) -- (252.0000,149.8622);
  \path[draw=color4,line join=miter,line cap=butt,even odd rule,line
    width=0.800pt] (257.0000,92.3622) -- (257.0000,149.8622);
  \path[draw=color4,line join=miter,line cap=butt,even odd rule,line
    width=0.800pt] (262.0000,92.3622) -- (262.0000,149.8622);
  \path[draw=color4,line join=miter,line cap=butt,even odd rule,line
    width=0.800pt] (267.5000,92.3622) -- (267.5000,149.8622);
  \path[draw=color4,line join=miter,line cap=butt,even odd rule,line
    width=0.800pt] (272.5000,92.3622) -- (272.5000,149.8622);
  \path[draw=color4,line join=miter,line cap=butt,even odd rule,line
    width=0.800pt] (277.5000,92.3622) -- (277.5000,149.8622);
  \path[draw=color4,line join=miter,line cap=butt,even odd rule,line
    width=0.800pt] (282.5000,92.3622) -- (282.5000,149.8622);
  \path[draw=color4,line join=miter,line cap=butt,even odd rule,line
    width=0.800pt] (288.0000,92.3622) -- (288.0000,149.8622);
  \path[draw=color4,line join=miter,line cap=butt,even odd rule,line
    width=0.800pt] (293.0000,92.3622) -- (293.0000,149.8622);
  \path[draw=color4,line join=miter,line cap=butt,even odd rule,line
    width=0.800pt] (297.5000,92.3622) -- (297.5000,149.8622);
  \path[draw=color4,line join=miter,line cap=butt,even odd rule,line
    width=0.800pt] (302.5000,92.3622) -- (302.5000,149.8622);
  \path[draw=color4,line join=miter,line cap=butt,even odd rule,line
    width=0.800pt] (308.0000,92.3622) -- (308.0000,149.8622);
  \path[draw=color4,line join=miter,line cap=butt,even odd rule,line
    width=0.800pt] (313.0000,92.3622) -- (313.0000,149.8622);
  \path[draw=color4,line join=miter,line cap=butt,even odd rule,line
    width=0.800pt] (318.0000,92.3622) -- (318.0000,149.8622);
  \path[draw=color4,line join=miter,line cap=butt,even odd rule,line
    width=0.800pt] (323.0000,92.3622) -- (323.0000,149.8622);
  \path[draw=color4,line join=miter,line cap=butt,even odd rule,line
    width=0.800pt] (328.5000,92.3622) -- (328.5000,149.8622);
  \path[draw=color4,line join=miter,line cap=butt,even odd rule,line
    width=0.800pt] (333.5000,92.3622) -- (333.5000,149.8622);
  \path[draw=color4,line join=miter,line cap=butt,even odd rule,line
    width=0.800pt] (339.0000,92.3622) -- (339.0000,149.8622);
  \path[draw=color4,line join=miter,line cap=butt,even odd rule,line
    width=0.800pt] (344.0000,92.3622) -- (344.0000,149.8622);
  \path[draw=color4,line join=miter,line cap=butt,even odd rule,line
    width=0.800pt] (349.0000,93.6622) -- (349.0000,148.3622);
  \path[draw=color4,line join=miter,line cap=butt,even odd rule,line
    width=0.800pt] (354.5000,96.3622) -- (354.5000,145.8622);
  \path[draw=color4,line join=miter,line cap=butt,even odd rule,line
    width=0.800pt] (359.5000,100.4262) -- (359.5000,141.8622);
  \path[draw=color4,line join=miter,line cap=butt,even odd rule,line
    width=0.800pt] (364.5000,106.7622) -- (364.5000,135.3662);
  \path[draw=color1,line join=miter,line cap=butt,even odd rule,line
    width=0.800pt] (52.9000,12.7622) -- (52.9000,31.0622);
  \path[draw=color1,line join=miter,line cap=butt,even odd rule,line
    width=0.800pt] (57.9000,12.0622) -- (57.9000,30.3622);
  \path[draw=color1,line join=miter,line cap=butt,even odd rule,line
    width=0.800pt] (63.4000,12.0622) -- (63.4000,30.3622);
  \path[draw=color1,line join=miter,line cap=butt,even odd rule,line
    width=0.800pt] (68.4000,12.0622) -- (68.4000,30.3622);
  \path[draw=color1,line join=miter,line cap=butt,even odd rule,line
    width=0.800pt] (73.4000,12.0622) -- (73.4000,30.3622);
  \path[draw=color1,line join=miter,line cap=butt,even odd rule,line
    width=0.800pt] (78.4000,12.0622) -- (78.4000,30.3622);
  \path[draw=color1,line join=miter,line cap=butt,even odd rule,line
    width=0.800pt] (83.9000,12.0622) -- (83.9000,30.3622);
  \path[draw=color1,line join=miter,line cap=butt,even odd rule,line
    width=0.800pt] (88.9000,12.0622) -- (88.9000,30.3622);
  \path[draw=color1,line join=miter,line cap=butt,even odd rule,line
    width=0.800pt] (93.4000,12.0622) -- (93.4000,30.3622);
  \path[draw=color1,line join=miter,line cap=butt,even odd rule,line
    width=0.800pt] (98.4000,12.0622) -- (98.4000,30.3622);
  \path[draw=color1,line join=miter,line cap=butt,even odd rule,line
    width=0.800pt] (103.9000,12.0622) -- (103.9000,30.3622);
  \path[draw=color1,line join=miter,line cap=butt,even odd rule,line
    width=0.800pt] (108.9000,12.0622) -- (108.9000,30.3622);
  \path[draw=color1,line join=miter,line cap=butt,even odd rule,line
    width=0.800pt] (113.9000,12.0622) -- (113.9000,30.3622);
  \path[draw=color1,line join=miter,line cap=butt,even odd rule,line
    width=0.800pt] (118.9000,12.0622) -- (118.9000,30.3622);
  \path[draw=color1,line join=miter,line cap=butt,even odd rule,line
    width=0.800pt] (124.4000,12.0622) -- (124.4000,30.3622);
  \path[draw=color1,line join=miter,line cap=butt,even odd rule,line
    width=0.800pt] (129.4000,12.0622) -- (129.4000,30.3622);
  \path[draw=color1,line join=miter,line cap=butt,even odd rule,line
    width=0.800pt] (133.4000,12.0622) -- (133.4000,30.3622);
  \path[draw=color1,line join=miter,line cap=butt,even odd rule,line
    width=0.800pt] (138.4000,12.0622) -- (138.4000,30.3622);
  \path[draw=color1,line join=miter,line cap=butt,even odd rule,line
    width=0.800pt] (143.9000,12.0622) -- (143.9000,30.3622);
  \path[draw=color1,line join=miter,line cap=butt,even odd rule,line
    width=0.800pt] (148.9000,12.0622) -- (148.9000,30.3622);
  \path[draw=color1,line join=miter,line cap=butt,even odd rule,line
    width=0.800pt] (153.9000,12.0622) -- (153.9000,30.3622);
  \path[draw=color1,line join=miter,line cap=butt,even odd rule,line
    width=0.800pt] (158.9000,12.0622) -- (158.9000,30.3622);
  \path[draw=color1,line join=miter,line cap=butt,even odd rule,line
    width=0.800pt] (164.4000,12.0622) -- (164.4000,30.3622);
  \path[draw=color1,line join=miter,line cap=butt,even odd rule,line
    width=0.800pt] (169.4000,12.0622) -- (169.4000,30.3622);
  \path[draw=color1,line join=miter,line cap=butt,even odd rule,line
    width=0.800pt] (173.9000,12.0622) -- (173.9000,30.3622);
  \path[draw=color1,line join=miter,line cap=butt,even odd rule,line
    width=0.800pt] (178.9000,12.0622) -- (178.9000,30.3622);
  \path[draw=color1,line join=miter,line cap=butt,even odd rule,line
    width=0.800pt] (184.4000,12.0622) -- (184.4000,30.3622);
  \path[draw=color1,line join=miter,line cap=butt,even odd rule,line
    width=0.800pt] (189.4000,12.0622) -- (189.4000,30.3622);
  \path[draw=color1,line join=miter,line cap=butt,even odd rule,line
    width=0.800pt] (194.4000,12.0622) -- (194.4000,30.3622);
  \path[draw=color1,line join=miter,line cap=butt,even odd rule,line
    width=0.800pt] (199.4000,12.0622) -- (199.4000,30.3622);
  \path[draw=color1,line join=miter,line cap=butt,even odd rule,line
    width=0.800pt] (204.9000,12.0622) -- (204.9000,30.3622);
  \path[draw=color1,line join=miter,line cap=butt,even odd rule,line
    width=0.800pt] (209.9000,12.0622) -- (209.9000,30.3622);
  \path[draw=color1,line join=miter,line cap=butt,even odd rule,line
    width=0.800pt] (214.9000,12.0622) -- (214.9000,30.3622);
  \path[draw=color1,line join=miter,line cap=butt,even odd rule,line
    width=0.800pt] (219.9000,12.0622) -- (219.9000,30.3622);
  \path[draw=color1,line join=miter,line cap=butt,even odd rule,line
    width=0.800pt] (225.4000,12.0622) -- (225.4000,30.3622);
  \path[draw=color1,line join=miter,line cap=butt,even odd rule,line
    width=0.800pt] (230.4000,12.0622) -- (230.4000,30.3622);
  \path[draw=color1,line join=miter,line cap=butt,even odd rule,line
    width=0.800pt] (235.4000,12.0622) -- (235.4000,30.3622);
  \path[draw=color1,line join=miter,line cap=butt,even odd rule,line
    width=0.800pt] (240.4000,12.0622) -- (240.4000,30.3622);
  \path[draw=color1,line join=miter,line cap=butt,even odd rule,line
    width=0.800pt] (245.9000,12.0622) -- (245.9000,30.3622);
  \path[draw=color1,line join=miter,line cap=butt,even odd rule,line
    width=0.800pt] (250.9000,12.0622) -- (250.9000,30.3622);
  \path[draw=color1,line join=miter,line cap=butt,even odd rule,line
    width=0.800pt] (255.4000,12.0622) -- (255.4000,30.3622);
  \path[draw=color1,line join=miter,line cap=butt,even odd rule,line
    width=0.800pt] (260.4000,12.0622) -- (260.4000,30.3622);
  \path[draw=color1,line join=miter,line cap=butt,even odd rule,line
    width=0.800pt] (265.9000,12.0622) -- (265.9000,30.3622);
  \path[draw=color1,line join=miter,line cap=butt,even odd rule,line
    width=0.800pt] (270.9000,12.0622) -- (270.9000,30.3622);
  \path[draw=color1,line join=miter,line cap=butt,even odd rule,line
    width=0.800pt] (275.9000,12.0622) -- (275.9000,30.3622);
  \path[draw=color1,line join=miter,line cap=butt,even odd rule,line
    width=0.800pt] (280.9000,12.0622) -- (280.9000,30.3622);
  \path[draw=color1,line join=miter,line cap=butt,even odd rule,line
    width=0.800pt] (286.4000,12.0622) -- (286.4000,30.3622);
  \path[draw=color1,line join=miter,line cap=butt,even odd rule,line
    width=0.800pt] (291.4000,12.0622) -- (291.4000,30.3622);
  \path[draw=color1,line join=miter,line cap=butt,even odd rule,line
    width=0.800pt] (296.9000,12.0622) -- (296.9000,30.3622);
  \path[draw=color1,line join=miter,line cap=butt,even odd rule,line
    width=0.800pt] (301.9000,12.0622) -- (301.9000,30.3622);
  \path[draw=color1,line join=miter,line cap=butt,even odd rule,line
    width=0.800pt] (306.9000,12.0622) -- (306.9000,30.3622);
  \path[draw=color1,line join=miter,line cap=butt,even odd rule,line
    width=0.800pt] (312.4000,12.0622) -- (312.4000,30.3622);
  \path[draw=color1,line join=miter,line cap=butt,even odd rule,line
    width=0.800pt] (317.3999,12.0622) -- (317.3999,30.3622);
  \path[draw=color1,line join=miter,line cap=butt,even odd rule,line
    width=0.800pt] (322.4000,16.6452) -- (322.4000,25.7487);
  \path[draw=color1,line join=miter,line cap=butt,even odd rule,line
    width=0.800pt] (47.9000,13.9382) -- (47.9000,32.2382);
  \path[draw=color1,line join=miter,line cap=butt,even odd rule,line
    width=0.800pt] (42.9000,15.5152) -- (42.9000,35.0152);
  \path[draw=color1,line join=miter,line cap=butt,even odd rule,line
    width=0.800pt] (37.9000,18.0312) -- (37.9000,39.5312);
  \path[draw=color1,line join=miter,line cap=butt,even odd rule,line
    width=0.800pt] (32.9000,20.5662) -- (32.9000,46.3662);
  \path[draw=color1,line join=miter,line cap=butt,even odd rule,line
    width=0.800pt] (27.9000,24.9932) -- (27.9000,97.1932);
  \path[draw=color1,line join=miter,line cap=butt,even odd rule,line
    width=0.800pt] (22.9000,30.2102) -- (22.9000,92.4102);
  \path[draw=color1,line join=miter,line cap=butt,even odd rule,line
    width=0.800pt] (17.9000,37.3422) -- (17.9000,85.4422);
  \path[draw=color1,line join=miter,line cap=butt,even odd rule,line
    width=0.800pt] (12.9000,50.7612) -- (12.9000,71.2612);
  \path[draw=color1,line join=miter,line cap=butt,even odd rule,line
    width=0.800pt] (52.9000,109.6622) -- (52.9000,91.3622);
  \path[draw=color1,line join=miter,line cap=butt,even odd rule,line
    width=0.800pt] (57.9000,110.3622) -- (57.9000,92.0622);
  \path[draw=color1,line join=miter,line cap=butt,even odd rule,line
    width=0.800pt] (63.4000,110.3622) -- (63.4000,92.0622);
  \path[draw=color1,line join=miter,line cap=butt,even odd rule,line
    width=0.800pt] (68.4000,110.3622) -- (68.4000,92.0622);
  \path[draw=color1,line join=miter,line cap=butt,even odd rule,line
    width=0.800pt] (73.4000,110.3622) -- (73.4000,92.0622);
  \path[draw=color1,line join=miter,line cap=butt,even odd rule,line
    width=0.800pt] (78.4000,110.3622) -- (78.4000,92.0622);
  \path[draw=color1,line join=miter,line cap=butt,even odd rule,line
    width=0.800pt] (83.9000,110.3622) -- (83.9000,92.0622);
  \path[draw=color1,line join=miter,line cap=butt,even odd rule,line
    width=0.800pt] (47.9000,108.4862) -- (47.9000,90.1862);
  \path[draw=color1,line join=miter,line cap=butt,even odd rule,line
    width=0.800pt] (42.9000,106.9092) -- (42.9000,87.4092);
  \path[draw=color1,line join=miter,line cap=butt,even odd rule,line
    width=0.800pt] (37.9000,104.3932) -- (37.9000,82.8932);
  \path[draw=color1,line join=miter,line cap=butt,even odd rule,line
    width=0.800pt] (32.9000,101.8582) -- (32.9000,76.0582);
   
  \path[fill=black,dash pattern=on 1.60pt off 1.60pt,line join=miter,miter
    limit=4.00,even odd rule,line width=1.600pt] (90.0000,71.3622) circle
    (0.0564cm);
  \path[fill=black,dash pattern=on 1.60pt off 1.60pt,line join=miter,miter
    limit=4.00,even odd rule,line width=1.600pt] (90.0000,91.3622) circle
    (0.0564cm);
  \path[fill=black,dash pattern=on 1.60pt off 1.60pt,line join=miter,miter
    limit=4.00,even odd rule,line width=1.600pt] (90.0000,111.3622) circle
    (0.0564cm);
  \path[fill=black,dash pattern=on 1.60pt off 1.60pt,line join=miter,miter
    limit=4.00,even odd rule,line width=1.600pt] (90.0000,130.8622) circle
    (0.0564cm);
  \path[fill=black,dash pattern=on 1.60pt off 1.60pt,line join=miter,miter
    limit=4.00,even odd rule,line width=1.600pt] (90.0000,150.8622) circle
    (0.0564cm);
  \path[fill=black,line join=miter,line cap=butt,line width=0.800pt]
    (304.8851,24.7312) node[above right] {$w$};
  \fill[color=black] (300,22) circle(2);
  \fill[color=black] (300,122) circle (2);
  \path[fill=black,line join=miter,line cap=butt,line width=0.800pt]
    (304.7799,124.8092) node[above right] {$z$};
  \path[fill=fig54color1,line join=miter,line cap=butt,line width=0.800pt]
    (110,69) node[above right] {\color{fig54color1} $\beta_1$};
  \path[fill=blue,line join=miter,line cap=butt,line width=0.800pt]
    (72,87) node[above right] {\color{fig54color2} $\alpha_1$};
  \path[fill=black,line join=miter,line cap=butt,line width=0.800pt]
    (96,83) node[above right] {$x_1$};
  \path[fill=black,line join=miter,line cap=butt,line width=0.800pt]
    (96,103) node[above right] {$x_2$};
  \path[fill=black,line join=miter,line cap=butt,line width=0.800pt]
    (96,123) node[above right] {$x_3$};
  \path[fill=black,line join=miter,line cap=butt,line width=0.800pt]
    (96,143) node[above right] {$x_4$};
  \path[fill=black,line join=miter,line cap=butt,line width=0.800pt]
    (96,163) node[above right] {$x_5$};
\end{tikzpicture}
  \caption{Holomorphic disks in a Heegaard diagram \ref{fig:destabilized_eight}. Two disks at infinity that connect pairs $(x_2, x_1)$ and $(x_4,x_1)$, and pass through points $z$ and $w$ respectively, are not shown.}
  \label{fig:eight_heegaard_disks}
\end{figure}
	
	From the same diagram we may find some of the relative Alexander gradings according to \eqref{eq:A}. This, together with Proposition~\ref{prop:alexander_conway_poly}, gives us a way to determine the absolute Alexander gradings $A(x_1), A(x_3), A(x_5) = 0, A(x_2) = 1, A(x_4) = -1$.
	We can also read off the differentials from Figure~\ref{fig:eight_heegaard_disks}, according to Problem \ref{problem:cfkinfty_differential}; the nontrivial ones are 
$\partial x_2 =x_1 + x_5$, $\partial x_4= U x_1 + Ux_5$, $\partial x_3=Ux_2+x_4$. In the chain complex $CFK^\infty$ this means that
$\partial[x_2, i, i + 1] = [x_1, i, i] + [x_5, i, i], \partial[x_4, i, i - 1] = [x_1, i - 1, i - 1] + [x_5, i - 1, i - 1], \partial[x_3, i, i] = [x_2, i - 1, i] + [x_4, i, i -1]$ for $i \in \ZZ$. For convenience let us change variables, setting $x_1' := x_1 + x_5$.
	
	The complex $CFK^\infty(S^3, 4_1)$, spanned by the elements $[x_1', i, i]$ and $[x_k, i, i + A(x_k)]$, where $k=2,\ldots,5$, is depicted in Figure~\ref{fig:CFKinfty4_1}.
	
\begin{figure}
\centering
\begin{tikzpicture}
\draw [step=1,dotted,color=black!40] (-2.5,-2.5) grid (2.5,2.5);
\draw[thin,->] (-2.75,0) -- (2.75,0);
\draw[thin, ->] (0,-2.75) -- (0,2.75);
\path (0 cm, 2.75 cm) node[above left] {$j$};
\path (2.75 cm, 0 cm) node[below right] {$i$};

\fill[black] (0 cm, 0 cm) circle (0.075);
\path (0 cm, 0 cm) node[below right] {$x_3$};
\fill[black] (-2 cm, -2 cm) circle (0.075);
\path (-2 cm, -2 cm) node[below right] {$Ux_3$};
\fill[black] (2 cm, 2 cm) circle (0.075);
\path (2 cm, 2 cm) node[below right] {$U^{-1}x_3$};

\fill[black] (-0.25 cm, -0.25 cm) circle (0.075);
\path (-0.25 cm, -0.25 cm) node[below left] {$x_5$};
\fill[black] (-2.25 cm, -2.25 cm) circle (0.075);
\path (-2.25 cm, -2.25 cm) node[below left] {$Ux_5$};
\fill[black] (1.75 cm, 1.75 cm) circle (0.075);
\path (1.75 cm, 1.75 cm) node[below left] {$U^{-1}x_5$};

\fill[black] (0.25 cm, 0.25 cm) circle (0.075);
\path (0.25 cm, 0.25 cm) node[above right] {$x_1'$};
\fill[black] (-1.75 cm, -1.75 cm) circle (0.075);
\path (-1.75 cm, -1.75 cm) node[above right] {$Ux_1'$};
\fill[black] (2.25 cm, 2.25 cm) circle (0.075);
\path (2.25 cm, 2.25 cm) node[above right] {$U^{-1}x_1'$};


\fill[black] (0.25 cm, 2 cm) circle (0.075);
\path (0.25 cm, 2 cm) node[above] {$x_2$};
\fill[black] (-1.75 cm, 0 cm) circle (0.075);
\path (-1.75 cm, 0 cm) node[above] {$Ux_2$};

\fill[black] (2 cm, 0.25 cm) circle (0.075);
\path (2 cm, 0.25 cm) node[right] {$U^{-1}x_4$};
\fill[black] (0 cm, -1.75 cm) circle (0.075);
\path (0 cm, -1.75 cm) node[right] {$x_4$};

\draw[very thick]  (0.25 cm, 2 cm) -- (0.25 cm , 0.25 cm);
\draw[very thick]  (2 cm, 0.25 cm) -- (0.25 cm , 0.25 cm);
\draw[very thick]  (2 cm, 2 cm) -- (2 cm , 0.25 cm);
\draw[very thick]  (2 cm, 2 cm) -- (0.25 cm , 2 cm);

\draw[very thick]  (0 cm , 0 cm) -- (0 cm, -1.75 cm);
\draw[very thick]  (0 cm , 0 cm) -- (-1.75 cm, -0 cm);
\draw[very thick]  (-1.75 cm , 0 cm) -- (-1.75 cm, -1.75 cm);
\draw[very thick]  (0 cm , -1.75 cm) -- (-1.75 cm, -1.75 cm);

\draw[very thick]  (-2 cm , -2 cm) -- (-2 cm, -2.5 cm);
\draw[very thick]  (-2 cm , -2 cm) -- (-2.5 cm, -2 cm);
\draw[very thick]  (2.25 cm , 2.25 cm) -- (2.25 cm, 2.50 cm);
\draw[very thick]  (2.25 cm , 2.25 cm) -- (2.50 cm, 2.25 cm);
\end{tikzpicture}
  \caption{A complex representing $CFK^\infty(S^3, 4_1)$. Note that the elements $U^{-1}x_1',U^{-1}x_3, U^{-1}x_5$ are in the same bifiltration level $(i,j) = (1,1)$, likewise their images under the endomorphism $U$.}
  \label{fig:CFKinfty4_1}
\end{figure}
\end{example}

\begin{example}\label{ex:sumoftrefoils}
Similarly, one can compute a complex $CFK^\infty(S^3, 3_1)$ and then use the K\"unneth formula (see Proposition~\ref{prop:connectedknots}) to obtain a full complex $CFK^\infty(S^3, 3_1 \# 3_1)$ of the connected sum of two copies of trefoils. Figure~\ref{fig:CFKtrefoils}, after tensoring with $\ZZ_2[U, U^{-1}]$, presents the result after a change of basis.
\begin{figure}
\centering
\begin{tikzpicture}
\draw [step=1,dotted,color=black!40] (-0.5,-0.5) grid (2.5,2.5);
\draw[thin,->] (-0.75,0) -- (2.75,0);
\draw[thin, ->] (0,-0.75) -- (0,2.75);
\path (0 cm, 2.75 cm) node[above left] {$j$};
\path (2.75 cm, 0 cm) node[below right] {$i$};

\fill[black] (2 cm, 2 cm) circle (0.075);
\fill[black] (1.25 cm, 2 cm) circle (0.075);
\fill[black] (2 cm, 1.25 cm) circle (0.075);
\fill[black] (1.25 cm, 1.25 cm) circle (0.075);

\draw[very thick]  (2 cm, 2 cm) -- (1.25 cm , 2 cm);
\draw[very thick]  (2 cm, 2 cm) -- (2 cm , 1.25 cm);
\draw[very thick]  (1.25 cm , 2 cm) -- (1.25 cm, 1.25 cm);
\draw[very thick]  (2 cm , 1.25 cm) -- (1.25 cm, 1.25 cm);

\fill[black] (0 cm, 2 cm) circle (0.075);
\fill[black] (1 cm, 2 cm) circle (0.075);
\fill[black] (1 cm, 1 cm) circle (0.075);
\fill[black] (2 cm, 1 cm) circle (0.075);
\fill[black] (2 cm, 0 cm) circle (0.075);

\draw[very thick]  (1 cm, 2 cm) -- (0 cm , 2 cm);
\draw[very thick]  (1 cm, 2 cm) -- (1 cm , 1 cm);
\draw[very thick]  (2 cm, 1 cm) -- (1 cm , 1 cm);
\draw[very thick]  (2 cm, 1 cm) -- (2 cm , 0 cm);
\end{tikzpicture}
  \caption{Tensoring this complex with $\ZZ_2[U, U^{-1}]$ results in the complex $CFK^\infty(S^3, 3_1 \# 3_1)$.}
  \label{fig:CFKtrefoils}
\end{figure}\end{example}
\begin{problem}
Calculate Example~\ref{ex:sumoftrefoils} by yourself.
\end{problem}
Even though the homology of $CFK^\infty(Y,K)$ is not very interesting, the bifiltered chain homotopy type of the complex contains a lot of
information about the knot. An important example of a piece of information contained in the chain complex $CFK^\infty(Y,K)$ that is lost when passing
to homology is given below.

\subsection{The $V_m$ invariants}\label{sec:Vm}

Let $K\subset S^3$ be a knot. For any $m\in \ZZ$ let $CFK^\infty(i<0,j<m)$ be the subcomplex of $CFK^\infty$ generated by elements at bifiltration
level $(i,j)$, where $i<0$ and $j<m$. Let $A_m^+$ be the quotient complex $CFK^\infty/CFK^\infty(i<0,j<m)$.
\begin{remark}
Sometimes one writes that $A_m^+$ is a complex generated by elements at filtration level $(i,j)$, where $i\ge 0$ or $j\ge m$, and if a differential
of an element leads out of $A_m^+$ we set it to be zero. This might be
sometimes convenient but is not very rigorous, because it suggests that $A_m^+$ is a subcomplex of $CFK^\infty$, while it is not. If an element
$x\in CFK^\infty$ is at filtration level $i\ge 0$ or $j\ge m$, and $\partial x=y$ with $y\in CFK^\infty(i<0,j<m)$, then $\partial x=0$ in $A_m^+$
by defintion.
\end{remark}
\begin{definition}
The $V_m$ invariant of a knot $K$ is minus one half of 
the minimal grading of a cycle $x\in A_m^+$, which is non-trivial in homology and such that for any $k\ge 0$ there
exists $y_k\in A_m^+$ such that $U^ky_k=x$.
\end{definition}
\begin{remark}
The notation $V_m$ for these invariants is taken from \cite{NiWu}. In the original source, that is, Rasmussen's thesis \cite{Ras03}, a related invariant $h_k$ was studied.
\end{remark}
\begin{problem}
Find a relation between $V_m$ and the invariant $h_k$ defined in Section 7.2 of the Rasmussen's thesis.
\end{problem}
\begin{problem}
Notice that $V_m\le V_{m-1}$. Prove that $V_{m-1}\le V_m+1$.
\end{problem}
\begin{problem}
Calculate $V_m$ for the sum of two trefoils and for the figure-eight knot. Observe that the `squares' in both chain complexes
do not contribute to $V_m$.
\end{problem}
\begin{proposition}\label{prop:vmisinvariant}
The number $V_m$ is a concordance invariant.
\end{proposition}
A proof using the large surgery formula is given as Problem~\ref{probl:concordance}. The original proof of \cite{Ras03} uses a different approach.

\subsection{Large integer surgeries}
There is a general way for calculating $HF^+$ (and so the $d$--invariants) of surgeries on a knot in $S^3$, 
see for instance \cite{OzSz-integer}, once the
chain complex $CFK^\infty(K)$ is known. Notice that knowing only $HFK^-$ or $\widehat{HFK}$ is usually not enough; recall Ranicki's motto.   
The general formula simplifies a lot, when the surgery coefficient is a large positive integer. Before we begin,
we need to show a useful way of enumerating \spinc{} structures on surgeries on a knot in $S^3$. The following result can be found in 
\cite[Lemma 7.10]{OzSz-absolute}.

\begin{proposition}\label{prop:unique}
Let $q>0$ be an integer and consider a knot $K\subset S^3$. Let $Y=S^3_q(K)$ and let $W$ be a four--dimensional handlebody obtained by gluing
a two--handle to the ball $B^4$ along a product neighborhood of $K$ with framing $q$, so that $\partial W=Y$. Let $F\subset W$
be a closed surface obtained by capping a Seifert surface for $K$ by the core of the two--handle.

For any integer $m\in[-q/2,q/2)$ there exists a unique \spinc{} structure $\sss_m$ on $Y$ characterized by the fact that it extends
to a \spinc{} structure $\sst_m$ on $W$ with the property that $\langle c_1(\sst_m),F\rangle+q=2m$.
\end{proposition}
\begin{problem}\
\begin{itemize}
\item Prove that the definition of $\sss_m$ does not depend on the choice of the Seifert surface used to construct $F$.
\item Explain the action of $H^2(Y;\ZZ)$ on the set of the \spinc{} structures under the identification in Proposition~\ref{prop:unique}.
\end{itemize}
\end{problem}

Now we are ready to state the Large Surgery Theorem.
\begin{theorem}[see \expandafter{\cite[Theorem 4.4]{OzSz-knot}}]\label{thm:lst}
Suppose that $K\subset S^3$ and $q\ge 2g_3(K)-1$. Then for any \spinc{} structure $\sss_m$ (with $m\in[-q/2,q/2)\cap\ZZ$), we have
an isomorphism between $A_m^+$ and $HF^+(S^3_q(K),\sss_m)$. The isomorphism changes the Maslov grading by $\frac{(q-2m)^2-q}{4q}$. In particular, we have $d(S^3_q(K),\sss_m)=\frac{(q-2m)^2-q}{4q}-2V_m(K)$.
\end{theorem}
As a corollary we give a proof of the concordance invariance of $V_m$. Suppose $K$ is concordant to $K'$. Let $m\in\ZZ$ and choose
a sufficiently large integer $q$, in particular we require that 
$q\ge \mathrm{max}\{2g_3(K)-1,2g_3(K')-1,2|m|+1\}$. The $d$--invariants of $q$--surgery on $K$ and $K'$ are given by Theorem~\ref{thm:lst}, therefore 
the invariance of $V_m$ under a concordance follows from the following fact.
\begin{lemma}\label{lem:concordance}
Suppose $K$ is concordant to $K'$ and $q>0$. Then there exists a four--manifold $W$ whose boundary is $S^3_q(K')\sqcup -S^3_q(K)$
and such that the inclusions $S^3_q(K)\hookrightarrow W$, $S^3_q(K')\hookrightarrow W$ induce isomorphisms on $\ZZ$ homology. Moreover,
for any integer $m\in[-q/2,q/2)$ there exists a \spinc{} structure $\sst_m$ on $W$ extending the \spinc{} structures $\sss_m$ on both
sides of the boundary.
\end{lemma}
\begin{problem}
Consider the following construction. Let $A\subset S^3\times[0,1]$ be a concordance between $K$ and $K'$. Glue a two--handle to $S^3\times[0,1]$
along a product neighborhood of $K'\subset S^3\times\{1\}$ with framing $q$. Denote by $W'$ the resulting four--manifold. Let $P\subset W'$ 
be the union of $A$ and the core of the two handle and let $N$ be a product neighborhood of $P$ in $W'$. Show that $W=\overline{W'\setminus N}$
has all the properties stated in Lemma~\ref{lem:concordance}. See \cite{BoGo} for a generalization of this construction.
\end{problem}
\begin{problem}\label{probl:concordance}
Conclude the proof of concordance invariance of $V_m$.
\end{problem}

\begin{problem}
Let $x_1,\ldots,x_n$ be all the chains of $CFK^\infty(K)$, which are cycles and which are at grading $0$. Prove that
\[V_m(K)=\min_{k=1,\ldots,n}\max(i(x_k),j(x_k)-m),\]
where $i(x),j(x)$ denote the $i$--th and the $j$--th bifiltration levels as described in Section~\ref{sec:cfkinf} above.
\end{problem}
\begin{problem}
Show by means of an example, that $V_m$ are in general not additive, that is, $V_m(K\# K')$ is not always equal to $V_m(K)+V_m(K')$.
\end{problem}
\begin{problem}
Show that for all $k,m\in\ZZ$ $V_m(K\#K')\le V_k(K)+V_{m-k}(K')$.
\end{problem}

\subsection{L--space knots}\label{sec:Lspace}
We will now introduce a class of knots for which the chain complex $CFK^\infty$ is especially easy to describe.

\begin{definition}
A knot $K\subset S^3$ is called an \emph{L--space} knot (sometimes called a positive L--space knot), if there exists a coefficient $q>0$
such that $S^3_q(K)$ is an L--space.
\end{definition}
The notion of an L--space knot was introduced in \cite{OzSz-lspace} in the context of the Berge conjecture, which predicts the list of all possible knots in $S^3$ such that a surgery on these knots with some coefficient gives a lens space. The notion of an L--space knot turns out to be very useful also for studying singularities of plane curves.
\begin{example}
By the result of Moser \cite[Proposition 3.2]{Mos}, if $|pqr - s| = 1$, then the $s/r$--surgery on a positive torus knot $T(p,q)$ is the lens space $L(|s|, rq^2)$. Therefore, every positive torus knot is an L--space knot.
\end{example}
We have the following properties of L--space knots.
\begin{lemma}\label{lem:propertiesofLspace}\
\begin{itemize}
\setlength{\itemsep}{1ex}
\item[(a)] L--space knots are prime. A connected sum of two non-trivial knots is never an L--space knot (see \cite{Krc}).
\item[(b)] If $K$ is an L--space knot, then $S^3_q(K)$ is an L--space if and only if $q\ge 2g_3(K)-1$ (see \cite[Proposition 9.6]{OzSz11} and \cite{Hom0}).
\item[(c)] L--space knots are quasipositive (see \cite{hedden1}).
\item[(d)] L--space knots are fibered.
\item[(e)] For an L--space knot $K$ we have $g_3(K)=g_4(K)$ (see \cite{OzSz-lspace} and \cite{hedden1}).
\item[(f)] For any $i\in\ZZ$ we have $\rk\widehat{HFK}(K,i)\le 1$ (see \cite{OzSz-lspace}).
\end{itemize}
\end{lemma}
\begin{remark}
Fiberedness of a knot admitting a lens space surgery was known to experts before the Heegaard Floer times, \cite{OzSz-lspace} contains an explicit proof.
The proof for general L-space knots follows from the explicit description of the fact that $\rk\widehat{HFK}(K,i)\le 1$ together with the result of \cite{Gigi,Ni} 
(Theorem~\ref{thm:giggini} of the present article).
\end{remark}
\begin{problem}
Prove that the $\tau$ invariant (see \cite{OS-fourball}) of an $L$--space knot is equal to its three--genus. Notice that this proves (e). Refer to a result of
Hedden \cite{hedden1} to prove (c).
\end{problem}
\begin{remark}
It is easy to find a positive knot which is not an L--space knot: take the connected sum of two trefoils. 
There are positive knots (even fibred positive knots) which are not even concordant to a connected sum of any number of L--space knots; see \cite{BH,FK}.
\end{remark}
We will now present an algorithm for describing the $CFK^\infty$ complex of an L--space knot based on the Alexander polynomial. The
algorithm was first described by Peters \cite{Pts}, nowadays it is widely used.

Suppose $K$ is an L--space knot of genus $g$. Let $\Delta$ be the Alexander polynomial for $K$, which we normalize in such a way that 
$\Delta(t^{-1})=\Delta(t)$. It was showed in \cite{OzSz-lspace} that $\Delta$ has the following form.
\[\Delta(t)=t^{n_0}-t^{n_1}+\ldots-t^{n_{2k-1}}+t^{n_{2k}},\]
where $n_0>n_1>\ldots>n_{2k}$ and $n_0=-n_{2k}=g$. Set 
\begin{equation*}
	\left\{
	\begin{aligned}
		& m_0 = 0 \\
		& m_{2i-1} = m_{2i} \qquad 1 \leq i \leq k \\
		& m_{2i+1} = m_{2i} + (n_{2i} - n_{2i+1}) \qquad 0 \leq i \leq k - 1
	\end{aligned} \right.
\end{equation*}

\begin{problem}
Show that $m_{2k}=g$.
\end{problem}

We will construct now an abstract chain complex over $\ZZ_2$ from the numbers $n_i$ and $m_i$. The chain complex will be graded and doubly
filtered. The construction is as follows.

For any $i=0,\ldots,k$ we place a generator $x_i$ with (Maslov) grading $0$ at bifiltration level $(m_{2k-2i},m_{2i})$
(in the notation of Section \ref{sec:cfkinf} it is $[x_i,m_{2k-2i},m_{2i}]$). We set $\partial x_i=0$.
For any $i=0,\ldots,k-1$ we place a generator $y_i$ with (Maslov) grading $1$ at bifiltration level $(m_{2k-2i-1},m_{2i+1})$
(that is, $[y_i,m_{2k-2i-1},m_{2i+1}]$). We set 
$\partial y_i=x_i+x_{i+1}$.
\begin{example}
For a torus knot $T(3,4)$ we have $\Delta=t^3-t^2+1-t^{-2}+t^{-3}$ so $m_0=0$, $m_1=1$, $m_2=1$, $m_3=3$, $m_4=3$. The $x$--generators
are at bifiltration levels $(0,3)$, $(1,1)$ and $(3,0)$, while the $y$--generators are at bifiltration level $(1,3)$ and $(3,1)$;
see Figure~\ref{fig:T34}.
\begin{figure}
\centering
\begin{tikzpicture}
\draw [step=1,dotted,color=black!40] (-0.5,-0.5) grid (3.5,3.5);
\draw[thin,->] (-0.75,0) -- (3.75,0);
\draw[thin, ->] (0,-0.75) -- (0,3.75);
\path (0 cm, 3.75 cm) node[above left] {$j$};
\path (3.75 cm, 0 cm) node[below right] {$i$};

\fill[black] (0 cm, 3 cm) circle (0.075);
\path (0 cm, 3 cm) node[above left] {$x_2$};

\fill[black] (1 cm, 3 cm) circle (0.075);
\path (1 cm, 3 cm) node[above left] {$y_1$};

\fill[black] (1 cm, 1 cm) circle (0.075);
\path (1 cm, 1 cm) node[above left] {$x_1$};

\fill[black] (3 cm, 1 cm) circle (0.075);
\path (3 cm, 1 cm) node[above left] {$y_0$};

\fill[black] (3 cm, 0 cm) circle (0.075);
\path (3 cm, 0 cm) node[above left] {$x_0$};

\draw[very thick]  (1 cm , 3 cm) -- (0 cm, 3 cm);
\draw[very thick]  (1 cm , 3 cm) -- (1 cm , 1 cm);
\draw[very thick]  (3 cm, 1 cm) -- (3 cm , 0 cm);
\draw[very thick]  (3 cm , 1 cm) -- (1 cm , 1 cm);

\end{tikzpicture}
  \caption{A staircase complex of a torus knot $T(3,4)$.}
  \label{fig:T34}
\end{figure}
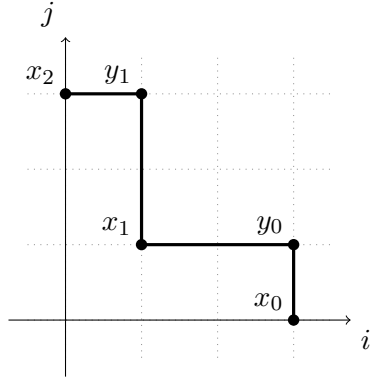
\end{example}
\begin{definition}
The chain complex obtained in this way is called the \emph{staircase complex} associated with an L--space knot $K$ and it is denoted
$St(K)$.
\end{definition}
The staircase complex will now be tensored by $\ZZ_2[U,U^{-1}]$, where $U$ is a formal variable. 
We write $St(K)\otimes_{\ZZ_2}\ZZ_2[U,U^{-1}]$ for the product.
It is generated by elements $U^jx_i$ and $U^jy_i$, $j\in\ZZ$. The grading and the filtration levels are defined by requiring that multiplication
by $U$ changes the (Maslov) grading by $-2$ and each of the filtration levels by $-1$, exactly as the action of $U$ on the knot Floer
chain complexes. The following result was described in a paper by Peters \cite{Pts} (see also \cite{NiWu}), but the idea that
the Alexander polynomial determines the complex $CFK^\infty$ can be traced back to \cite{OzSz-lspace}.
\begin{proposition}\label{prop:staircasecomplex}
Let $K$ be an L--space knot.
The chain complex $St(K)\otimes_{\ZZ_2}\ZZ_2[U,U^{-1}]$ is bifiltered chain homotopy equivalent to $CFK^\infty(K)$.
\end{proposition}

\begin{problem}\label{prob:sumconvolution}
Prove that if $K$ and $K'$ are L--space knots, then we have $V_m(K\#K')=\min_{k\in\ZZ} (V_k(K)+V_{m-k}(K'))$. 
Show that the same holds if $K$ and $K'$ are connected sums of L--space knots.
\end{problem}
\section{Cuspidal singularities}\label{sec:cuspidal}
The scenery changes for a while. We need to recall a few facts from singularity theory.

\subsection{Links of singular points}

Consider a complex curve $C$ in some connected, open set $\Omega\subset\CC^2$. Suppose $C$ is defined as a zero set $F^{-1}(0)$, where $F\colon\Omega\to\CC$ is a holomorphic function. We will assume that $F$ is \emph{reduced}, which might be interpreted as requiring that the gradient of $F$ does not vanish identically on any open subset of $C$.

\begin{problem}
The rigorous definition of `reduced' reads that $F$ is not divisible (in the ring of holomorphic functions $\cO(\Omega)$) by any square of
a non-invertible element. Prove that the two definitions are equivalent. 
\end{problem}

\begin{definition}\label{def:singularpoint}
A point $z\in C$ is called \emph{singular} if $\nabla F(z)=0$.
\end{definition}

\begin{problem}
Prove that if $z\in C$ is a singular point and $F$ is reduced, then $z$ is \emph{isolated}, that is, there is no sequence 
$z_n\in C \setminus \{z\}$ of singular points converging to $z$.
\end{problem}

By Tougeron's theorem, see \cite[Section 2.1]{Zo}, 
any isolated singular point is finitely presented. That is, for each singular point $z$ there is a local analytic
change of coordinates, which transforms $C$ to a $F_{fin}^{-1}(0)$, where $F_{fin}$ is a Taylor expansion of $F$ at $z$ of sufficiently high
order (the original Tougeron theorem says that the order equal to the Milnor number plus one will do, but for some specific singularities
a lower order expansion may be sufficient).

Let $z\in C$ be a singular point. Take a ball $B\subset\Omega$ with center $z$ of sufficiently small radius.

\begin{definition}
The intersection $\partial B\cap C\subset\partial B$ is called the \emph{link of singularity}.
\end{definition}

\begin{problem}
Prove that the isotopy type of the link of singularity is independent of the radius of the curve, once the starting curve is sufficiently small.

Hint. The distance function to the singular point is a Morse function when restricted to $C$. Try showing that the restriction
has no critical points on $C$ near $z$, except for $z$ itself. See also \cite{Milnor-cob}.
\end{problem} 
\begin{problem}
Prove that $C\cap B$ is homeomorphic to the cone over the link $C\cap\partial B$.
\end{problem}

\begin{definition}
The \emph{number of branches} of $C$ at the singular point is the number of connected components of $B\cap C\setminus\{z\}$. A singular point
is called \emph{cuspidal} if $C$ has precisely one branch.
\end{definition}
Two singular points $(C,z)$ and $(C',z')$ are analytically equivalent if there exists a biholomorphic map of neighborhoods of $z$ and $z'$
in $\CC^2$, which takes locally $C$ to $C'$. In general, analytic equivalence is a surprisingly complicated notion.
There is a coarser equivalence, which proves very useful.
\begin{definition}
Two singular points $(C,z)$ and $(C',z')$ are called \emph{topologically equivalent} if there exist small balls $B,B'\subset\CC^2$
with centers $z$ and $z'$ and a homeomorphism $h\colon B\to B'$ that takes $C\cap B$ to $C'\cap B'$.
\end{definition}
\begin{problem}
Show that two singular points are topologically equivalent if and only if their links are isotopic.
\end{problem}
The two notions of equivalence give rise to notions of analytic and topological invariants of singular points. These
are quantities associated with a singular points which are preserved under an analytic (respectively: topological) equivalence. The distinction
can be 
quite subtle. For example, the Milnor number $\mu=\dim_{\CC}\cO_z/(\frac{\partial F}{\partial x},\frac{\partial F}{\partial y})$ (here
$\cO_z$ is the local ring and we consider its quotient over by an ideal generated by $\frac{\partial F}{\partial x}$ and 
$\frac{\partial F}{\partial y}$)
is a topological invariant. For a cuspidal singularity $\mu$ is equal to twice the genus of the link and a slightly more complicated formula
calculates the Milnor number from the genera of the components of the link and the linking numbers of the components; see 
\cite[Section 10]{Milnor-sing}.

On the other hand, the Tjurina number, $\tau=\dim_{\CC}\cO_z/(F,\frac{\partial F}{\partial x},\frac{\partial F}{\partial y})$, whose
definition looks very similar to $\mu$, is \emph{not} a topological invariant; see \cite[Section I.1.2]{GLS}.
\begin{problem}
Show that if $F$ is quasihomogeneous, then $\tau=\mu$.
\end{problem}
\begin{problem}
Play around with some examples of $F$ using your favorite computer algebra system (\texttt{sage}, \texttt{macaulay}, \texttt{singular}) and find examples
of singularities which have the same topological type but different Tjurina numbers. 

Hint. Take $F=x^p-y^q$ with $p,q$ coprime
and try adding to it terms of weighted degree greater than $pq$, where $x$ has degree $q$ and $y$ has degree $p$.
\end{problem}

To conclude the section we list a few different objects related to a singular point that have (almost) the same meaning.
\begin{itemize}
\item The Milnor number $\mu$ defined as above. By the celebrated Milnor's theorem, the map $z\mapsto F(z)/|F(z)|$ from 
$\partial B\setminus (C\cap\partial B)$ to $S^1$ is a locally trivial fibration, whose fiber has homotopy type of a wedge of $\mu$ copies
of $S^1$.
\item The $\delta$--invariant, whose original definition is algebraic; see \cite[Section I.3.4]{GLS}. For a singular point with $r$ branches we have
that $2\delta=\mu+r-1$, a formula proved by Milnor in \cite[Section 10]{Milnor-sing}.
\item The genus of the link $g_3(C\cap\partial B)$ is equal to half the Milnor number if the link has one branch. By Kronheimer--Mrowka's
result, the three--genus is also equal to the smooth four--genus of the link.
\end{itemize}
\begin{problem}
Establish an explicit relation between $g_3(C\cap\partial B)$ and the $\delta$--invariant for a singular point with arbitrarily many branches.
The algebraic definition of the $\delta$--invariant is given in \cite[Section I.3.4]{GLS} or in \cite[Section 10]{Milnor-sing}.
\end{problem}
\subsection{Topological classification of cuspidal singular points}
For completeness we recall a topological classification of cuspidal singular points. For us it is convenient to write the classification
in terms of a so-called \emph{characteristic sequence}. A characteristic sequence is a finite sequence of numbers $(p;q_1,q_2,\ldots,q_m)$
with $p>1$, $p<q_1<\ldots<q_m$.
These numbers satisfy the following relation. Set $r_0=p$, $r_{i+1}=\gcd(r_i,q_{i+1})$. We require that the sequence $r_i$ be strictly
decreasing and $r_m=1$. To each characteristic sequence we can associate a model singular point on a curve, which is locally parametrized
as
\begin{align*} x(t)&=t^p\\ y(t)&=t^{q_1}+t^{q_2}+\ldots+t^{q_m}.\end{align*}
\begin{theorem}
The characteristic sequence is a complete invariant of the topological type of cuspidal singular points. That is, any cuspidal singular point
is topologically equivalent to precisely one model singularity.
\end{theorem}
The number $m$ is called the \emph{length} of the characteristic sequence. There are several alternative ways of encoding a characteristic
sequence. For example, there are so--called Newton pairs, and Puiseux or characteristic pairs (both Newton pairs and Puiseux pairs
might have slightly different meaning), which are sequences of pairs of integers. The quantity $m$ is also the length of such a sequence
and so we will often refer to $m$ as the number of Puiseux pairs. Note however, that the multiplicity sequence, see \cite{BK}, might be much
longer than the characteristic sequence.

The isotopy class of the link of singularity is also an invariant of topological type, and in fact, it is also a complete invariant. 
There is an explicit algorithm for determining the link from the characteristic sequence; see
\cite{EN}.
We record one basic example for future reference.
\begin{example}
If $m=1$, the characteristic sequence is $(p;q)$ for some coprime integers with $0<p<q$. The link of singularity
is the torus knot $T(p,q)$.
\end{example}

\subsection{Semigroup of a singular point}
Let $(C,z)$ be a singular point of a plane curve. For any complex polynomial $G$, which does not vanish on any of the components of $C$ containing $z$
(component in the analytic sense), we can define the local intersection index $C\cdot_z G^{-1}(0)$. 
\begin{example}
Suppose $z$ is cuspidal. By the Puiseux theorem there exists a local parametrization $t\mapsto (x(t),y(t))$ of $C$ near $z$, such that
$z=(x(0),y(0))$. Then the local intersection index is the order at $t=0$ of the map $t\mapsto G(x(t),y(t))$.
\end{example}
\begin{problem}
Suppose $z=(0,0)$ and $C=\{F\equiv 0\}$ with $F=x^p-y^q$. Show that a number $l\ge 0$ can be obtained as $C\cdot_z G^{-1}(0)$ if and only if
$l$ can be presented as $ip+jq$, where $i,j\ge 0$ are integers. For $l=ip+jq$ write explicitly a polynomial $G$ such that 
$C\cdot_z G^{-1}(0)=l$.
\end{problem}
We have the following notion.
\begin{definition}
The \emph{semigroup} of a singular point $S(z)$ is a semigroup of $\ZZ_{\ge 0}$ whose elements are local intersection indices $C\cdot_z G^{-1}(0)$
as $G$ ranges through all the polynomials $\CC[x,y]$ that do not vanish on any of the components of $C$ containing $z$. By convention, zero is always considered as an element of $S(z)$: it corresponds
to a polynomial $G$ that does not vanish at $z$.
\end{definition}
\begin{problem}
Show that $S$ is in fact a semigroup.
\end{problem}
\begin{problem}
Show that the smallest non--zero element of the semigroup is the multiplicity of a singular point.
\end{problem}
The notion of the semigroup as defined here is useful mostly for cuspidal singular points. If $z$ has $r>1$ branches, it might be more
natural to consider a semigroup of $\ZZ^r$, whose elements are vectors formed by local intersection indices with the branches. There is
a significant difference between the cuspidal and non-cuspidal case. In the present notes we focus mostly on the cuspidal case.
\begin{theorem}[see e.g. \expandafter{\cite[Chapter 4]{Wa}}]\label{thm:semigroup}
The semigroup of a cuspidal singular point $z$ has the following properties.
\begin{itemize}
\item The \emph{gap set} $G:=\ZZ_{\ge 0}\setminus S$ has cardinality $\mu/2$. Here $\mu$ is the Milnor number.
\item The maximal element of $G$ is equal to $\mu-1$.
\item The semigroup has the following symmetry property: for any $x\in\ZZ$, either $x\in S$, or $2g-1-x\in S$, but never both.
\end{itemize}
\end{theorem}
\begin{problem}
Deduce the first two properties in the statement of Theorem~\ref{thm:semigroup} from the third one.
\end{problem}
\begin{problem}
Prove elementarily that if $S$ is a semigroup generated by $p$ and $q$, then the maximal element that does not belong to the semigroup
is $(p-1)(q-1)-1$.
\end{problem}
\begin{problem}
Suppose $S$ is a semigroup of $\ZZ_{\ge 0}$ such that $G=\ZZ_{\ge 0}\setminus S$ is finite. Assume that $S$ has three generators $(p,q,r)$.
Try finding an explicit formula for the maximal element of $G$ if $p=2$ or $p=3$ and see how hard it is. This shows that the second
property of Theorem~\ref{thm:semigroup} is very special. See \cite{RA} for a detailed discussion on numerical semigroups.
\end{problem}
We have the following fact first established in \cite{CDG}. We refer also to \cite[Chapters 4,5]{Wa}.

\begin{theorem}\label{thm:semalex} 
Let $z$ be a cuspidal singular point with a semigroup $S$. Let $G=\ZZ_{\ge 0}\setminus S$ be the gap set. Then
\[1+(t-1)\sum_{r\in G} t^r\] 
is the Alexander polynomial of the link of the singular point $z$.
\end{theorem}

The result is unexpected and shows very deep relations between singularity theory and knot theory, see \cite{CDG} for more details.
Nevertheless, 
the theorem is not hard to prove, since we exactly know which links can arise from cuspidal singularities. They are, see \cite{Bu,Za}, iterated cables on torus knots. Both the link of the singularity and the semigroup can be
determined from the Puiseux pairs of singular points. The proof of Theorem~\ref{thm:semalex} consists of calculating
both sides in terms of Puiseux pairs and in fact, the only non-trivial result that is used is the formula for the Alexander polynomial
of a cable. On the other hand we have just shown that the semigroup is a topological invariant of a singular point.

\subsection{Links of singularities as L--space knots}
In \cite{hedden2} Hedden proved the following result.
\begin{theorem}
The link of a cuspidal singularity is an L--space knot.
\end{theorem}
Suppose $z$ is a cuspidal singular point with semigroup $S$ and link $K$. The semigroup determines the Alexander polynomial by 
Theorem~\ref{thm:semalex}.
As $K$ is an L--space knot, the Alexander polynomial 
of $K$ determines the chain complex $CFK^\infty$. This chain complex determines the concordance invariants $V_m$. Therefore, the numbers
$V_m$ can be calculated directly from the semigroup $S$. An explicit computation is not hard.
\begin{theorem}[compare \expandafter{\cite[Proposition 4.6]{BL}}]\label{thm:BL1}
We have $V_{g+m}=\#\{j\ge m\colon j\notin S\}$, where $g$ is the genus of the knot $K$.
\end{theorem}
\begin{problem}
Use Theorem~\ref{thm:semalex}  and the explicit algorithm for calculating $CFK^\infty$ of an L--space knot
(see Proposition~\ref{prop:staircasecomplex} and the algorithm above it) to prove Theorem~\ref{thm:BL1}.
\end{problem}

In conjunction with Large Surgery Theorem~\ref{thm:lst} 
this result will allow us to calculate $d$--invariants of large surgeries on links of cuspidal singularities from the semigroup only.

\begin{remark}
Even if Theorem~\ref{thm:BL1} is easy to believe and rather straightforward to prove, it sets a right perspective. The semigroup
is a natural object to study when one is interested in applications of Heegaard Floer techniques in singularity theory.
\end{remark}
\section{Rational cuspidal curves and beyond}\label{sec:rational}

\subsection{What is a rational cuspidal curve?}
We now pass to considering complex curves in $\CC P^2$. Let $C\subset\CC P^2$ be an \emph{irreducible} curve, that is, a curve
which cannot be presented as a union of two curves $C_1\cup C_2$. Put differently, an irreducible curve is a curve that can be
realized as a zero set of a homogeneous polynomial $F$ which is irreducible in $\CC[x,y,z]$. The \emph{degree} of the curve $C$
is the degree of a reduced homogenoeous polynomial whose zero set is $C$.

If $C$ is a smooth curve of degree $d$, its genus is determined by $d$, namely $g(C)=\frac{(d-1)(d-2)}{2}$. For singular curves
the notion of genus can be generalized in many non-equivalent ways. The most useful to us is the notion of a geometric genus. To introduce it, recall that
any complex curve in $\CC P^2$ admits a so called \emph{normalization}. This is a smooth complex curve $\gS$
together with a complex map $\pi\colon \gS\to C$, such that the inverse image of each of the singular points is finite and the preimage
of each smooth point consists of a single point. It is not hard to show that a normalization exists and is well defined up to a biholomorphism.
\begin{definition}
The \emph{geometric genus} $p_g(C)$ is the genus of the normalization~$\gS$. A curve $C$ is called \emph{rational} if its geometric
genus is zero. A curve is called \emph{rational cuspidal} if it is rational and all its singular points (if any) are cuspidal.
\end{definition}
\begin{problem}
Prove that $C$ is rational cuspidal if and only if it is homeomorphic to the sphere $S^2$.
\end{problem}

For completeness, we recall a classical numerical formula for the geometric genus.
\begin{theorem}\label{thm:genusformula}
Suppose $C$ has degree $d$ and singular points $z_1,\ldots,z_n$. Let $\delta_1,\ldots,\delta_n$ be the $\delta$--invariants of $z_1,\ldots,z_n$
(for a cuspidal singularity the $\delta$--invariant is equal to the genus of the link; if $z_i$ has $r_i>1$ branches, then 
$2\delta_i=\mu_i+r_i-1$). Then
\[p_g(C)=\frac12(d-1)(d-2)-\sum_{i=1}^n\delta_i.\]
\end{theorem}
Milnor in \cite[Section 10]{Milnor-sing} attributes Theorem~\ref{thm:genusformula} to Serre, however at least some variant of it was known
already in the XIXth century.

\subsection{A quick tour of rational cuspidal curves}

Rational cuspidal curves have been an object of interest at least since the end of the XIXth century. 
Before we state one of the most important conjectures on rational cuspidal curves, we give a definition; see \cite[Section I.4]{Hart}.
\begin{definition}\
\begin{itemize}
\item A \emph{rational map} between two algebraic irreducible varieties  $f\colon X\to Y$ is an equivalence class of pairs $(U,f_U)$, where $U$ is a Zariski open
subset of $X$ and $f_U\colon U\to Y$. Two pairs $(U,f_U)$ and $(U',f_{U'})$ are said to be equivalent if they agree on $U\cap U'$.
\item A \emph{birational map} $f\colon X\to Y$ is a rational map that admits a rational inverse, that is, a rational map $g\colon Y\to X$
such that $f\circ g=id_Y$ and $g\circ f=id_X$, where the equalities are understood as equivalences of rational maps.
\end{itemize} 
\end{definition}
A reader not familiar with algebraic geometry might be worried that a rational map is defined only on an open subset of $X$. The key
word here is `Zariski open'. The Zariski topology is completely different from the metric topology. Open sets are basically complements
of hypersurfaces, so an open set in Zariski topology means an open-dense subset of $X$ in the metric topology, whose complement
is of complex codimension at least $1$.

\begin{example}
A blow-up and blow-down are birational maps.
\end{example}

\begin{example}
It was proved already by Zariski, see \cite{Zarbook}, that any birational map between two algebraic surfaces is a sequence of blow-ups and blow-downs.
\end{example}

Now we pass to an important definition.
\begin{definition}
A curve $C\subset\CC P^2$ is called \emph{rectifiable} if there exists a birational map $f\colon\CC P^2\to\CC P^2$ such that the (closure
of) the image $f(C)$ is a straight line.
\end{definition}
\begin{problem}
Show that a curve $C$ given by $x^3=y^2z$ in homogeneous coordinates $[x:y:z]$ in $\CC P^2$ is rectifiable.
\end{problem}

In 1928 Coolidge \cite{Cool28}
stated a conjecture, which was given its final shape by Nagata \cite{Naga60}. 
\begin{conjecture}[The Coolidge--Nagata conjecture]
Any rational cuspidal curve is rectifiable.
\end{conjecture}
The conjecture eluded all approach until 2015, when two mathematicians, Koras and Palka,
found a brilliant proof relying on the minimal model program.
\begin{theorem}[see \cite{KP,Pal}]
The Coolidge--Nagata conjecture is true. That is, 
every rational cuspidal curve
in $\CC P^2$ can be transformed into a line by means of birational transformations of $\CC P^2$. 
\end{theorem}

The meaning of the conjecture is that every rational cuspidal curve can be constructed by taking
a line and applying a sequence of blow--ups and blow--downs. This does not solve the problem of classifying
all the rational cuspidal curves, because the configurations of blow--ups and blow--downs might be rather complicated.

\begin{problem}[Open]\label{prob:onlyopen}
Use methods of Koras and Palka to prove that every rational cuspidal curve in a Hirzebruch surface is rectifiable. See
\cite{Moerat,Moe2,BM} for more on rational cuspidal curves in Hirzebruch surfaces.
\end{problem}

Another problem concerning rational cuspidal curves is to establish bounds for the number of possible singular points. The following conjecture
is due to Orevkov. It circulated among the experts for a long time and was stated explicitly
in a paper by Piontkowski \cite{Piontkowski}.
\begin{conjecture}
Any rational cuspidal curve $C\subset\CC P^2$ has at most four singular points. Moreover, there is only one curve (of degree $5$)
that has precisely four singular points.
\end{conjecture}
For a long time the best upper bound was $8$ \cite{Tono05}. Recently Palka improved this bound to $6$, see \cite{Pal-fin}.

There is another conjecture due to Flenner and Zajdenberg \cite{FlZa95}, called the Rigidity Conjecture. Introducing all
the terminology needed to state it is beyond the scope of the present article, so we will be rather informal. Suppose
$C\subset\CC P^2$ is a rational cuspidal curve. We resolve the singularities of $C$ to obtain a surface $V$ together with
a rational map $\pi\colon V\to\CC P^2$. The inverse image $D=\pi^{-1}(C)$ (for algebraic geometers: we take a reduced scheme structure
on $D$) is a simple normal crossing divisor, that is, it is a union of holomorphic spheres intersecting transversally such that
no self--intersections are allowed and triple intersection points are excluded. Such a resolution $(V,D)$ always exists, see \cite{BK,GLS,EN}
or almost any book on complex plane curves.

One studies the infinitesimal deformations of the pair $(V,D)$ in the spirit of Kodaira and Spencer \cite{KS}. There is a sheaf, $\Theta_V\langle D\rangle$ of complex vector fields on $V$ that are tangent to $D$. It turns out, see \cite{FlZa95}, that this sheaf controls the deformations
of the pair $(V,D)$, that is, $h^1$ of this sheaf is the space of infinitesimal deformations of the pair $(V,D)$ and $h^2$ is the space
of obstructions to the deformations. If $h^2(\Theta\langle D\rangle)=0$, the deformations are unobstructed, because higher obstructions
($h^i$ for $i>2$) vanish for dimensional reasons. Now the Flenner--Zajdenberg rigidity conjecture states that $h^2(\Theta\langle D\rangle)=0$,
that is, infinitesimal deformations are unobstructed. In most interesting cases 
$h^0(\Theta\langle D\rangle)=0$, so $\chi(\Theta\langle D\rangle)\le 0$ (recall that $\chi=h^0-h^1+h^2$).
On the other hand, the Riemann--Roch theorem for surfaces tells us that $\chi(\Theta\langle D\rangle)=K(K+D)$, so the Rigidity Conjecture implies that
$K(K+D)\le 0$, but the converse implication does not necessarily hold, which is one of the reasons why the conjecture is so difficult. It
is well--known to the experts that the Rigidity Conjecture implies the Cooligde--Nagata conjecture, but again, the converse implication
is not true; see also \cite{Pal-fin} for a more detailed discussion.

\subsection{Partial results on classification}
Rational cuspidal curves with logarithmic Kodaira dimension
less than $2$ have already been classified, see the introduction in \cite{FLMN06}
for a concise summary of the results. The logarithmic Kodaira dimension, $\ol{\kappa}$, defined in \cite{Iit}, 
is an invariant of a complement~$V\setminus D$, where $V$ is a projective surface and $D$ a divisor on it. If $V$ is a surface,
then $\ol{\kappa}(V\setminus D)\subset\{-\infty,0,1,2\}$. It is a result of Wakabayashi \cite{Wak}, that if~$C\subset \CC P^2$
is a rational cuspidal curve such that $\ol{\kappa}(\CC P^2\setminus C)\le 0$, then $C$ has at most one singular point,
moreover if $\ol{\kappa}(\CC P^2\setminus C)=1$, then~$C$ has at most two singular points. 

The classification of rational cuspidal curves such that $\ol{\kappa}(\CC P^2\setminus C)=-\infty$ was achieved by
Kashiwara \cite{Kas}. The case~$\ol{\kappa}(\CC P^2\setminus C)=0$ was excluded by Tsunoda in \cite{Ts}, another reference is \cite{Orev02}.
Classification of curves with $\ol{\kappa}(\CC P^2\setminus C)=1$ was started by Kishimoto \cite{Kis} and completed by Tono in \cite{TonoR}.

The case~$\ol{\kappa}(\CC P^2\setminus C)=2$ is the hardest. There is a program of Palka and Pe\l{}ka on classifying
all rational cuspidal curves that satisfy the Flenner--Zaidenberg Rigidity conjecture; see \cite{PalPel} for the first important results
in that direction.

On the other side, somehow setting aside the logarithmic Kodaira dimension, 
in \cite{FLMN04} an attempt was made to classify rational cuspidal curves. The result was only the first step, namely 
the following result is proved in \cite{FLMN04}.
\begin{theorem}\label{thm:flmnclass}
Suppose $C$ is a rational cuspidal curve in $\CC P^2$ having precisely one singular point. Assume additionally that this
singular point has a single Puiseux pair $(p,q)$. Then the pair $(p,q)$ belongs to one of the following
list. Moreover, each pair below can be realized by a rational cuspidal curve.
\begin{itemize}
\item[(a)] $(d-1,d)$ for any $d>1$,
\item[(b)] $(d/2,2d-1)$ for any even $d>1$,
\item[(c)] $(\phi_{j-2}^2,\phi_j^2)$ for $j$ odd and $j\ge 5$, where $\phi_i$ are Fibonacci numbers normalized in such a way that
$\phi_0=0$, $\phi_1=1$,
\item[(d)] $(\phi_{j-2},\phi_{j+2})$ for $j\ge 5$ odd,
\item[(e)] $(\phi_4,\phi_8+1)=(3,22)$,
\item[(f)] $(2\phi_4,2\phi_8+1)=(6,43)$.
\end{itemize}
\end{theorem}
\begin{problem}
Determine the degree of $C$ in each of the cases (c)--(f). Cases (a) and (b) are trivial.
\end{problem}

In \cite{Bod}, based on the thesis of Tiankai Liu \cite{TLiu}, Bodn\'ar gave an analogue of Theorem~\ref{thm:flmnclass}
for rational cuspidal curves with one singular point such that the singular point has two Puiseux pairs. The result is
more complicated.

\subsection{The tubular neighborhood of a rational cuspidal curve}
We pass to applications of Heegaard Floer theory to rational cuspidal curves.

Let $C\subset\CC P^2$ be a rational cuspidal curve of degree $d$. We aim to construct a `tubular' neighborhood of $C$ in $\CC P^2$.
The word `tubular' is in quotation marks, because $C$ is not locally flat and cannot have a product neighborhood. However, the
following, rather obvious, construction will fit well into our applications.

For any singular point $z_i$ of $C$ take a small ball $B_i$ with center $z_i$. The complement $C\setminus \bigcup B_i$ is a smooth curve
so we take product neighborhood $N_0$. We will require that $N_0$ is thin as compared to the radii of all the $B_i$. Set $N=N_0\cup\bigcup B_i$.
Clearly $N$ is an open set containing $C$. Alternatively we could define $N$ as a set of points at distance less than $\varepsilon$ of $C$
for $\varepsilon>0$ sufficiently small; this leads to essentially the same space $N$. 
However, the first construction has an advantage, namely the following Lemma is easy to notice.

\begin{lemma}\label{lem:boundary}
Let $Y=\partial N$. Let $z_1,\ldots,z_n$ be the singular points of $C$ and $K_1,\ldots,K_n$ its links. Set $K=K_1\#\ldots\# K_n$. Then 
$Y=S^3_{d^2}(K)$.
\end{lemma}
\begin{problem}
Prove Lemma~\ref{lem:boundary} for $n=1$ (hint: notice that $d^2$ is the self--intersection of $C$).
\end{problem}
For $n>1$ the proof of Lemma~\ref{lem:boundary} is given in \cite[Section 3]{BL}.

Let us consider $W=\CC P^2\setminus N$. The homology of $W$ can be easily calculated: notice that $C$ is a generator of $H_2(\CC P^2;\QQ)$,
removing $C$ from $\CC P^2$ should yield a rational homology ball. This indeed is so.
\begin{problem}
Prove that $H_k(W;\QQ)=0$ if $k>0$.
\end{problem}
\begin{problem}
Calculate the $\ZZ$--homologies of $W$.
\end{problem}

We pass to describing \spinc{} structures on $W$, with the aim to calculate which \spinc{} structures on $Y=\partial N$ extend over $W$.
The three--manifold $Y$ is a $d^2$ surgery on $K$. Therefore, we can enumerate \spinc{} structures on $Y$ by integers $m\in[-d^2/2,d^2/2)$
as in Proposition~\ref{prop:unique} above.
\begin{problem}
Show that the \spinc{} structure $\sss_m$ on $Y$ extends to a \spinc{} structure $\sst_m$ on $N$ such that $\langle c_1(\sst_m),C\rangle+d^2=2m$.
\end{problem}
Suppose now a \spinc{} structure $\sss_m$ on $Y$ extends to a \spinc{} structure $\sst'_m$ on $W$. The \spinc{} structures $\sst_m$ and
$\sst_m'$ on $N$ and $W$ glue together to a \spinc{} structure $\sst''_m$ on $\CC P^2$. Now $\CC P^2$ is a closed 
simply connected four--manifold. By Corollary~\ref{cor:spinc4}
it follows
that $c_1(\sst''_m)=(2j+1)[H]$ for some $j\in\ZZ$, where $[H]$ is the generator of $H^2(\CC P^2;\ZZ)$. In particular
\[\langle c_1(\sst_m),C\rangle=\langle c_1(\sst''_m),C\rangle=(2j+1)d.\]
Applying Proposition~\ref{prop:unique} we obtain the following statement.
\begin{lemma}\label{lem:extendoverW}
If a \spinc{} structure $\sss_m$ on $Y$ extends over $W$, then $m=\frac12(d^2-(2j+1)d)$ for some $j\in\ZZ$.
\end{lemma}

\subsection{Heegaard Floer homology applied to rational cuspidal curves}
Let us now gather all pieces of a puzzle to restrict the Alexander polynomials of links of singular points a rational cuspidal curve.
We suppose first that $C$ is a rational cuspidal curve of degree $d$ with one singular point $z$, whose link is $K$ and whose semigroup is $S$.
\begin{itemize}
\item The boundary of the tubular neighborhood of $C$ is $S^3_{d^2}(K)$. 
\item $K$ is an algebraic knot, hence an L--space knot.
\item The $V_m$ invariants of $K$ can be calculated from the semigroup $S$.
\item The genus of $K$ is $\frac12(d-1)(d-2)$. The surgery coefficient $d^2$ is greater than twice the genus.
\item The Large Surgery Theorem applies. We can express the $d$--invariants of $Y$ in terms of the semigroup.
\item On the other hand $Y$ bounds a rational homology ball $W$. Hence $d(Y,\sss_m)=0$ for every \spinc{} structure $\sss_m$ on $Y$
that extends over $W$.
\item The \spinc{} structures on $Y$ that extend over $W$ were calculated in Lemma~\ref{lem:extendoverW} above. 
We get restrictions for the distribution of
elements in the semigroup $S$.
\end{itemize}
These restrictions can be stated as follows.
\begin{theorem}[see \cite{BL}]\label{thm:BL-main1}
For any $j=0,\ldots,d-2$ we have $\#S\cap[0,jd+1)=\frac12(j+1)(j+2)$.
\end{theorem}
\begin{problem}
Using the itemized list, prove Theorem~\ref{thm:BL-main1}.
\end{problem}

\smallskip
The case $n>1$ is similar; the new technical difficulties are rather minor. Suppose $C$ is a rational cuspidal curve of degree $d$
with singular points $z_1,\ldots,z_n$, whose links are $K_1,\ldots,K_n$ respectively and the associated semigroups are
$S_1,\ldots,S_n$. Set $K=K_1\#\ldots\# K_n$. Then
$Y=\partial N$ is $S^3_{d^2}(K)$ as states Lemma~\ref{lem:boundary} above. However, as mentioned in Lemma~\ref{lem:propertiesofLspace}(a)
$K$ has no chances to be an L--space knot if $n>1$, in
fact, $K$ is not prime. Luckily $K$ is a connected sum of L--space knots $K_1,\ldots,K_n$, hence by the K\"unneth formula 
(Proposition~\ref{prop:connectedknots}) we have
$CKF^\infty(K)=CFK^\infty(K_1)\otimes\ldots\otimes CFK^\infty(K_n)$. The K\"unneth formula allows us to express the $V_m$ invariants
of $K$ in terms of the $V_m$ invariants of the summands. Acting as in Problem~\ref{prob:sumconvolution} we obtain
\[V_m(K)=\min_{m_1+\ldots+m_n=m} V_{m_1}(K_1)+\ldots+V_{m_n}(K_n).\]
Now each of the $V_{m_i}(K_i)$ can be expressed from the semigroup of the singular point $z_i$. Putting things together and
acting as in the case $n=1$ we arrive at the following result.
\begin{theorem}[see \cite{BL}]\label{thm:BL-main2}
For any $j=0,\ldots,d-2$ we have
\[\min_{k_1+\ldots+k_n=jd+1}\sum_{i=1}^n\#S_i\cap[0,k_i)=\frac12(j+1)(j+2),\]
\end{theorem}

\subsection{Strength and weakness of Theorems~\ref{thm:BL-main1} and~\ref{thm:BL-main2}}
Theorem~\ref{thm:BL-main1} has proved very useful in classifying rational cuspidal curves with one singular point. It is possible 
to give a full list of possible rational cuspidal curves with one singular point having one Puiseux pair
(this is equivalent to saying that the link is a torus knot), using
essentially Theorem~\ref{thm:BL-main1}. This classification was first done in \cite{FLMN04}, the proof using Theorem~\ref{thm:BL-main1}
is considerably simpler.
\begin{problem}
Show that there is a value $d_0>0$ such that for any $d>d_0$ there are no rational cuspidal curves of degree $d$
with one Puiseux pair $(p,q)$ such that $p<q$ and $p\in(d/2,d-1)$.
\end{problem}
\begin{problem}
Use Theorem~\ref{thm:BL-main1} to prove that if a rational cuspidal curve of degree $d$ has one singular point, then its multiplicity is at least $d/3$. The original proof
of Matsuoka and Sakai \cite{MaSa89} uses the BMY inequality.
\end{problem}

As it was shown in \cite[Section 6]{BL}, 
for $n=1$ and a general number of Puiseux pairs, the restriction of Theorem~\ref{thm:BL-main1} has approximately
the same strength as the spectrum semicontinuity property (see \cite{FLMN04} for more details). There are relatively few cases when 
Theorem~\ref{thm:BL-main1} gives an obstruction, while the spectrum semicontinuity does not. There are also very few cases when the opposite holds.

Surprisingly, for $n\ge 2$ the situation changes and Theorem~\ref{thm:BL-main2} is not that strong anymore. A potential problem was
discovered by Bodn\'ar and N\'emethi \cite{BodNe} (see also \cite{FK}). Before we state it, we give an example.


When trying to classify all rational cuspidal curves of degree $5$ with two singular points, both having multiplicity $2$, one finds that the genus formula (Theorem~\ref{thm:genusformula}) implies that we might have three cases: either the singular points are $(2;3),(2;11)$, or $(2;5),(2;9)$, or $(2;7),(2;7)$.
Such classification was already known long before; see \cite{Moe08}.
\begin{problem}
Prove that in each of the three cases, if $S_1$ and $S_2$ denote the corresponding semigroups, we have
\[\min_{i+j=k}\# S_1\cap[0,i)+\# S_2\cap[0,j)=
\begin{cases} 
\lfloor (k+1)/2\rfloor& k\le 12\\
k-6&k\ge 12.
\end{cases}
\]
\end{problem}
Therefore Theorem~\ref{thm:BL-main2} is unable to distinguish between the three cases. As the curve of degree $5$ with
singular points $(2;5)$ and $(2;9)$ actually exists, we cannot obstruct any of the remaining two cases. On the other hand,
these remaining two cases do not exist.

The deeper reason was discovered in \cite[Section 5]{BodNe}. 
To describe it we introduce a bit of notation. Namely, to any cuspidal singular point $z$
we associate its \emph{multiplicity sequence} $M_z$. For a set of singular points $z_1,\ldots,z_n$ the union
$\cM=M_1\cup\ldots\cup M_n$ is an unordered tuple of integers greater than $1$ (each integer can enter several times in the union).
We say that $\cM=\cM'$ if for each integer $i\ge 2$ the number of times $i$ appears in $\cM$ is equal to the number of times it appears in $\cM'$.
We have the following result.
\begin{theorem}[see \expandafter{\cite[Theorem 5.1.3]{BodNe}}]
Suppose $z_1,\ldots,z_n$ and $z_1',\ldots,z_n'$ are two collections of singular points. Let $S_1,\ldots,S_n$ and $S_1',\ldots,S'_{n'}$ be
corresponding semigroups and $M_1,\ldots,M_{n'}'$ be the multiplicity sequences. Set $\cM=M_1\cup\ldots\cup M_n$ and 
$\cM'=M_1'\cup\ldots\cup M_{n'}'$. If $\cM=\cM'$, then for every $k\in\ZZ$ we have
\[\min_{i_1+\ldots+i_n=k}\sum_{j=1}^n\#S_j\cap[0,i_j)=\min_{i'_1+\ldots+i'_{n'}=k}\sum_{j'=1}^{n'}\#S'_{j'}\cap[0,i'_{j'}).\]
\end{theorem}
The result greatly limits the applicability of Theorem~\ref{thm:BL-main2} when $n>1$.

\begin{remark}
There exists a Heegaard Floer proof of the fact that a rational cuspidal curve of degree $5$ cannot have two
singular points $(2;3)$ and $(2;11)$, neither can it have two singular points $(2;7)$ and $(2;7)$; see \cite[Section 6.1.3]{Moe08}. 
The proof involves
involutive Floer theory as developed by Hendricks and Manolescu \cite{HM}, which is beyond the scope of the present article.
See \cite{BoHo} for details.
\end{remark}
\subsection{Relation to the FLMN conjecture}
In 2006, Fern\'andez de Bo\-ba\-dilla, Luengo Velasco, Melle Hern\'andez and N\'emethi suggested 
the following conjecture.
\begin{conjecture}[see \cite{FLMN06}]\label{conj:flmn}
Let $C\subset\CC P^2$ be a rational cuspidal curve of degree $d$. Let $K$ be the connected sum of links of singularities of $K$.
Write the Alexander polynomial of $K$ as $\Delta_K(t)=1+(t-1)\delta+(t-1)^2Q(t)$ for some polynomial $Q(t)$ and let $c_j$
be the coefficient of $Q$ at $t^{d(d-3-j)}$. Then for $j=0,\ldots,d-3$
\[c_j\le \frac12(j+1)(j+2).\]
Moreover, if $C$ has precisely one singular point, then $c_j=\frac12(j+1)(j+2)$ for all $j=0, \ldots, d-3$.
\end{conjecture}
\begin{problem}
Show that $\delta$ in the statement of Conjecture~\ref{conj:flmn} is always equal to $\frac12(d-1)(d-2)$.
\end{problem}
Before we discuss the relation of Conjecture~\ref{conj:flmn} to Theorem~\ref{thm:BL-main2} in greater detail, let us first say
something about the motivation of the conjecture. Namely, in a series of papers, N\'emethi and Nicolaescu studied
the relation of the Seiberg--Witten invariants of normal surface singularities and their geometric genus $p_g$. In \cite{NN1} they stated
a conjecture, called the Seiberg--Witten invariant conjecture. The conjecture was verified
 for many families of surface singularities in \cite{NN1,NN2,NN3}.
However, in \cite{LN} it was shown that superisolated surface singularities are expected to satisfy the opposite inequality to the one conjectured
by N\'emethi and Nicolaescu. Superisolated surface singularities were introduced by Luengo in \cite{Lue} and are tightly related
to rational cuspidal curves. In fact, each rational cuspidal curve $C$ gives rise to a superisolated surface singularity whose link
is $S^3_{-d}(K)$, where $d$ is the degree of the curve $C$ and $K$ is the connected sum of links of singular points of $C$.
Conjecture~\ref{conj:flmn} arose as a translation the Seiberg--Witten invariant conjecture for superisolated surface singularities
into the language of rational cuspidal curves.
\begin{remark}
It is no surprise that the Alexander polynomial of $K$ appears in the context of a conjecture related to Seiberg--Witten invariants
of the link $S^3_{-d}(K)$. In fact, the relation of Seiberg--Witten invariants with the Reidemeister--Turaev torsion (see \cite{Tur2} and
references therein) allows to calculate the Seiberg--Witten invariants of $S^3_{-d}(K)$ from the Alexander polynomial of $K$; see e.g. \cite[Formula (3)]{FLMN06}.
\end{remark}

Now we pass to the relations of Conjecture~\ref{conj:flmn} to Theorem~\ref{thm:BL-main2}. We begin with the easy case.
\begin{problem}
Prove that if $C$ has precisely one singular point, then Conjecture~\ref{conj:flmn} is equivalent to Theorem~\ref{thm:BL-main2}.
\end{problem}
The case that $C$ has two singular points is more complicated.
\begin{theorem}[see \cite{BodNe,NaPi}]
If $C$ has two singular points, then Conjecture~\ref{conj:flmn} follows from Theorem~\ref{thm:BL-main2}.
\end{theorem}
However, if $C$ has three or more singular points, Conjecture~\ref{conj:flmn} is false. The following example is elaborated in \cite{BodNe}.
\begin{problem}
Let $C$ be a rational cuspidal curve of degree $8$ with singular points $(6;7)$, $(2;9)$ and $(2;5)$. Prove that $C$ violates
Conjecture~\ref{conj:flmn}.
\end{problem}


\begin{thebibliography}{999}

\bibitem{Baa} S.~Baader, P.~Feller, L.~Lewark, L.~Liechti, \emph{On the topological 4-genus of torus knots}, preprint 2015, arxiv:1509.07634,
to appear in Trans. Amer. Math. Soc.

\bibitem{Bod} J.~Bodn\'ar, \emph{Classification of rational unicuspidal curves with two Newton pairs}, Acta Math. Hungar. \textbf{148} (2016), no. 2, 294--299.

\bibitem{BodNe} J.~Bodn\'ar, A.~N\'emethi, \textit{Lattice cohomology and rational cuspidal curves},  Math. Res. Lett. \textbf{23} (2016), no. 2, 339--375. 

\bibitem{BoGo} M.~Borodzik, E.~Gorsky, \textit{Immersed concordances of links and Heegaard Floer homology},
preprint 2016, arxiv:1601.07507, to appear in Indiana Univ. Math. J.
   
\bibitem{BH} M.~Borodzik, M.~Hedden, \textit{The $\Upsilon$ function of L--space knots is a Legendre transform}, preprint 2015, arxiv:1505.06672. 

\bibitem{BoHo} M.~Borodzik, J.~Hom, \textit{Involutive Heegaard Floer homology and rational cuspidal curves}, preprint 2016,
arxiv:1609.08303.

\bibitem{BL} M.~Borodzik, C.~Livingston,  \textit{Heegaard Floer homologies and rational cuspidal curves}, Forum of Math. Sigma, \textbf{2} (2014), e28, 23 pages.

\bibitem{BM} M.~Borodzik, T.~Moe, \emph{Topological obstructions for rational cuspidal curves in Hirzebruch surfaces}, preprint,
arxiv:1410.4464, to appear in Michigan Math. Journal.
  
\bibitem{BK} E. Brieskorn, H. Kn\"orrer, \textit{Plane Algebraic Curves}, Birkh\"auser, Basel, 1986.
 
\bibitem{Bu} W.~Burau, \textit{Kennzeichnung der Schlauchknoten}, Abhandlungen aus dem Mathematischen Seminar der Universit\"at Hamburg \textbf{9} (1932), 125--133.

\bibitem{CDG}
A.~Campillo, F.~Delgado, S.~Gusein-Zade, 
\emph{The Alexander polynomial of a plane curve singularity via the ring of functions on it},
Duke Math. J. \textbf{117} (2003), no. 1, 125--156. 

\bibitem{Cerf} J.~Cerf, \emph{La stratification naturelle des espaces de fonctions diff\'erentiables r\'eelles et le th\'eor\`eme 
de la pseudo-isotopie}, Inst. Hautes \'Etudes Sci. Publ. Math. No. \textbf{39} (1970), 5--173.

\bibitem{Cool28} J.~Coolidge, \emph{A treatise on plane algebraic curves}, Oxford Univ. Press, Oxford, 1928.

\bibitem{EN} D. Eisenbud, W. Neumann, \textit{Three-dimensional link theory and invariants of plane curve singularities}, Annals Math. Studies \textbf{110}, Princeton University Press, Princeton, 1985.

\bibitem{FK} P.~Feller, D.~Krcatovich, \emph{On cobordisms between knots, braid index, and the Upsilon-invariant},
preprint 2016, arxiv:1602.02637.

\bibitem{FLMN04} J. Fern\'andez de Bobadilla, I. Luengo, A.~Melle-Hern\'andez, A.~N\'emethi, \emph{Classification of rational unicuspidal projective curves whose singularities have one Puiseux pair}, Proceedings of S\~ao Carlos Workshop 2004 Real and Complex Singularities, Series Trends in Mathematics, Birkh\"auser 2007, 31--46.

\bibitem{FLMN06} J. Fern\'andez de Bobadilla, I. Luengo, A.~Melle-Hern\'andez, A.~N\'emethi, \emph{On rational cuspidal projective plane curves}, Proc. of London Math. Soc., \textbf{92}  (2006), 99--138. 
 
\bibitem{FlZa95}
H.~Flenner and M.~Zaidenberg, \emph{On a class of rational cuspidal plane curves}, Manuscripta Math. \textbf{89} (1996), no.~4, 439--459. 

\bibitem{Fri} T.~Friedrich, \emph{Dirac operators in Riemannian geometry}, Graduate Studies in Mathematics, 25. American Mathematical Society, Providence, RI, 2000.

\bibitem{Gigi} P.~Ghiggini, \emph{Knot Floer homology detects genus-one fibred knots}, Amer. J. Math. \textbf{130} (2008), no. 5, 1151--1169.
 
\bibitem{GS} R. E. Gompf, A. I. Stipsicz, \textit{4--Manifolds and Kirby Calculus (Graduate Studies in Mathematics)}, American Mathematical Society, 1999.

\bibitem{GLS} G-M.~Greuel, C. Lossen, E.~Shustin, \textit{Introduction to Singularities and Deformations}, Springer--Verlag, Berlin--Heidelberg--New York, 2006.

\bibitem{GH}  J.~Guckenheimer, P.~Holmes, \emph{Nonlinear oscillations, dynamical systems, and bifurcations of vector fields}, Revised and corrected reprint of the 1983 original. Applied Mathematical Sciences, 42. Springer-Verlag, New York, 1990. 

\bibitem{Hart} R.~Hartshorne, \emph{Algebraic geometry}, Graduate Texts in Mathematics, No. 52. Springer-Verlag, New York-Heidelberg, 1977.

\bibitem{hedden1} M.~Hedden, \emph{Notions of positivity and the Ozsváth-Szabó concordance invariant}, J. Knot Theory Ramifications \textbf{19} (2010), no. 5, 617--629.

\bibitem{hedden2} M.~Hedden, \textit{On knot Floer homology and cabling.  II.}, Int. Math. Res. Not. IMRN 2009,  2248--2274.

\bibitem{HM} K.~Hendricks, C.~Manolescu, \emph{Involutive Heegaard Floer homology}, preprint 2015, arxiv:1507.00383, to appear in Duke Math. Journal.

\bibitem{HMZ} K.~Hendricks, C.~Manolescu, I.~Zemke, \emph{A connected sum formula for involutive Heegaard Floer homology}, preprint 2016, arxiv:1607.07499.

\bibitem{Hom0} J.~Hom, \emph{A note on cabling and L-space surgeries}, Algebr. Geom. Topol. \textbf{11} (2011), no. 1, 219--223.

\bibitem{Hom} J.~Hom, \emph{A survey on Heegaard Floer homology and concordance}, preprint 2015, arxiv:1512.00383.

\bibitem{HLW} J.~Hom, T.~Lidman and L.~Watson, \emph{The Alexander invariant, Seifert forms, and categorification}, preprint, arXiv: 1501.04866, to appear
in J. Topology.

\bibitem{Iit} S. Iitaka, \textit{On logarithmic Kodaira dimension of algebraic varietes}, in: `Complex Analysis and Algebraic Geometry' (A collection of papers dedicated to K. Kodaira), Iwanami, 1977, pp. 175--189.

\bibitem{Juh-sut} A.~Juh\'asz, \emph{Holomorphic discs and sutured manifolds},  
Algebr. Geom. Topol. \textbf{6} (2006), 1429--1457.
 
\bibitem{Juh} A.~Juh\'asz,  \emph{A survey of Heegaard Floer homology},  New Ideas in Low Dimensional Topology, World Scientific, 2014, 237--296.

\bibitem{JT} A.~Juh\'asz, D.~Thurston, \emph{Naturality and mapping class groups in Heegaard Floer homology}, preprint 2012, arxiv:1210.4996.

\bibitem{Kas}  H.~Kashiwara, \emph{Fonctions rationnelles de type (0,1) sur le plan projectif complexe},  Osaka J. Math. \textbf{24} 
(1987), no. 3, 521--577. 

\bibitem{Kis} T.~Kishimoto, \emph{Projective plane curves whose complements have logarithmic Kodaira dimension one}, Japan. J. Math.  \textbf{27} (2001), no. 2, 275–310.

\bibitem{KS} K.~Kodaira, D.~Spencer, \emph{On deformations of complex analytic structures. I, II.}, Ann. of Math. \textbf{67} (1958) 328--466.

\bibitem{KP} M.~Koras, K.~Palka, \emph{The Coolidge-Nagata conjecture}, preprint 2015, arxiv:1502.07149.

\bibitem{Krc} D.~Krcatovich, \emph{The reduced knot Floer complex}, Topology Appl. \textbf{194} (2015), 171--201.


\bibitem{KM} P.~Kronheimer, T.~Mrowka, \textit{The genus of embedded surfaces in the projective plane}, Math. Res. Lett. \textbf{1} (1994),   797--808. 

\bibitem{KM-book} P.~Kronheimer, T.~Mrowka, \emph{Monopoles and three--manifolds}, 
New Mathematical Monographs, 10. Cambridge University Press, Cambridge, 2007. 

\bibitem{KuLeTa11} Ç.~Kutluhan, Y.~Lee, C.~H.~Taubes, \textit{$HF = HM$ I: Heegaard Floer homology and Seiberg-Witten Floer homology}, preprint 2011, 
{arxiv:1007.1979v5}

\bibitem{LeeWil} R.~Lee,  D.~Wilczyński, 
\emph{Locally flat 2-spheres in simply connected 4--manifolds},
Comment. Math. Helv. \textbf{65} (1990), no. 3, 388--412.

\bibitem{levine-ruberman} A. Levine and D.Ruberman, {\em Generalized Heegaard Floer correction terms}, Proceedings of the 20th G\"okova Geometry/Topology Conference, 76--96.


\bibitem{Lip} R.~Lipshitz,
\emph{A cylindrical reformulation of Heegaard Floer homology}, Geom. Topol. \textbf{10} (2006), 955--1097. 

\bibitem{TLiu} T.~Liu, \textit{On planar rational cuspidal curves}, Ph.D. thesis, 2014, at M.I.T., available at \url{http://dspace.mit.edu/bitstream/handle/1721.1/90190/890211671.pdf}.

\bibitem{Liu} Y. Liu, \emph{L-space surgeries on links}, to appear in Quant. Topol., arXiv:1409.0075.

\bibitem{Lue} I.~Luengo, \emph{The $\mu$--constant stratum is not smooth}, Invent. Math. \textbf{90} (1987) 139--152. 

\bibitem{LN} I.~Luengo, A. Melle Hern\'andez, A. N\'emethi, \emph{Links and analytic invariants of superisolated singularities}, 
J. Algebraic Geom. 14 (2005) 543--565.

\bibitem{LOT0} R.~Lipshitz, P.~Ozsv\'ath, D.~Thurston, \emph{Bordered Heegaard Floer homology: Invariance and pairing}, preprint,  arXiv:0810.0687.

\bibitem{LOT1} R.~Lipshitz, P.~Ozsv\'ath, D.~Thurston, \emph{Tour of bordered Floer theory}, Proc. Natl. Acad. Sci. USA \textbf{108} (2011), no. 20, 8085--8092.

\bibitem{LOT2} R.~Lipshitz, P.~Ozsv\'ath, D.~Thurston, \emph{Notes on bordered Floer homology}, Contact and symplectic topology, 275--355, 
Bolyai Soc. Math. Stud., 26, J\'anos Bolyai Math. Soc., Budapest, 2014. 

\bibitem{Liv} C.~Livingston,  \textit{Computations of the Ozsv\'ath-Szab\'o knot concordance invariant}, Geom. Topol. \textbf{8} (2004), 735--742. 

\bibitem{Man} C.~Manolescu, \emph{An introduction to knot Floer homology}, preprint 2014, arxiv:1401.7107. To appear in Proceedings of the 2013 SMS Summer School on Homology Theories of Knots and Links. 

\bibitem{MOw} C.~Manolescu, B.~Owens, \emph{A concordance invariant from the Floer homology of double branched covers},
Int. Math. Res. Not. IMRN 2007, no. 20, Art. ID rnm077, 21 pp. 

\bibitem{MOq} C.~Manolescu, P.~Ozsv\'ath, \emph{On the Khovanov and knot Floer homologies of quasi-alternating links},  Proceedings of G\"okova Geometry--Topology 
Conference 2007, 60--81, G\"okova Geometry/Topology Conference (GGT), G\"okova, 2008.  

\bibitem{OM} C.~Manolescu, P.~Ozsv\'ath, \emph{Heegaard Floer homology and integer surgeries on links}, preprint 2010, arxiv:1011.1317.

\bibitem{MaSa89} T.~Matsuoka, F.~Sakai, \emph{The degree of rational cuspidal curves}, Math. Ann. \textbf{285} (1989), 233--247.

\bibitem{Milnor-cob} J.~Milnor, \textit{Lectures on the $h$-cobordism theorem}, Princeton University Press, Princeton, NJ, 1965.

\bibitem{Milnor-sing} J.~Milnor, \textit{Singular points of complex hypersurfaces}, Annals of Mathematics Studies. 61,  Princeton University Press and the University of Tokyo Press, Princeton, NJ, 1968.

\bibitem{Moe08} T. K.~Moe, \emph{Rational cuspidal curves}, Master Thesis, University of Oslo 2008, 
available at arXiv:1511.02691.

\bibitem{Moerat} T. K.~Moe, \emph{Rational cuspidal curves with four cusps on Hirzebruch surfaces}, Le Ma\-te\-ma\-ti\-che Vol. LXIX (2014) Fasc. II, 295--318. doi: 10.4418/2014.69.2.25.

\bibitem{Moe2} T. K.~Moe, \emph{On the number of cusps on cuspidal curves on Hirzebruch surfaces}, Math. Nachrichten. \textbf{288} (2015), 76--88.

\bibitem{Mos} L.~Moser,  \emph{Elementary surgery along a torus knot}, Pacific J. Math. \textbf{38} (1971), 737--745.

\bibitem{Naga60} M.~Nagata, \emph{On rational surfaces. I: Irreducible curves of arithmetic genus 0 or 1}, Mem. Coll. Sci., Univ. Kyoto, Ser. A \textbf{32} (1960), 351--370.

\bibitem{Namba}  M.~Namba, \emph{Geometry of projective algebraic curves}, 
Monographs and Textbooks in Pure and Applied Mathematics, 88. Marcel Dekker, Inc., New York, 1984.
\bibitem{NaPi} P.~Nayar, B.~Pilat, \emph{A note on the rational cuspidal curves}, Bull. Polish Acad. Science., \textbf{62} (2014), no. 2, 117--123.

\bibitem{NN1}
A. N\'emethi, L~Nicolaescu, \emph{Seiberg--Witten invariants and surface singularities}, Geom. Topol. \textbf{6} (2002) 269--328.

\bibitem{NN2} A.~N\'emethi, L.~Nicolaescu, 
\emph{Seiberg–Witten invariants and surface singularities II, singularities with good $\CC^{*}$--action}, 
J. London Math. Soc. \textbf{69} (2004) 593--607. 

\bibitem{NN3} A. N\'emethi, L.~Nicolaescu, \emph{Seiberg--Witten invariants and surface singularities: splicings and cyclic covers}, 
Sel. Math. New Ser. \textbf{11} (2005), 399--451.

\bibitem{Ni} Y.~Ni, \emph{Knot Floer homology detects fibred knots}, Invent. Math. \textbf{170} (2007), no. 3, 577--608.
\emph{Erratum: Knot Floer homology detects fibred knots} Invent. Math. \textbf{177} (2009), no. 1, 235--238.

\bibitem{Ni-thurston} Y.~Ni, \emph{Link Floer homology detects the Thurston norm}, Geom. Topol. \textbf{13} (2009), no. 5, 2991--3019.

\bibitem{NiWu} Y. Ni, Z. Wu, \emph{Cosmetic surgeries on knots in $S^3$}, J. Reine Angew. Math. \textbf{706} (2015), 1--17.

\bibitem{Nic} L.~Nicolaescu, 
\emph{Notes on Seiberg-Witten theory}, Graduate Studies in Mathematics, 28. American Mathematical Society, Providence, RI, 2000. 

\bibitem{Orev02} S.~Orevkov, \emph{On rational cuspidal curves. I. Sharp estimates for degree via multiplicity}, Math. Ann. \textbf{324} (2002), 657--673.

\bibitem{OSS} P.~Ozsv\'ath, A.~Stipsicz, Z. Szab\'o, \textit{Concordance homomorphisms from knot Floer homology}, preprint  2014, {arxiv:1407.1795}.

\bibitem{OSS-book} P.~Ozsv\'ath, A.~Stipsicz, Z.~Szab\'o, \emph{Grid homology for knots and links}, 
Mathematical Surveys and Monographs, 208. American Mathematical Society, Providence, RI, 2015.

\bibitem{OzSz-absolute} P.~Ozsv\'ath, Z.~Szab\'o,  \textit{Absolutely graded Floer homologies and intersection forms for four--manifolds with boundary}, Adv. Math. \textbf{173} (2003), 179--261.

\bibitem{OS-alternating} P.~Ozsv\'ath, Z.~Szab\'o, \emph{Heegaard Floer homology and alternating knots}, Geom. Topol. \textbf{7} (2003), 225--254.                                          
\bibitem{OS-fourball} P.~Ozsv\'ath, Z.~Szab\'o, \textit{Knot Floer homology and the four-ball genus}, Geom. Topol. \textbf{7} (2003), 615--639.

\bibitem{os-threemanifold} P.~Ozsv{\'a}th, Z.~Szab{\'o}, {\em Holomorphic disks and topological invariants for closed three-manifolds}, Ann. of Math. (2) {\bf 159} (2004), 1027--1158.

\bibitem{os-threemanifoldapps} P.~Ozsv{\'a}th, Z.~Szab{\'o},  {\em Holomorphic disks and three manifold invariants: properties and applications},   Ann. of Math. (2) {\bf 159} (2004),  1159--1245. 

\bibitem{OzSz-knot} P.~Ozsv\'ath, Z.~Szab\'o, \textit{Holomorphic disks and knot invariants}, Adv. Math. \textbf{186} (2004), 58--116.

\bibitem{OzSz-genus}   P.~Ozsv\'ath, Z.~Szab\'o,  \textit{Holomorphic disks and genus bounds}, Geom. Topol. \textbf{8} (2004), 311--334. 
                           
\bibitem{OzSz-lspace} P.~Ozsv\'ath,  Z.~Szab\'o, \textit{On knot Floer homology and lens space surgeries}, Topology, {\bf 44} (2005), 1281--1300.
                       
\bibitem{OzSz-cover} P.~Ozsv\'ath, Z.~Szab\'o, \emph{On the Heegaard Floer homology of branched double-covers}, Adv. Math. \textbf{194} (2005), no. 1, 1--33.

\bibitem{OzSz-triangles} P.~Ozsv\'ath, Z.~Szab\'o,
 \emph{Holomorphic triangles and invariants for smooth four-manifolds}, Adv. Math. \textbf{202} (2006), no. 2, 326--400.

\bibitem{OzSz-intro1} P.~Ozsv\'ath, Z.~Szab\'o,  \textit{An introduction to Heegaard Floer homology}, in:  \textit{Floer homology, gauge theory, and low-dimensional topology}, 3--27, Clay Math. Proc., 5, Amer. Math. Soc., Providence, RI, 2006.

\bibitem{OzSz-intro2} P.~Ozsv\'ath, Z.~Szab\'o,  \textit{Lectures on Heegaard Floer homology}, in:  \textit{Floer homology, gauge theory, and low-dimensional topology}, 29--70, Clay Math. Proc., 5, Amer. Math. Soc., Providence, RI, 2006.

\bibitem{OzSz-integer} P.~Ozsv\'ath, Z. Szab\'o, \emph{Knot Floer homology and integer surgeries}, Algebr. Geom. Topol. \textbf{8} (2008), no. 1, 101--153.

\bibitem{OzSz11} P.~Ozsv\'ath, Z. Szab\'o, \textit{Knot Floer homology and rational surgeries}, Algebr. Geom. Topol. \textbf{11} (2011), 1--68.

\bibitem{Pal-fin} K.~Palka, \emph{Cuspidal curves, minimal models and Zaidenberg's finiteness conjecture}, J. Reine Angew. Math (Crelle's Journal), preprint 2016, arxiv:1405.5346.

\bibitem{Pal}
K.~Palka, \emph{The Coolidge-Nagata conjecture, part I}, Adv. Math. \textbf{267} (2014), 1--43. 

\bibitem{PalPel}
K.~Palka, T.~Pe\l{}ka, \emph{Classification of planar rational cuspidal curves. I. $\CC^{**}$-fibrations}, preprint 2016, arxiv:1609.03992.

\bibitem{Per}
T.~Perutz,
\emph{Hamiltonian handleslides for Heegaard Floer homology},  Proceedings of G\"okova Geometry-Topology Conference 2007, 15--35, 
G\"okova Geometry/Topology Conference (GGT), G\"okova, 2008. 

\bibitem{Pts} T.~Peters,  {\em A concordance invariant from the Floer homology of $\pm 1$ surgeries}, preprint 2010, arxiv:1003.3038.

\bibitem{Piontkowski}
J.~Piontkowski, \emph{On the number of cusps of rational cuspidal plane
  curves}, Experiment. Math. \textbf{16} (2007), no.~2, 251--255.

\bibitem{RA} J.~Ram\'i{}rez Alfons\'i{}n,  \textit{The Diophantine Frobenius problem}, Oxford Lecture Series in Mathematics and its Applications 30, Oxford University Press, Oxford, 2005.

\bibitem{Ras03} J.~Rasmussen, \emph{Floer homology and knot complements}, Harvard thesis, 2003, available at arxiv:math/0306378.

\bibitem{Rolfsen} D. Rolfsen, \emph{Knots and links, Publish or Perish}, 1976.

\bibitem{RS93}  J.~Robbin, D.~Salamon, \emph{The Maslov index for paths}, Topology \textbf{32} (4) (1993), 827--844.

\bibitem{Rud} L.~Rudolph, 
\emph{Quasipositivity as an obstruction to sliceness},
Bull. Amer. Math. Soc. (N.S.) \textbf{29} (1993), no. 1, 51--59. 

\bibitem{Sar} 
S.~Sarkar, \emph{Moving basepoints and the induced automorphisms of link Floer homology}, Algebr. Geom. Topol. \textbf{15} (2015), no. 5, 
2479--2515.                      

\bibitem{Scor} A.~Scorpan, \emph{The wild world of 4--manifolds}, American Mathematical Society, Providence, RI, 2005.

\bibitem{TonoR} K.~Tono, \emph{Rational unicuspidal plane curves with $\ol{\kappa}=1$}, Newton polyhedra and singularities (Kyoto, 2001). S\=urikaisekikenky\=usho K\=oky\=uroku No. 1233 (2001), 82--89.

\bibitem{Tono05} K.~Tono, \emph{On the number of cusps of cuspidal plane curves}, Math. Nachr. \textbf{278} (2005), 216--221.
\bibitem{Ts} S.~Tsunoda, \emph{The complements of projective plane curves}, RIMS-K\^oky\^uroku, \textbf{446} (1981), 48--56, available
at
\url{http://www.kurims.kyoto-u.ac.jp/~kyodo/kokyuroku/contents/pdf/0446-06.pdf}.

\bibitem{Tur97} V.~Turaev, \emph{Torsion invariants of $\Spin^c$--structures on 3--manifolds}, Math. Res. Lett., \textbf{4} (5) (1997), 679--695.
 
\bibitem{Tur1} V.~Turaev, \emph{Introduction to combinatorial torsions}, Notes taken by Felix Schlenk. 
Lectures in Mathematics ETH Z\"urich. Birkh\"auser Verlag, Basel, 2001.

\bibitem{Tur2} V.~Turaev, \emph{Torsions of 3--dimensional manifolds}, Progress in Mathematics, 208. Birkh\"auser Verlag, Basel, 2002.

\bibitem{Wak} I. Wakabayashi, \textit{On the logarithmic Kodaira dimension of the complement of a curve in} $\mathbb{P}^2$, Proc. Japan Acad. Ser. A. Math. Sci. \textbf{54} (1978), 157--162.

\bibitem{Wa} C.~Wall, \textit{Singular Points of Plane Curves}, London Mathematical Society Student Texts, 63. Cambridge University Press, Cambridge, 2004.

\bibitem{Za} O.~Zariski, \textit{On the topology of algebroid singularities}, Amer. J. Math.  \textbf{54} (1932),   453--465. 

\bibitem{Zarbook} O.~Zariski, \emph{Algebraic surfaces}, With appendices by S. Abhyankar, J. Lipman and D. Mumford. Preface to the appendices by Mumford. Reprint of the second (1971) edition. Classics in Mathematics. Springer-Verlag, Berlin, 1995.

\bibitem{Zem1} I.~Zemke, \emph{Quasi-stabilization and basepoint moving maps in link Floer homology}, preprint 2016, arxiv:1604.04316.


\bibitem{Zem2} I.~Zemke, \emph{A connected sum formula for involutive link Floer homology}, in preparation.

\bibitem{Zo} H.~Żołądek, \textit{The monodromy group}, Mathematical monographs (new series), \textbf{67}, Birkh\"auser Verlag, Basel, 2006.

\end{thebibliography}
\end{document}